\documentclass{eajam}
\renewcommand\thisnumber{x}
\renewcommand\thisyear {2024}
\renewcommand\thismonth{xxx}
\renewcommand\thisvolume{xx}
\usepackage{charter}
\usepackage[charter]{mathdesign}
\usepackage{color,xcolor}
\usepackage{soul}
\begin{document}

	\markboth{Q.~Yu et al}{DSM based on the Green's function for time-dependent inverse scattering problems}
	\title{A direct sampling method based on the Green's function for time-dependent inverse scattering problems}
	
	% single author:
	% \author[AUTHOR]{AUTHOR\corrauth}
	% \address{address of AUTHOR}
	% \email{{\tt email address of AUTHOR} (AUTHOR)}
	%\author[Only Author]{Only Author\corrauth}
	%\address{School of Mathematical Sciences, Beijing International University,
		%Beijing 12345, China.}
	%\email{{\tt email@address} (Only Author)}
	
	% multiple authors:
	\author{Qingqing Yu, Bo Chen \corrauth, Jiaru Wang and Yao Sun~*}
	
	\address{ College of Science, Civil Aviation University of China, Tianjin, 300300, China. }

	\emails{{\tt charliecb@163.com} (Bo Chen), {\tt sunyao10@mails.jlu.edu.cn} (Yao Sun)}

	%%%%% Begin Abstract %%%%%%%%%%%
	\begin{abstract}
		This paper concerns the numerical simulation of time domain inverse acoustic scattering problems with a point-like scatterer, multiple point-like scatterers or normal size scatterers. Based on the Green's function and the application of the time convolution, direct sampling methods are proposed to reconstruct the location of the scatterer. The proposed methods involve only integral calculus without solving any equations and are easy to implement. Numerical experiments are provided to show the effectiveness and robustness of the methods.
	\end{abstract}
	%%%%% end %%%%%%%%%%%
	
	%%%%% Keywords %%%%%%%%%%%
	
	%%%% AMS subject classifications %%%%
	
	%%%% maketitle %%%%%
	\maketitle
	
	%%%% Start %%%%%%
	
	\section{Introduction}
	The inverse scattering problems \cite{Springer-2,CCP,EABE} are widely used in geophysical exploration, sonar detection, non-destructive testing and medical imaging. Among the studies of the inverse acoustic scattering problems in the past decades, the research on the frequency domain problems accounts for the majority \cite{JCP-3,SIAM-4}. Nevertheless, since the dynamic scattered data is easier to obtain in practical applications, the time domain inverse scattering problems have attracted the attention of numerous researchers in recent years \cite{IP-5,Springer-1}.
	
	This paper is concerned with a simple direct sampling method to solve the inverse scattering problems in the time domain. Sampling methods are significant methods of solving inverse problems. Various sampling methods, such as the linear sampling method \cite{IP-1,IP-2,SIAM-3}, the direct sampling method \cite{IPI-2}, the reverse time migration method \cite{NMTMA1} and the factorization method \cite{IP-4}, have been studied in the past decades.
	The basic idea of the direct sampling methods is to construct indicator functions that have different performances inside and outside the area of the scatterer.
	Compared to the iterative methods, sampling methods do not require a priori information of the geometry and physical properties of the unknown target and are easy to implement.
	
	For the frequency domain inverse scattering problems, the linear sampling method was first proposed in 1996 \cite{IP-1}.
	A multilevel linear sampling method and its enhancement are proposed in \cite{SIAM-2} and \cite{JCP-1}, respectively.
	Based on the far field operator and the far field pattern of the scattered field, a direct sampling method to reconstruct the scatterers in the frequency domain is proposed in \cite{IP-6}.
	In \cite{ERA,IP-8}, simultaneous reconstructions of an obstacle and its excitation sources in frequency domain are considered.
	
	For the time domain inverse scattering problems, the linear sampling method was proved to be effective in 2010 \cite{IP-2}.
	In \cite{AA-1} and \cite{EAJAM}, the linear sampling method is used to solve the time domain inverse scattering problems in a locally perturbed half-plane and the time domain inverse scattering problems with cracks, respectively.
	The time domain linear sampling method is also considered to reconstruct an obstacle with Robin or Neumann boundary condition from near field measurements and reconstruct penetrable obstacles in \cite{AA-2} and \cite{IP-3}, respectively.
	Based on the migration method, effective sampling schemes to reconstruct small and extended scatterers from knowledge of time-dependent scattered data are proposed in \cite{JCP-2}.

	The reconstruction of point-like scatterers has been studied by image-based direct methods in \cite{IPI-1}. Note that the Green's function is significant for the analysis of inverse problems in mathematical physics \cite{IP-7,JSC,NMTMA2,AAM}. In \cite{IPSE-1}, the approximations of the solution to the forward scattering problem are analyzed utilizing the Green's function and the retarded single-layer potential, and an indicator function based on an approximate solution is defined to reconstruct the location of a single point-like scatterer. Unfortunately, the proposed method in \cite{IPSE-1} is not effective to reconstruct multiple point-like scatterers or normal size scatterers. In this paper, a novel direct sampling method is proposed to reconstruct the point-like scatterer based on the approximate solution to the forward scattering problem and the application of the time convolution. Moreover, modified methods are defined, which are feasible to reconstruct a single point-like scatterer, multiple point-like scatterers and normal size scatterers.
	
	The rest of the paper is organized as follows. A brief introduction of the forward and inverse scattering problems under consideration is given in Section 2. In Section 3, a novel direct sampling method is proposed and the theoretical analysis of the proposed method to reconstruct a point-like scatterer is provided. Moreover, two modified methods are proposed in this section. Numerical examples in both two and three dimensional spaces are provided to show the effectiveness and robustness of the proposed methods in Section 4. Finally, the paper is concluded in Section 5.

\section{Problem setting}

	Let $D\in\mathbb{R}^3$ be a bounded region where the scatterer locates and $c>0$ be the sound speed of the homogeneous background medium in $\mathbb{R}^3\backslash{\overline{D}}$. The Green's function of the D'Alembert operator $c^{-2}\partial_{tt}-\Delta$ is
	\begin{equation}
		G(x,t;y)=\frac{\delta(t-c^{-1}|x-y|)}{4\pi|x-y|}, \quad x,y\in\mathbb{R}^3,x\neq y,t\in\mathbb{R},
	\end{equation}
	where $\Delta$ is the Laplacian in $\mathbb{R}^3$ and $\delta(t)$ is the Dirac delta distribution.
	
	Let $\lambda(t)$ be a casual signal function, the casuality means that $\lambda(t)=0$ for $t<0$. Assume that $y$ represents the source point which satisfies $\{y\}\cap\overline{D}=\varnothing$, then the incident field $u^i$ due to $y$ is given by
	\begin{equation*}
		u^i(x,t;y)=G(x,t;y)\ast\lambda(t), \quad x\in{\mathbb{R}^3\setminus\left\{y\right\}},t\in\mathbb{R},
	\end{equation*}
	where
	\begin{equation}\label{timeconvolution}
		G(x,t;y)\ast\lambda(t)=\frac{\lambda{(t-c^{-1}|x-y|)}}{4\pi|x-y|}
	\end{equation}
	denotes the time convolution of the Green's function and the signal function. The time convolution (\ref{timeconvolution}) is an invaluable tool for the analysis of time domain scattering problems, which has been utilized to solve the forward scattering problem \cite{Springer-1}, reconstruct stationary point sources and a moving point source \cite{IP-7} and reconstruct a point-like scatterer \cite{IPSE-1}.

	For the analysis of the scattering problem, the total field $u^{tot}$ is usually divided into the incident field $u^i$ and the scattered field $u$, namely $u^{tot}=u^i+u$. The forward scattering problem for the impenetrable sound-soft obstacle $D$ is to find the scattered field $u$ which satisfies the homogeneous wave equation, Dirichlet boundary condition and homogeneous initial conditions
	
	\begin{align}
	\label{wave_eq}  c^{-2}\partial_{tt}{u}(x,t)-\Delta{u}(x,t)=0, & \quad x\in{\mathbb{R}^3}\setminus{\overline{D}},t\in\mathbb{R},\\
	\label{Dirichlet}u(x,t)=-u^i(x,t), & \quad x\in\partial{D},t\in\mathbb{R},\\
	\label{initial}  u(x,0)=0,  & \quad x\in{\mathbb{R}^3}\setminus{\overline{D}},\\
	\label{initiall} \partial_t{u}(x,0)=0, & \quad x\in{\mathbb{R}^3}\setminus{\overline{D}}.
	\end{align}

	Since $\lambda(t)$ is a casual function, the incident field $u^{i}$ and the scattered field $u$ vanish for $t<0$. Then the homogeneous initial conditions (\ref{initial}) and (\ref{initiall}) are direct conclusions of the causality. The forward scattering problem (\ref{wave_eq})-(\ref{initiall}) can be solved utilizing classical methods such as the retarded potential boundary integral equation method (refer to \cite{Springer-1,SIAM-5} for details).

	Assume that the region $B$ contains $D$ in it. The incident surface $\Gamma_i$ where the source points locate and the measurement surface $\Gamma_m$ where the measurement points locate are chosen as $\Gamma_i=\Gamma_m=\partial{B}$. The inverse scattering problem under consideration is: Determine the location and shape of the unknown target $D$ from a partial knowledge of the scattered waves
	\begin{equation*}
		\Big\{u(x,t;y):x\in\Gamma_m, t\in \mathbb{R}, y\in\Gamma_i\Big\}.
	\end{equation*}
	
	A special case of the scatterer is the point-like scatterer. For the time-harmonic scattering problems with a certain wavelength, a scatterer is called point-like if its diameter is far less than the wavelength. For the time-dependent scattering problem, however, the wavelength is unfixed and we need the definition of the center wavelength. In time domain, the signal function is often chosen as a Gaussian-modulated sinusoidal pulse
	\begin{equation*}
		\lambda(t)=\sin(\omega t)\mathrm{e}^{-\sigma(t-t_0)^2},
	\end{equation*}
	where $\omega>0$ is the center frequency, $\sigma>0$ is the frequency bandwidth parameter and $t_0$ is the time-shift parameter. Then the center wavelength is $2\pi c/\omega$ and a scatterer is called point-like if its diameter is far less than the center wavelength.
	
\section{The direct sampling methods}

In this section, several indicator functions are defined on the basis of the analysis of the Green's function and the scattering problem with a point-like scatterer. Direct sampling methods can be established based on the indicator functions.

\subsection{A direct sampling method}

In \cite{IPSE-1}, an approximate solution of the forward scattering problem with a point-like scatterer is provided based on the analysis of the Green's function.

\begin{lemma}\label{Lemma1}$($\cite{IPSE-1}~$)$
	Assume that $u(x,t;y)$ is the solution to the scattering problem (\ref{wave_eq})-(\ref{initiall}) with a point-like scatterer $D$, the incident point $y$ and the nontrivial signal function $\lambda(t)$. The incident surface $\Gamma_i$ and the measurement surface $\Gamma_m$ satisfy $\Gamma_i=\Gamma_m=\partial B$, where $\partial B$ is the boundary of the region $B$ which contains $D$ in it. The diameter of the scatterer $D$ satisfies $\textup{diam}(D)=\varepsilon(x)\min\limits_{x'\in \partial{D}}|x-x'|$  for arbitrary $x\in\Gamma_m$, in which $0<\varepsilon(x)\ll 1$. Denote by $y_0$ the geometric center of the point-like scatterer. Then we have
	\begin{equation*}
		u{(x,t;y)}=CU_{y_0}{(x,t;y)}+O(\varepsilon'(x;y)),\quad x\in\Gamma_m, t\in\mathbb{R}, y\in\Gamma_i,
	\end{equation*}
	where $C$ is a constant depending only on $x'$ and $D$, $\varepsilon'(x;y)=\max\{\varepsilon(x), \varepsilon(y)\}$ and $U_{y_0}$ is defined by
	\begin{equation*}
		U_{y_0}{(x,t;y)}=-\frac{\lambda(t-c^{-1}|x-y_0|-c^{-1}|y-y_0|)} {4\pi|x-y_0||y-y_0|}.
	\end{equation*}
\end{lemma}

Choose a sampling domain $\Omega$ which satisfies $\Omega\cap\Gamma_m=\varnothing$ and $D\subset \Omega$. Divide the sampling region into the sampling grid with the sampling points $z\in\Omega$ as the grid points. Let $u(x,t;y)$ be the solution to the forward scattering problem (\ref{wave_eq})-(\ref{initiall}) with the incident point $y$. Then, based on Lemma \ref{Lemma1}, define the indicator function
\begin{equation}\label{I1(z)}
	\resizebox{0.9\hsize}{!}{$
		I_1(z)=\frac{\left\|\int_{\Gamma_m}\left(u\ast U_{z}\right)(x,t;y)\mathrm{d}s_{x}\right\|_{L^2{(\mathbb{R} \times\Gamma_i)}}^2}
		{\left\|\int_{\Gamma_m}\left(u\ast u\right)(x,t;y)\mathrm{d}s_{x} \right\|_{L^2{(\mathbb{R}\times\Gamma_i)}} \cdot \left\|\int_{\Gamma_m}\left(U_{z}\ast U_{z}\right)(x,t;y)\mathrm{d}s_{x} \right\|_{L^2{(\mathbb{R}\times\Gamma_i)}}}, \quad z\in\Omega,
		$}
\end{equation}
where $\left(u\ast U_{z}\right)(x,t;y)$ denotes the time convolution of $u(x,t;y)$ and $U_{z}(x,t;y)$ and
\begin{equation*}
	\|\varphi(t;y)\|_{L^2{(\mathbb{R}\times\Gamma_i)}} := \left(\int_{\Gamma_i}\int_{\mathbb{R}}|\varphi(t;y)|^2 \mathrm{d}t\mathrm{d}s_{y}\right)^{1/2}.
\end{equation*}

For the reconstruction of a single point-like scatterer utilizing the indicator function $I_1(z)$, we expect that $I_1(z)$ reaches the maximum when the sampling point $z$ coincides with the position of the point-like scatterer. Another indicator function
\begin{equation}\label{I1'(z)}
		I_1^{'}(z)=\frac{\left\|\int_{\Gamma_m}\left(u\ast u_{z}\right)(x,t;y)\mathrm{d}s_{x}\right\|_{L^2{(\mathbb{R} \times\Gamma_i)}}^2}
		{\left\|\int_{\Gamma_m}\left(u\ast u\right)(x,t;y)\mathrm{d}s_{x} \right\|_{L^2{(\mathbb{R}\times\Gamma_i)}} \cdot \left\|\int_{\Gamma_m}\left(u_{z}\ast u_{z}\right)(x,t;y)\mathrm{d}s_{x} \right\|_{L^2{(\mathbb{R}\times\Gamma_i)}}}, \quad z\in\Omega
\end{equation}
\noindent
is needed for the analysis, where $u_{z}(x,t;y)$ is the solution of the scattering problem (\ref{wave_eq})-(\ref{initiall}) when the point-like scatterer is located at $D^{'}=\{x\in\Bbb R^3|x-z+y_0\in D\}$.

\begin{theorem}\label{Theorem1}
	Assume that $u(x,t;y)$ is the solution to the scattering problem (\ref{wave_eq})-(\ref{initiall}) with a point-like scatterer $D$, the incident point $y$ and the nontrivial signal function $\lambda(t)$. Denote by $y_0$ the geometric center of $D$. The incident surface $\Gamma_i$ and the measurement surface $\Gamma_m$ satisfy $\Gamma_i=\Gamma_m=\partial B$, where $\partial B$ is the boundary of the region $B$ which contains $D$ in it. The diameter of the scatterer $D$ satisfies $\textup{diam}(D)=\varepsilon(x)\min\limits_{x'\in \partial{D}}|x-x'|$  for arbitrary $x\in\Gamma_m$, in which $0<\varepsilon(x)\ll 1$. Denote by $\Omega$ the sampling region that contains $D$ in it and assume that the distance between $\Omega$ and $\Gamma_m$ satisfies $\textup{dist}(\Omega,\Gamma_m)>\textup{diam}(D)$. For a point $z\in\Omega$, let $u_{z}(x,t;y)$ be the solution to (\ref{wave_eq})-(\ref{initiall}) with a point-like scatterer located at $D^{'}=\{x\in\Bbb R^3|x-z+y_0\in D\}$. The indicator function $I_1^{'}(z)$ defined by (\ref{I1'(z)}) satisfies
	
	\begin{equation*}
		I_1^{'}(z)
		\left\{
		\begin{array}{ll}
			<1, \quad &z\neq y_0,\\
			=1, \quad &z=y_0.
		\end{array}
		\right.
	\end{equation*}
	The indicator function $I_1(z)$ defined by (\ref{I1(z)}) satisfies
	$$I_1(z)=I_1^{'}(z)+O(\varepsilon),$$
	where $0<\varepsilon\ll1$.
\end{theorem}

\begin{proof}
	Denote by $\mathcal{F}[\varphi](x,\omega;y)$ the fourier transform of $\varphi(x,t;y)$ on the time variable $t$. According to the convolution theorem, the Plencherel theorem and the Cauchy-Schwarz inequality, we have
	\begin{equation*}
		\begin{array}{cl}
			I_1^{'}(z)&=\frac{\left\|\mathcal{F}\left[\int_{\Gamma_m}\left(u\ast u_{z}\right)\mathrm{d}s_{x}\right](\omega;y)\right\|_{{L^2} \left(\mathbb{R}\times\Gamma_i\right)}^2}
			{\left\|\mathcal{F}\left[\int_{\Gamma_m}\left(u\ast u\right)\mathrm{d}s_{x}\right](\omega;y)\right\|_{L^2{\left( \mathbb{R}\times\Gamma_i\right)}}\cdot
				\left\|\mathcal{F}\left[\int_{\Gamma_m}\left(u_{z}\ast u_{z}\right)\mathrm{d}s_{x}\right](\omega;y)\right\|_{L^2{\left( \mathbb{R}\times\Gamma_i\right)}}}\\
			&=\frac{\left\|\int_{\Gamma_m}\mathcal{F}\left[u\right] (x,\omega;y)\cdot \mathcal{F}\left[u_{z}\right](x,\omega;y)\mathrm{d}s_{x} \right\|_{{L^2} \left(\mathbb{R}\times\Gamma_i\right)}^2}
			{\left\|\int_{\Gamma_m}\left(\mathcal{F}\left[u\right] (x,\omega;y)\right)^2\mathrm{d}s_{x} \right\|_{L^2{\left( \mathbb{R}\times\Gamma_i\right)}}\cdot
				\left\|\int_{\Gamma_m}\left(\mathcal{F}\left[u_{z}\right] (x,\omega;y)\right)^2\mathrm{d}s_{x} \right\|_{L^2{\left(\mathbb{R}\times\Gamma_i\right)}}}\\
			&\leq\frac{\left\|\left(\int_{\Gamma_m}\left(\mathcal{F}\left[u\right] (x,\omega;y)\right)^2\mathrm{d}s_{x} \right)^{1/2} \cdot \left(\int_{\Gamma_m}\left(\mathcal{F}\left[u_z\right] (x,\omega;y)\right)^2\mathrm{d}s_{x} \right)^{1/2} \right\|_{{L^2} \left(\mathbb{R}\times\Gamma_i\right)}^2}
			{~~~\left\|\int_{\Gamma_m}\left(\mathcal{F}\left[u\right] (x,\omega;y)\right)^2\mathrm{d}s_{x} \right\|_{L^2{\left( \mathbb{R}\times\Gamma_i\right)}}\cdot
				\left\|\int_{\Gamma_m}\left(\mathcal{F}\left[u_z\right] (x,\omega;y)\right)^2\mathrm{d}s_{x} \right\|_{L^2{\left(\mathbb{R}\times\Gamma_i\right)}}~~~}\\
			&\leq 1.
		\end{array}
	\end{equation*}
	Noticing that $u_{y_0}(x,t;y)=u(x,t;y)$, we have
	\begin{equation*}	
			I_1^{'}(y_0)=\frac{\left\|\int_{\Gamma_m}\left(u\ast u_{y_0}\right)(x,t;y)\mathrm{d}s_{x}\right\|_{L^2{(\mathbb{R} \times\Gamma_i)}}^2}
			{\left\|\int_{\Gamma_m}\left(u\ast u\right)(x,t;y)\mathrm{d}s_{x} \right\|_{L^2{(\mathbb{R}\times\Gamma_i)}} \cdot \left\|\int_{\Gamma_m}\left(u_{y_0}\ast u_{y_0}\right)(x,t;y)\mathrm{d}s_{x} \right\|_{L^2{(\mathbb{R}\times\Gamma_i)}}}=1.	
	\end{equation*}
	Moreover, combining with the uniqueness of the solution to the scattering problem (\cite{IP-2}, Theorem 4), we have
	\begin{equation*}
		I_1^{'}(z)
		\left\{
		\begin{array}{ll}
			<1, \quad &z\neq y_0,\\
			=1, \quad &z=y_0.
		\end{array}
		\right.
	\end{equation*}
	
	Note that $\textup{diam}(D^{'})=\textup{diam}(D)$ and Lemma \ref{Lemma1} implies
	\begin{equation*}
		u_{z}(x,t;y)=CU_{z}{(x,t;y)}+O(\varepsilon'(x;y)).
	\end{equation*}
	Define $\varepsilon=\sup\limits_{x\in\Gamma_m,y\in\Gamma_i}\varepsilon'(x;y)$. Then we have $0<\varepsilon\ll1$ and $u_z=CU_{z}+O(\varepsilon)$. Moreover, a direct computation implies
	
		\begin{equation*}
			\begin{array}{cl}
				I_1^{'}(z)
				&=\frac{\left\|\int_{\Gamma_m}\left(u\ast u_{z}\right)(x,t;y)\mathrm{d}s_{x}\right\|_{L^2{(\mathbb{R} \times\Gamma_i)}}^2}
				{\left\|\int_{\Gamma_m}\left(u\ast u\right)(x,t;y)\mathrm{d}s_{x} \right\|_{L^2{(\mathbb{R}\times\Gamma_i)}} \cdot \left\|\int_{\Gamma_m}\left(u_{z}\ast u_{z}\right)(x,t;y)\mathrm{d}s_{x} \right\|_{L^2{(\mathbb{R}\times\Gamma_i)}}}\\
				
				&=\frac{\left\|\int_{\Gamma_m}\left(u\ast (CU_{z}+O(\varepsilon))\right)(x,t;y)\mathrm{d}s_{x}\right\|_{L^2{(\mathbb{R} \times\Gamma_i)}}^2}
				{\left\|\int_{\Gamma_m}\left(u\ast u\right)(x,t;y)\mathrm{d}s_{x} \right\|_{L^2{(\mathbb{R}\times\Gamma_i)}} \cdot \left\|\int_{\Gamma_m}\left((CU_{z}+O(\varepsilon))\ast (CU_{z}+O(\varepsilon))\right)(x,t;y)\mathrm{d}s_{x} \right\|_{L^2{(\mathbb{R}\times\Gamma_i)}}}\\
				
				&=\frac{C^2\left\|\int_{\Gamma_m}\left(u \ast U_{z}\right)(x,t;y)\mathrm{d}s_{x}\right\|_{L^2{(\mathbb{R} \times\Gamma_i)}}^2+O(\varepsilon)}
				{\left\|\int_{\Gamma_m}(u\ast u)(x,t;y)\mathrm{d}s_{x} \right\|_{L^2{(\mathbb{R}\times\Gamma_i)}} \cdot \left(C^2\left\|\int_{\Gamma_m}\left(U_{z} \ast U_{z} \right)(x,t;y)\mathrm{d}s_{x} \right\|_{L^2{(\mathbb{R}\times\Gamma_i)}} +O(\varepsilon)\right)}\\
				
				&=\frac{\left\|\int_{\Gamma_m}\left(u \ast U_{z}\right)(x,t;y)\mathrm{d}s_{x}\right\|_{L^2{(\mathbb{R} \times\Gamma_i)}}^2}
				{\left\|\int_{\Gamma_m}(u\ast u)(x,t;y)\mathrm{d}s_{x} \right\|_{L^2{(\mathbb{R}\times\Gamma_i)}} \cdot \left(\left\|\int_{\Gamma_m}\left(U_{z} \ast U_{z} \right)(x,t;y)\mathrm{d}s_{x} \right\|_{L^2{(\mathbb{R}\times\Gamma_i)}}\right)} +O(\varepsilon)\\
				
				&=I_1(z)+O(\varepsilon).
			\end{array}
		\end{equation*}
	
	This completes the proof.
	\qed
\end{proof}

According to Theorem \ref{Theorem1}, we can derive that the indicator function $I_1^{'}(z)$ reaches its global maximum if and only if $z=y_0$, and $I_1(z)$ is an approximation of $I_1^{'}(z)$. Then we can say that, approximately, $I_1(z)$ reaches its global maximum if and only if $z=y_0$. Thus, we can get the approximate location of the point-like scatterer by drawing the image of $I_1(z)$. The algorithm can be summarized as Algorithm 1. The feasibility of the algorithm to reconstruct a point-like scatterer is verified by the numerical experiments in the next section.

\begin{table}\label{algorithm1}
	\centering
	\begin{tabular}{ll}
		\hline
		\multicolumn{2}{l}{\textbf{Algorithm 1:} The reconstruction scheme using the indicator function $I_1(z)$.}\\
		\hline
		Step1 & Choose the scatterer $D$, the signal function $\lambda(t)$, the incident surface $\Gamma_i$  \\
		& and the measurement surface $\Gamma_m=\Gamma_i$. Select the sensing points  \\
		& $x_i\in\Gamma_m~( i=1,\cdots,N_m)$, the equidistant time nodes $t_k\in[0,T]~(k=0,\cdots,N_t)$   \\
		& and the source points $y_j\in\Gamma_i~( j=1,\cdots,N_i)$. Collect the scattered data on $\Gamma_m$:\\
		& \quad\quad\quad\quad $u(x_i,t_k;y_j),\quad i=1,\cdots,N_m$, $k=0,\cdots,N_t$, $j=1,\cdots,N_i$. \\
		Step2 & Choose a cubic sampling domain $\Omega$ such that $D\subset\Omega$, $\overline{\Omega}\cap\Gamma_m=\varnothing$. Divide $\Omega$ \\
		& into equidistant sampling grid, for any sampling point $z\in\Omega$, compute\\
		&\quad\quad $\displaystyle{ N_{u,U_z}(z)=\left({\sum\limits_{j=1}\limits^{N_i}} {\sum\limits_{k=0}\limits^{N_t}} \left|{\sum\limits_{i=1}\limits^{N_m}} \sum\limits_{l=0}\limits^{k}u(x_i,t_{k-l};y_j) U_z(x_i,t_l;y_j)\Delta t
			\Delta s_{x_i}\right|^2\Delta t \Delta s_{y_j}\right)^{1/2}}$,\\
		&\quad\quad $\displaystyle{ N_{u,u}(z)=\left({\sum\limits_{j=1}\limits^{N_i}} {\sum\limits_{k=0}\limits^{N_t}} \left|{\sum\limits_{i=1}\limits^{N_m}} \sum\limits_{l=0}\limits^{k}u(x_i,t_{k-l};y_j) u(x_i,t_l;y_j)\Delta t \Delta s_{x_i}\right|^2\Delta t \Delta s_{y_j}\right)^{1/2}}$,\\
		&\quad\quad $\displaystyle{ N_{U_z,U_z}(z)=\left({\sum\limits_{j=1}\limits^{N_i}} {\sum\limits_{k=0}\limits^{N_t}} \left|{\sum\limits_{i=1}\limits^{N_m}} \sum\limits_{l=0}\limits^{k}U_z(x_i,t_{k-l};y_j) U_z(x_i,t_l;y_j)\Delta t
			\Delta s_{x_i}\right|^2\Delta t \Delta s_{y_j}\right)^{1/2}}$,\\
		&and\\
		&\quad \quad $\displaystyle{ I_1(z)=\frac{N_{u,U_z}^2(z)}
			{N_{u,u}(z)\cdot N_{U_z,U_z}(z)},}$\\
		&where $\Delta t$ is the length of the time step, $\Delta s_{x_i}$ and $\Delta s_{y_j}$ are the areas of the  \\
		&grid cells of the sensor $x_i$ and the source $y_j$, respectively.\\
		Step3 & Mesh $I_1(z)$ on the sampling grid. The location of the point-like scatterer is  \\
		& determined by the global maximum point of $I_1(z)$.\\
		\hline
	\end{tabular}
\end{table}

\begin{remark}\label{remark1}
	A time shift parameter $\mu$ appears in the numerical computation of the time domain inverse scattering problems $($see, e.g., \cite{AA-1,IP-3,IPSE-1}~$)$. The time shift parameter has a significant impact for the reconstruction of the scatterer. However, the theoretical analysis of the time shift is scarce and the time-shift parameter is usually chosen by trial and error in the numerical experiments.
	
	In the definition of the indicator function (\ref{I1(z)}), we choose the time convolution of $u$ and $U_z$ instead of the product as that in \cite{IPSE-1}. An important property of the convolution is
	\begin{equation*}
		(\tau_{\mu}f)\ast g=f\ast(\tau_{\mu}g)=\tau_{\mu}(f\ast g),
	\end{equation*}
	where $\tau_{\mu}$ is the time shift operator defined by $(\tau_{\mu}f)(t)=f(t-\mu)$. On this basis, we have
	\begin{equation*}
		\int_{\mathbb{R}}\left|\left((\tau_{\mu}f)\ast g\right)(t)\right|^2\mathrm{d}t=\int_{\mathbb{R}}\left|\left(\tau_{\mu}(f\ast g)\right)(t)\right|^2\mathrm{d}t=\int_{\mathbb{R}}\left|\left(f\ast g\right)(t)\right|^2\mathrm{d}t.
	\end{equation*}
	The property is found interesting since the time shift is no longer an important factor when we utilize the time convolution in the indicator function.
\end{remark}

\subsection{The modified indicator functions}

For $x\in\Gamma_m, t\in\mathbb{R}$ and $z\in\Omega$, compare the functions
\begin{equation*}
	U_z{(x,t;y)}=-\frac{\lambda(t-c^{-1}|x-z|-c^{-1}|y-z|)}{4\pi|x-z||y-z|}
\end{equation*}
and
\begin{equation*}
	(G\ast{\lambda})(x,t;z)=\frac{\lambda(t-c^{-1}|x-z|)}{4\pi|x-z|}.
\end{equation*}
Note that the difference between the signal functions $\lambda(t-c^{-1}|x-z|-c^{-1}|y-z|)$ and $\lambda(t-c^{-1}|x-z|)$ is the time shift $c^{-1}|y-z|$. Define
\begin{equation*}
	\begin{array}{rl}
		f_{1,z}{(x,t;y)}&:=u(x,t;y)\ast\lambda(t-c^{-1}|x-z|),\\
		f_{2,z}{(x,t;y)}&:=u(x,t;y)\ast\lambda(t-c^{-1}|x-z|-c^{-1}|y-z|).
	\end{array}
\end{equation*}
According to Remark \ref{remark1}, we have
$$f_{2,z}{(x,t;y)}=f_{1,z}{(x,t-c^{-1}|y-z|;y)}.$$
Then a direct computation implies
\begin{equation*}
	\begin{array}{rl}
		\left\|\int_{\Gamma_m}f_{2,z}{(x,t;y)}\mathrm{d}s_{x} \right\|_{L^2{(\mathbb{R} \times\Gamma_i)}} &= \left\|\int_{\Gamma_m}f_{1,z}{(x,t-c^{-1}|y-z|;y)}\mathrm{d}s_{x} \right\|_{L^2{(\mathbb{R} \times\Gamma_i)}}\\
		&=\left\|\int_{\Gamma_m}f_{1,z}{(x,t;y)}\mathrm{d}s_{x} \right\|_{L^2{(\mathbb{R} \times\Gamma_i)}},
	\end{array}
\end{equation*}
which means $U_z{(x,t;y)}$ and $(G\ast{\lambda})(x,t;z)$ play similar roles while calculating time convolutions.

For convenience of expression, denote
\begin{equation*}
	G_{z}(x,t)=(G\ast{\lambda})(x,t;z).
\end{equation*}
Then we define a new indicator function
\begin{equation}\label{I2(z)}
	\resizebox{0.8\hsize}{!}{$
		I_2(z)=\frac{\left\|\int_{\Gamma_m}\left(u\ast G_{z}\right) (x,t;y)\mathrm{d}s_{x}\right\|_{L^2{(\mathbb{R} \times\Gamma_i)}}^2}
		{\left\|\int_{\Gamma_m}\left(u\ast u\right)(x,t;y)\mathrm{d}s_{x} \right\|_{L^2{(\mathbb{R}\times\Gamma_i)}} \cdot \left\|\int_{\Gamma_m}\left(U_{z}\ast U_{z}\right)(x,t;y)\mathrm{d}s_{x} \right\|_{L^2{(\mathbb{R}\times\Gamma_i)}}}, \quad z\in\Omega.
		$}
\end{equation}

Note that only the molecular is slightly modified in $I_2(z)$ compared to $I_1(z)$. Therefore, Algorithm 2 based on the indicator function $I_2(z)$ is similar to Algorithm 1 except that we use
$$I_2(z)=\frac{N_{u,G_z}^2(z)}{N_{u,u}(z)\cdot N_{U_z,U_z}(z)}$$
instead of $I_1(z)$ as the indicator, in which $N_{u,u}(z)$ and $N_{U_z,U_z}(z)$ are defined the same as that in Algorithm 1 and
$$N_{u,G_z}(z)=\left({\sum\limits_{j=1}\limits^{N_i}} {\sum\limits_{k=0}\limits^{N_t}} \left|{\sum\limits_{i=1}\limits^{N_m}} \sum\limits_{l=0}\limits^{k}u(x_i,t_{k-l};y_j) G_z(x_i,t_l) \Delta t \Delta s_{x_i} \right|^2\Delta t \Delta s_{y_j} \right)^{1/2}.$$

Furthermore, notice that $\left\|\int_{\Gamma_m}\left(u\ast u\right)(x,t;y)\mathrm{d}s_{x} \right\|_{L^2{(\mathbb{R}\times\Gamma_i)}}$ is independent of $z$ and has no influence to the effect of the algorithm. Based on this thought, we go a step further and define a new modified indicator function
\begin{equation}\label{I3(z)}
	I_3(z)=\left\|\int_{\Gamma_m}\left(u\ast G_{z}\right) (x,t;y)\mathrm{d}s_{x}\right\|_{L^2{(\mathbb{R} \times\Gamma_i)}}, \quad z\in\Omega.
\end{equation}
Once again, Algorithm 3 based on the indicator function $I_3(z)$ is similar to Algorithm 1 except that we use
$$I_3(z)=N_{u,G_z}(z)$$
as the indicator.

Note that Algorithm 3 is feasible to reconstruct a single point-like scatterer, multiple point-like scatterers, a normal size scatterer or some of the more complicated cases. The effectiveness and robustness of the algorithm can be seen in Section 4.

\section{Numerical experiments}

This section demonstrates the effectiveness of the proposed algorithms with numerical experiments in both two and three dimensional spaces.

In $\mathbb{R}^2$, the Green's function of the D'Alembert operator $c^{-2}\partial_{tt}-\Delta$ is
\begin{equation*}
	G_2(x,t;y)=\frac{H(t-c^{-1}|x-y|)}{2\pi \sqrt{t^2-c^{-2}|x-y|^2}},\quad
	x,y\in\mathbb{R}^2,x\neq y,t\in\mathbb{R},
\end{equation*}
where $H(t)$ is the Heaviside function
\begin{equation*}
	H(t)=
	\left\{
	\begin{array}{ll}
		1, \quad t>0,\\
		0, \quad t\leq0.
	\end{array}
	\right.
\end{equation*}
The time convolution of the Green's function and the signal function is
\begin{equation*}
		G_2(x,t;y)\ast\lambda(t)=\int_{-\infty}^{t-c^{-1}|x-y|} \frac{\lambda{(\tau)}}{2\pi \sqrt{(t-\tau)^2-c^{-2}|x-y|^2}}\mathrm{d}\tau, \quad x,y\in\mathbb{R}^2,x\neq y,t\in\mathbb{R}.
\end{equation*}

In the experiments, denote by $N=N_i=N_m$ the number of the incident/measurement points. The time interval $[0,T]$ is equally divided into $N_{t}$ time steps. Choose the sound speed in the homogeneous background medium as $c=1$. The random noise is added to the scattered data by
$$u_{\varepsilon}=u(1+\varepsilon r),$$
where $\varepsilon$ denotes the noise level and $r$ are uniformly distributed random numbers ranging from $-1$ to $1$. The signal function is chosen to be a sinusoidal pulse
\begin{equation*}
	\lambda(t)=\sin(4 t)\mathrm{e}^{-1.6(t-3)^2}.
\end{equation*}

For the experiments in $\mathbb{R}^2$, unless otherwise specified, we choose $T=25$ for Algorithm 1 and Algorithm 2, choose $T=15$ for Algorithm 3, and choose $N_t=128$ for all the three algorithms. $N=20$ is chosen for the reconstructions with full aperture data. The measurement points are evenly distributed on the circle of radius 4 with the center at the origin, which means the measurement points are
$$x_i=4\left(\cos{\frac{i\pi}{10}},\sin{\frac{i\pi}{10}}\right),\quad i=0,1,\cdots,19.$$
The sampling points are chosen as $21\times21$ uniform discrete points in the sampling region $\Omega=[-2.6,2.6]\times[-2.6,2.6]$.

\begin{example}\textbf{Reconstructions of a point-like scatterer}\label{example1}

In this example,  we consider the reconstructions of a point-like scatterer utilizing the Algorithms 1-3. The boundary of the scatterer is parameterized as
\begin{equation}\label{point}
	s(\theta) = (a,b)+0.001(\cos\theta,\sin\theta),\quad \theta\in[0,2\pi).
\end{equation}
The reconstructions of point-like scatterers centered at $(a,b)=(0,0)$ with different noise are shown in Figure \ref{fig-i1i2i3-point}. The measurement points are marked with black asterisks and the green asterisks represent the exact locations of the point-like scatterers. Figure \ref{fig-i1i2i3-point} shows that the Algorithms 1-3 are feasible and robust to reconstruct a point-like scatterer.

%% Figure1 - a single point-like scatterer - I_1-I_2-I_3-0.05-0.2
\begin{figure}
	\centering
	\begin{tabular}{ccc}
		\includegraphics[width=0.33\textwidth]{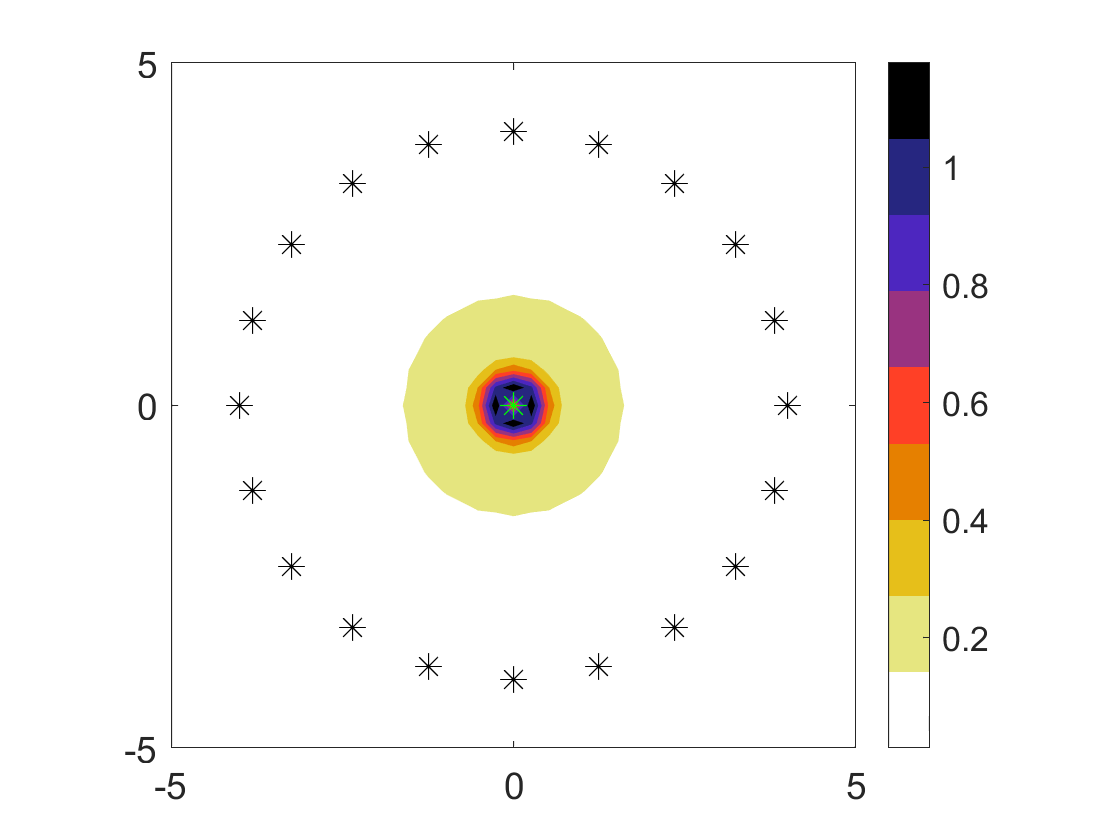}
		& \includegraphics[width=0.33\textwidth]{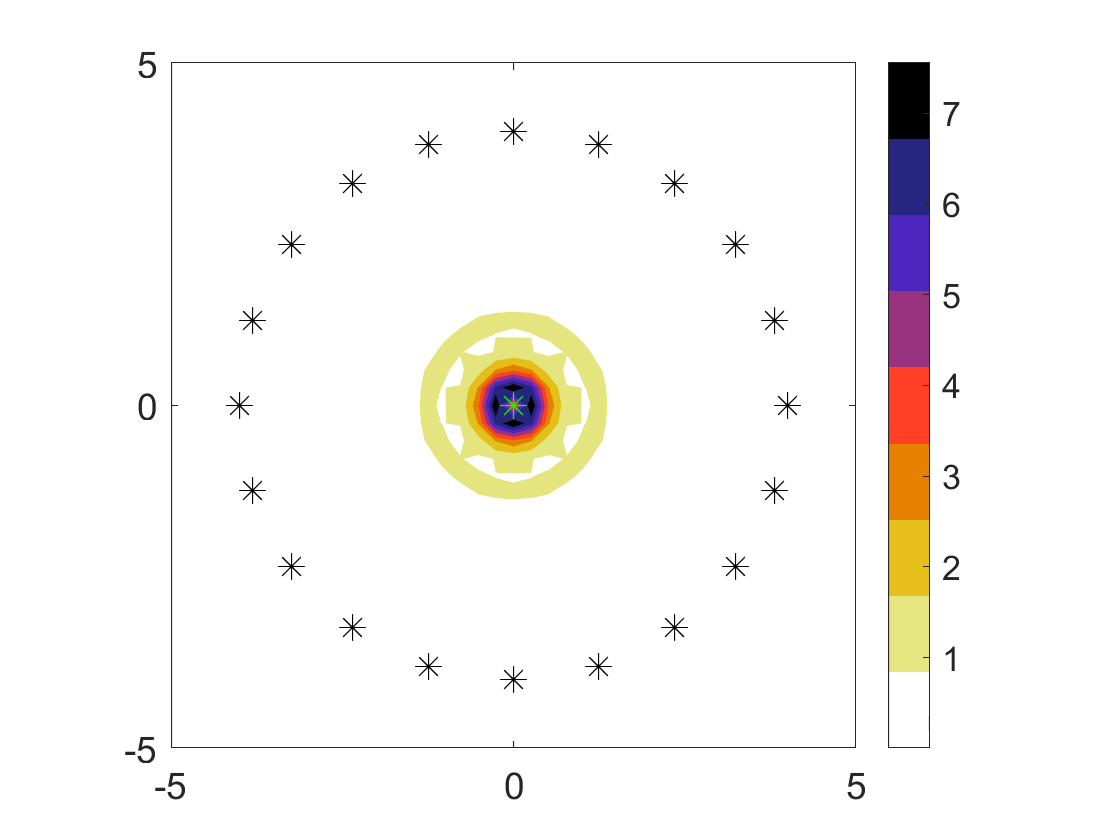}
		& \includegraphics[width=0.33\textwidth]{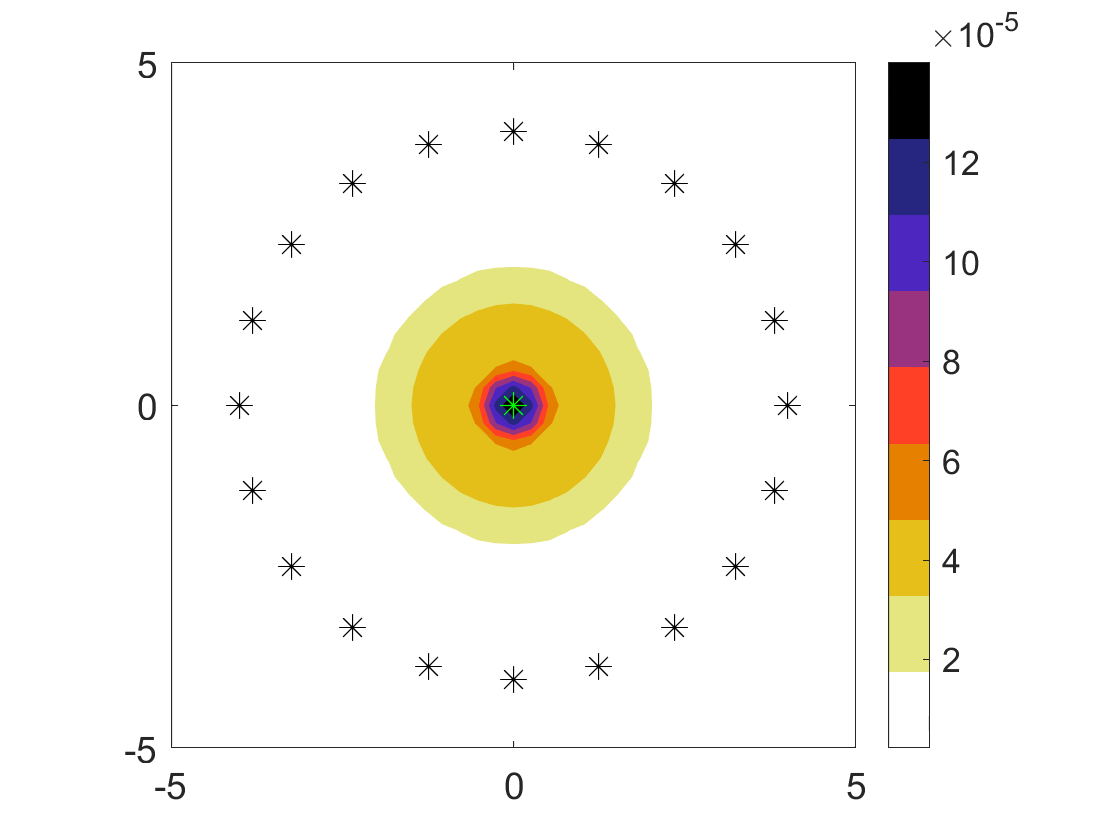} \\
		(a)~$I_1,~\varepsilon=5\%$ & (b)~$I_2,~\varepsilon=5\%$ & (c)~$I_3,~\varepsilon=5\%$ \\
		\includegraphics[width=0.33\textwidth]{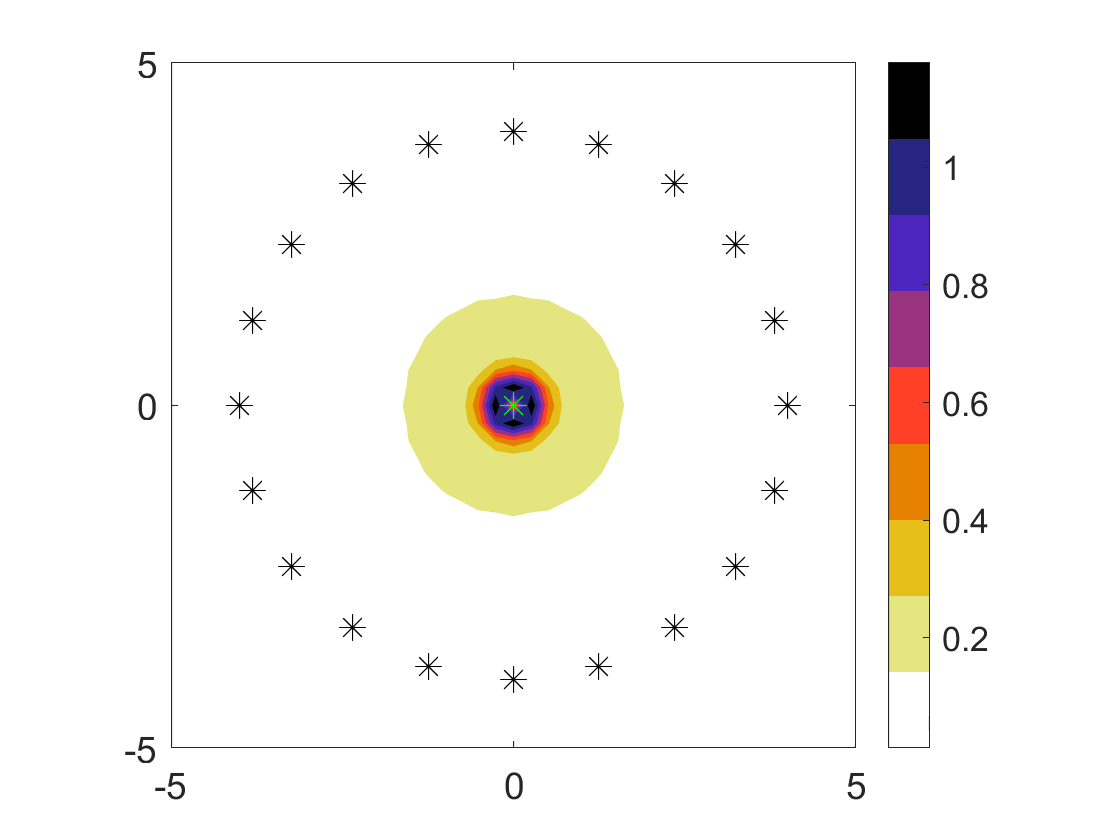}
		& \includegraphics[width=0.33\textwidth]{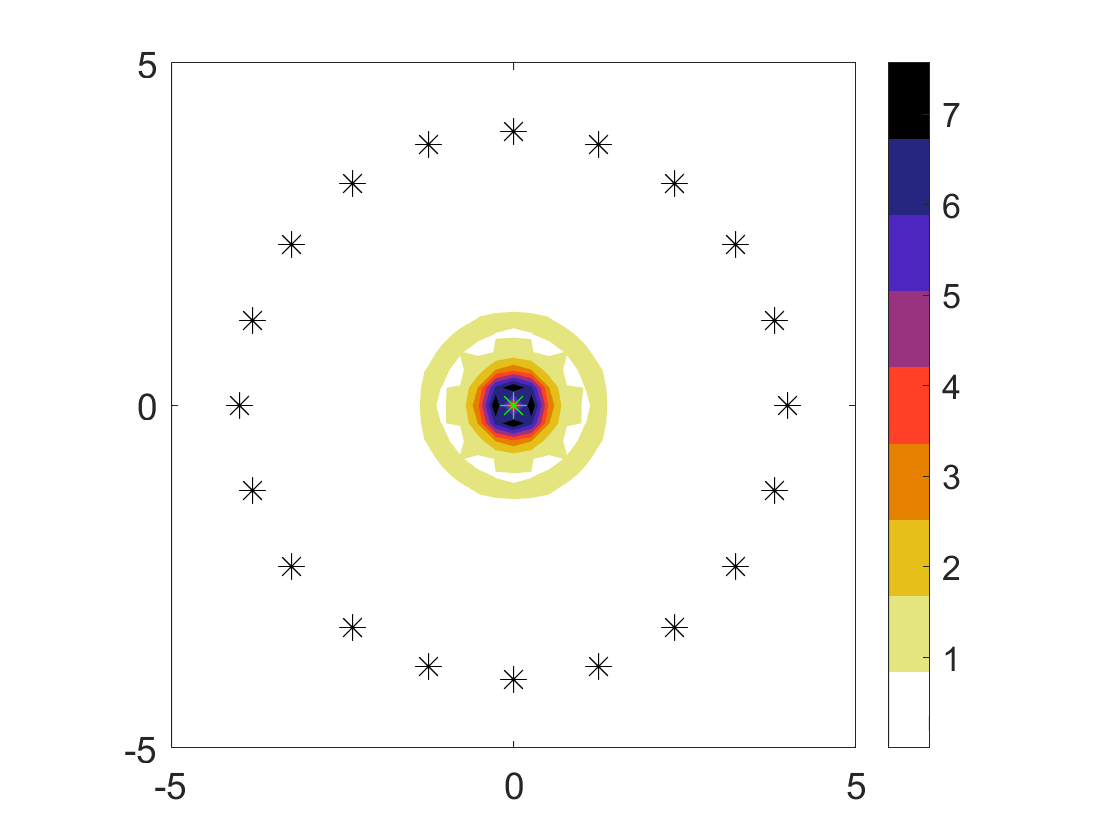}
		& \includegraphics[width=0.33\textwidth]{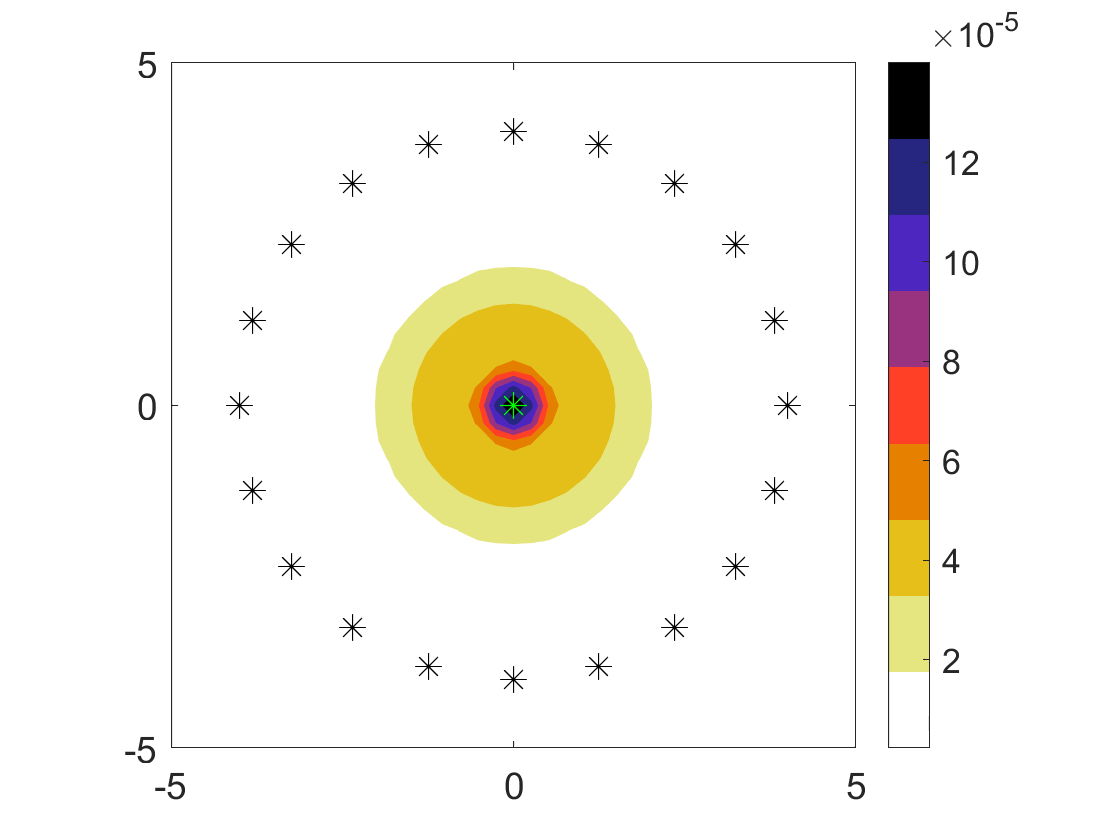} \\
		(d)~$I_1,~\varepsilon=20\%$ & (e)~$I_2,~\varepsilon=20\%$ & (f)~$I_3,~\varepsilon=20\%$
	\end{tabular}
	\caption{Reconstructions of a single point-like scatterer centered at $(0,0)$ with different indicators and noise levels. }\label{fig-i1i2i3-point}
\end{figure}
\end{example}

\begin{example}\textbf{Reconstructions of multiple point-like scatterers}\label{example2}

The simultaneous reconstructions of multiple point-like scatterers are considered in this example. The boundaries of the scatterers are also parameterized as (\ref{point}) with different centers $(a,b)$. Figure \ref{fig-points} shows the reconstructions of different numbers of point-like scatterers with the noise level $\varepsilon=5\%$. The exact locations of the scatterers are marked with green asterisks.

%% Figure2 - multiple point-like scatterers - I1-I2-I3-5%
\begin{figure}
	\centering
	\begin{tabular}{ccc}
		\includegraphics[width=0.33\textwidth]{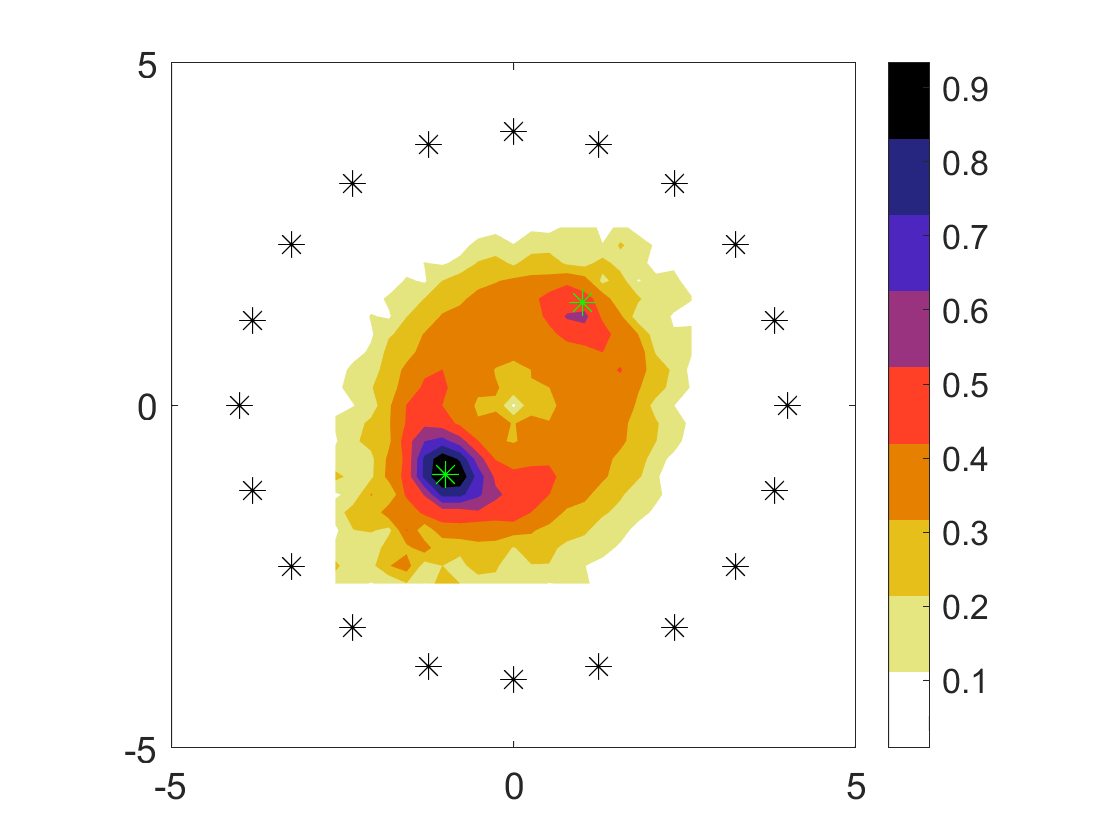}
		& \includegraphics[width=0.33\textwidth]{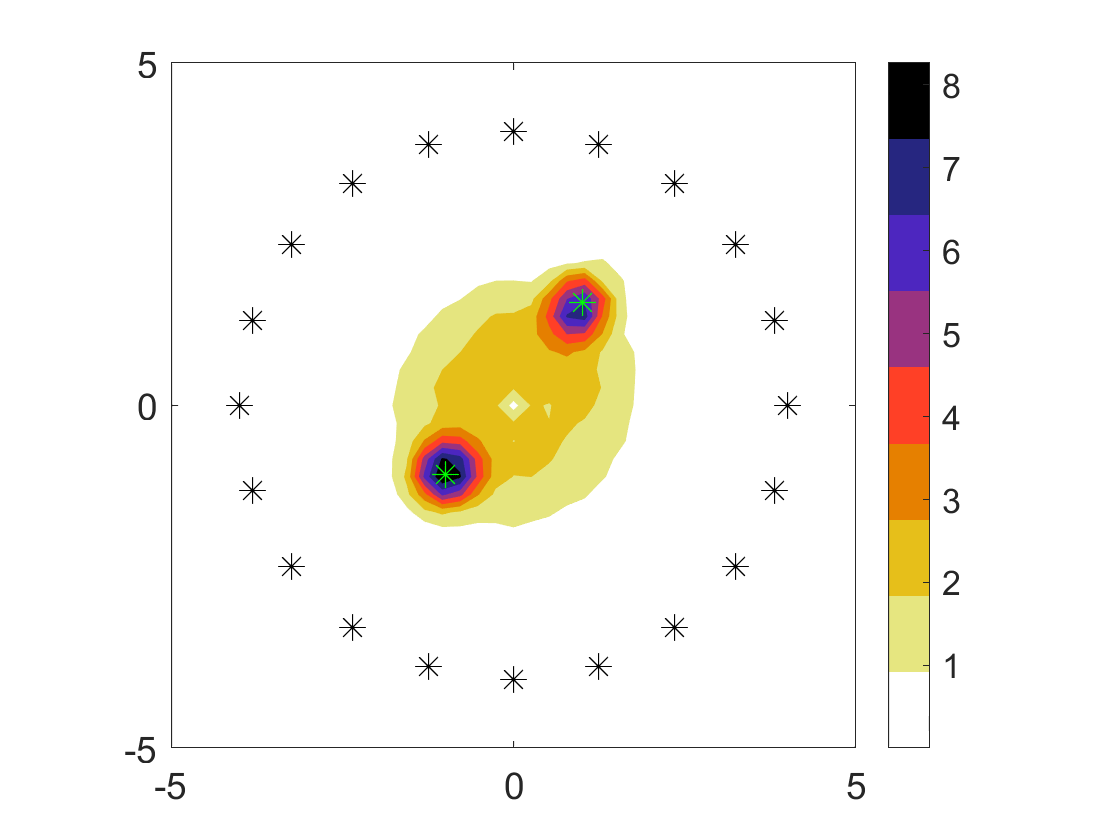}
		& \includegraphics[width=0.33\textwidth]{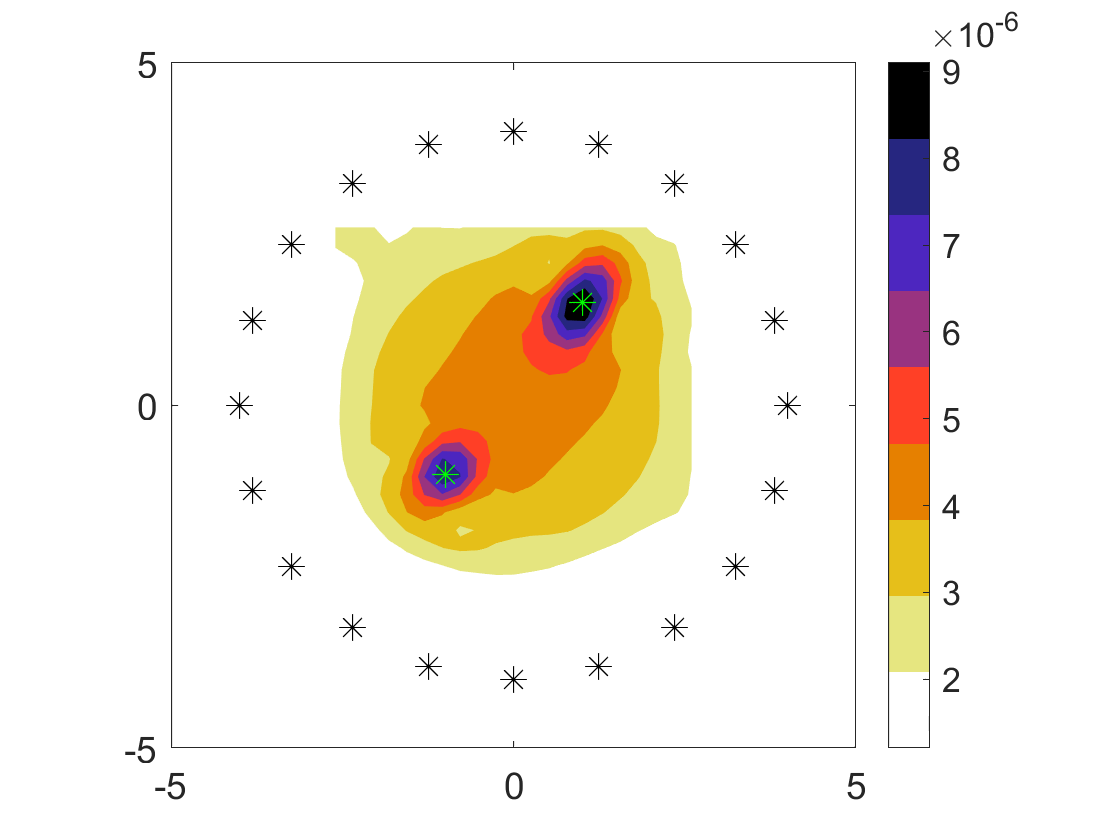} \\
		(a)~$I_1$ & (b)~$I_2$ & (c)~$I_3$ \\
		\includegraphics[width=0.33\textwidth]{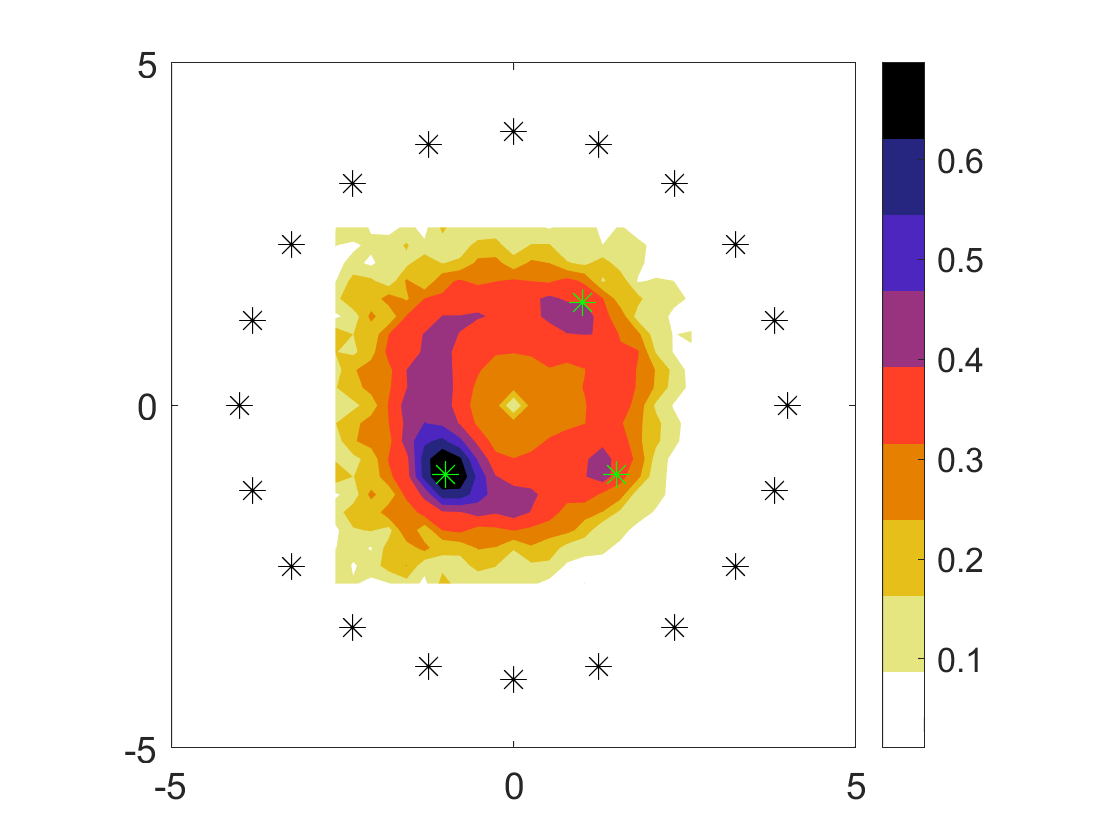}
		& \includegraphics[width=0.33\textwidth]{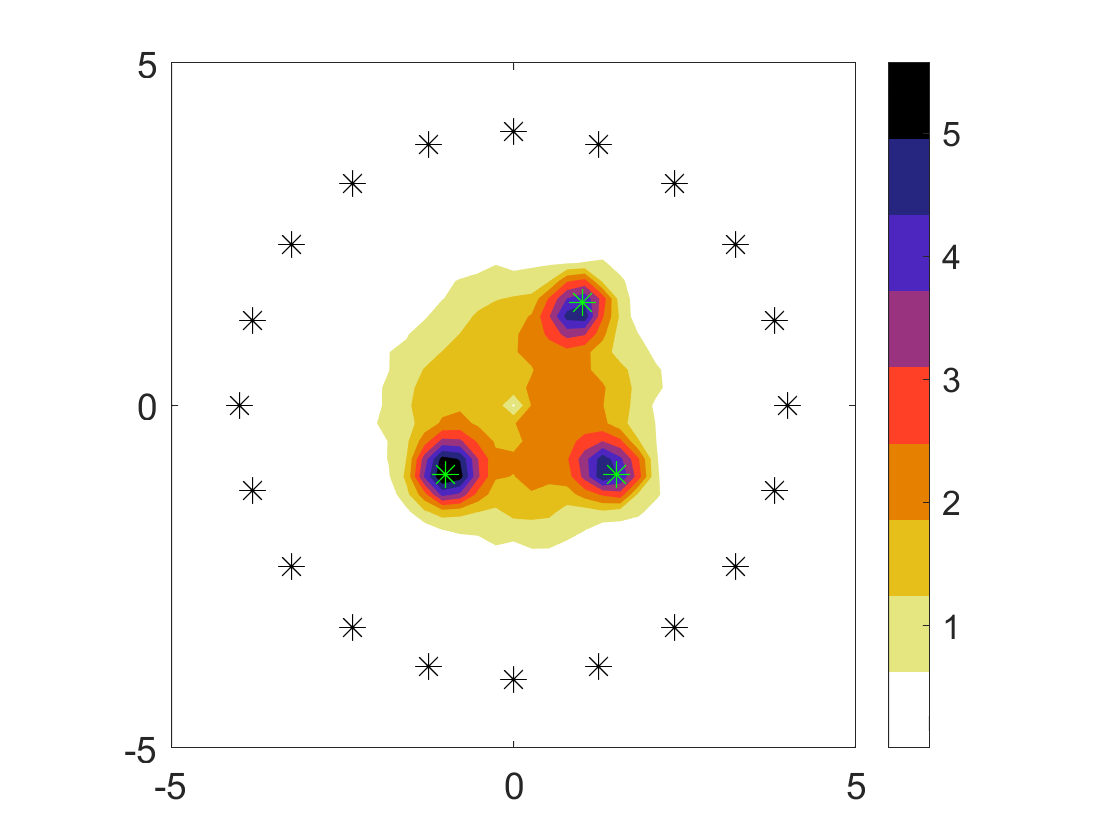}
		& \includegraphics[width=0.33\textwidth]{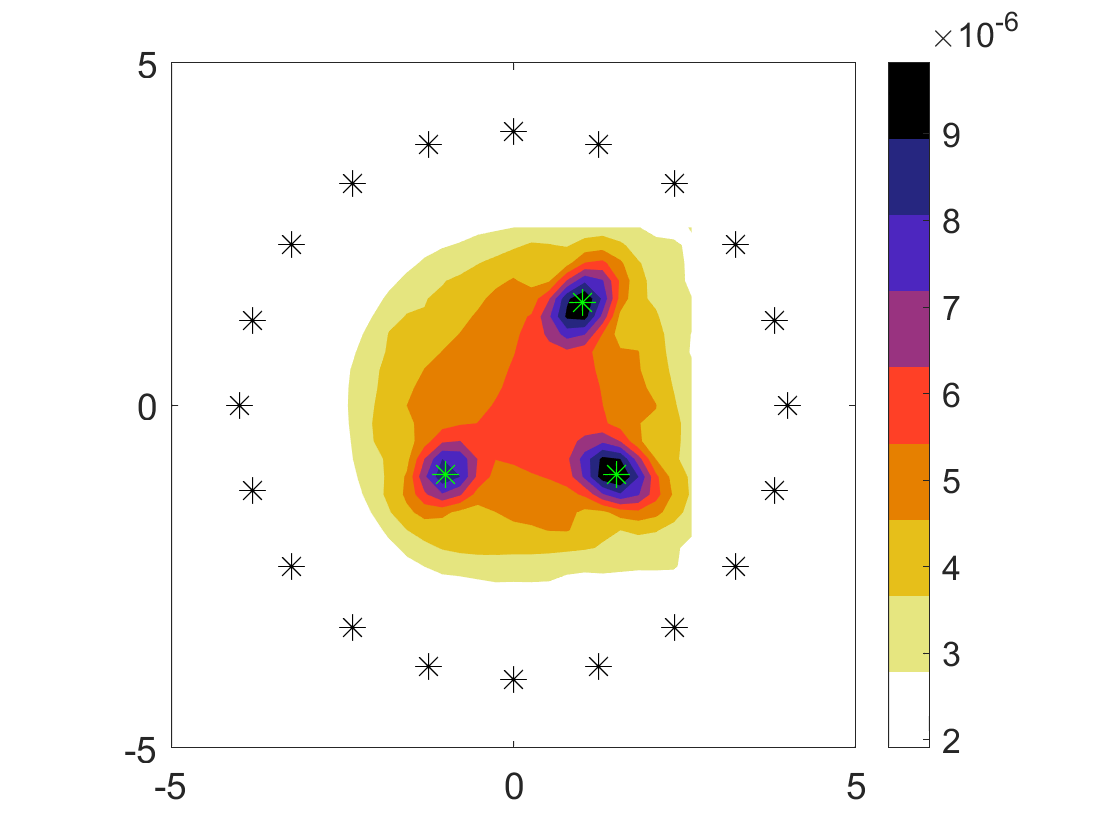} \\
		(d)~$I_1$ & (e)~$I_2$ & (f)~$I_3$ \\
		\includegraphics[width=0.33\textwidth]{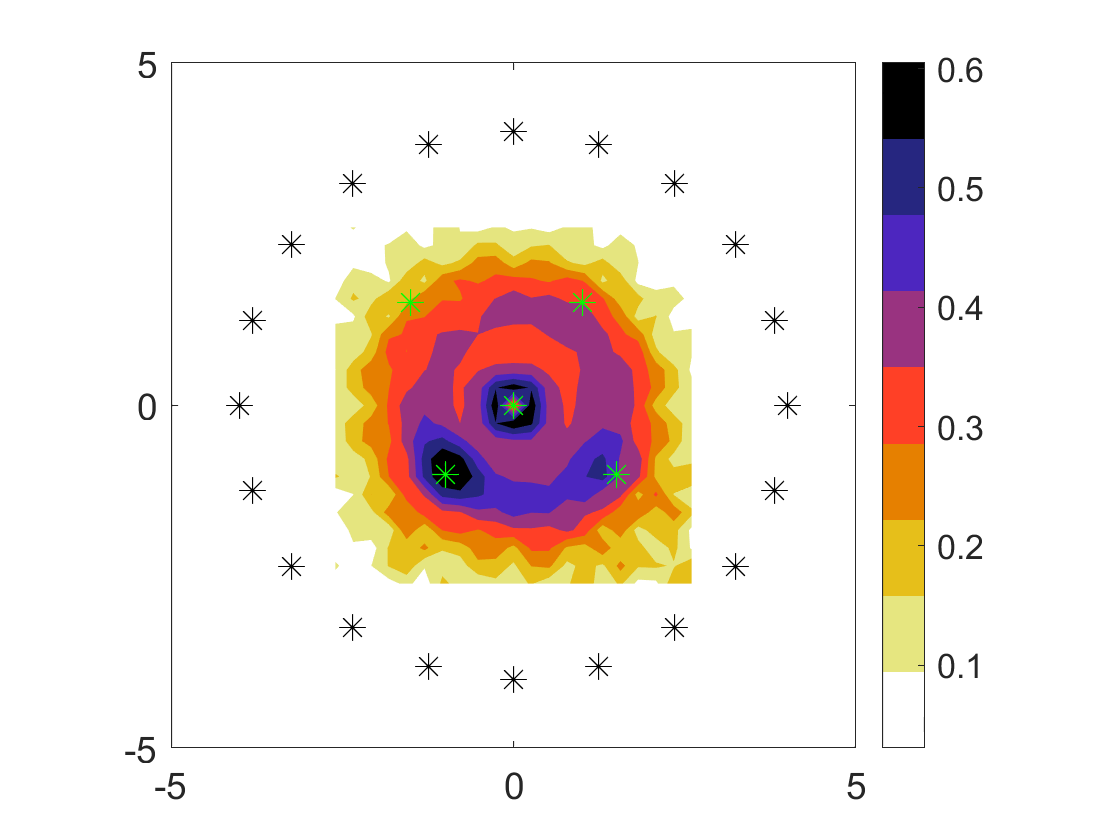}
		& \includegraphics[width=0.33\textwidth]{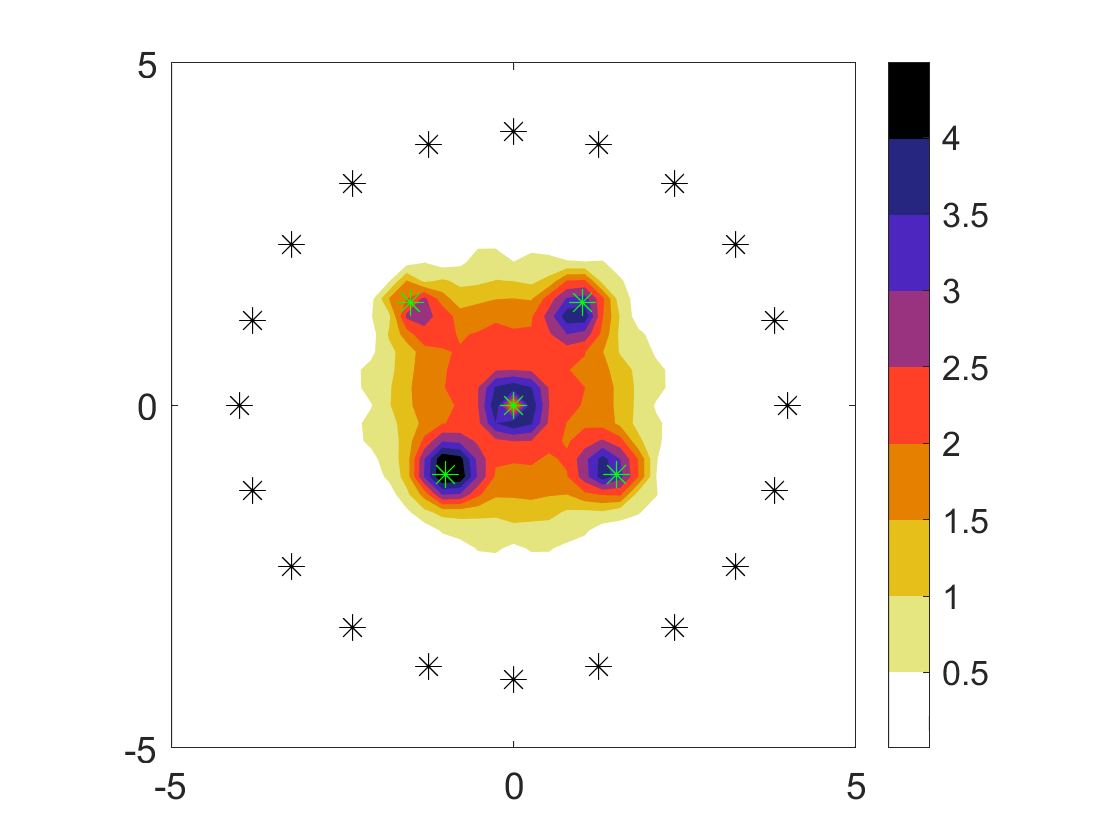}
		& \includegraphics[width=0.33\textwidth]{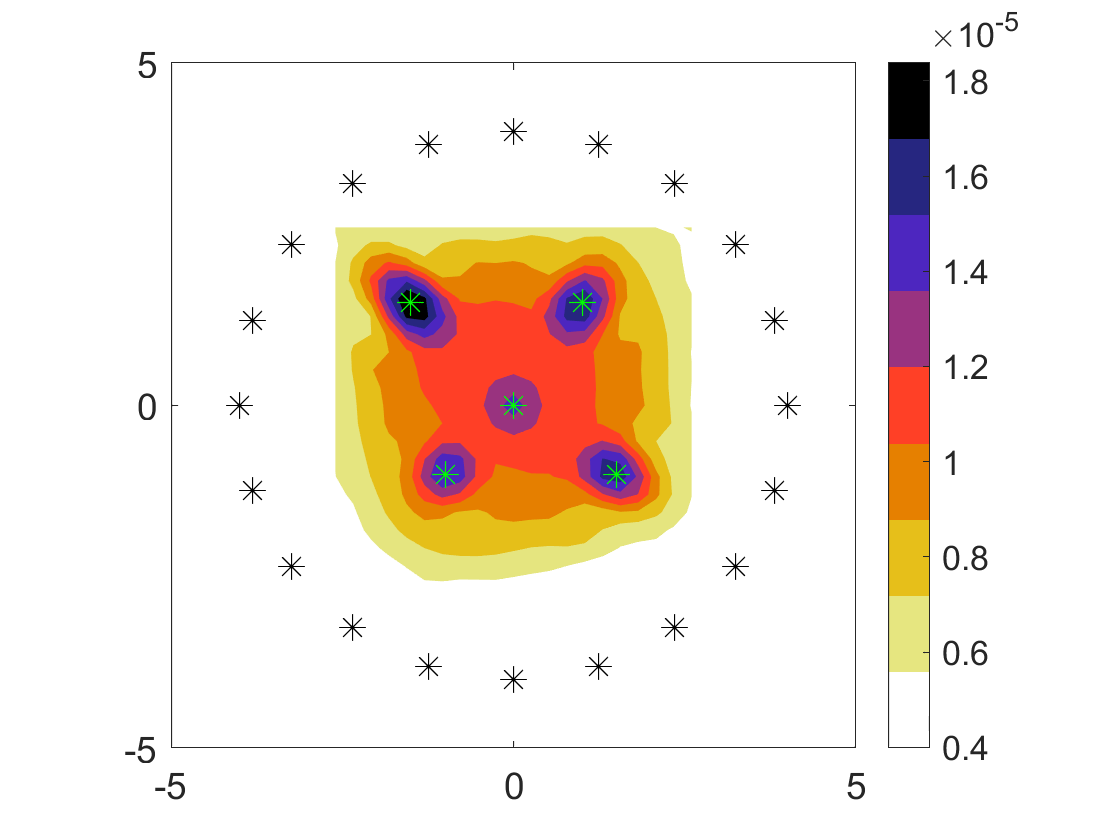} \\
		(g)~$I_1$ & (h)~$I_2$ & (i)~$I_3$
	\end{tabular}
	\caption{Reconstructions of multiple point-like scatterers using different indicator functions, $\varepsilon=5\%$. (a-c)~Reconstructions of $2$ point-like scatterers centered at $(-1,-1)$ and $(1,1.5)$, respectively. (d-f)~Reconstructions of $3$ point-like scatterers centered at $(-1,-1)$, $(1,1.5)$ and $(1.5,-1)$, respectively. (g-i)~Reconstructions of $5$ point-like scatterers centered at $(-1,-1)$, $(1,1.5)$, $(1.5,-1)$, $(-1.5,1.5)$ and $(0,0)$, respectively.}\label{fig-points}
\end{figure}
\end{example}

\begin{example}\textbf{Reconstructions of a normal size scatterer}\label{example3}

The reconstructions of a normal size scatterer from both full and limited aperture scattered data are considered in this example. The boundaries of the scatterers, such as the circle-shaped scatterer, the kite-shaped scatterer and the starfish-shaped scatterer, can be parameterized as
\begin{align}
	\label{circle} Circle:~& s(\theta)=(a,b)+1.5(\cos\theta,\sin\theta),\theta\in[0,2\pi),\\
	\label{kite} Kite:~& s(\theta)=(a,b)+(\cos\theta+0.65\cos2{\theta}-0.65,1.5 {\sin\theta}), \theta\in[0,2\pi),\\
	\label{starfish} Starfish:~& s(\theta)=(a,b)+(1+0.2\cos5\theta)(\cos\theta,\sin\theta), \theta\in[0,2\pi).
\end{align}
The reconstructions of the three kinds of scatterers with full aperture data and the reconstructions of the starfish-shaped scatterer using limited aperture data are provided in Figure \ref{fig-circle}-Figure \ref{fig-starfish-lim}. For the reconstructions from limited aperture data, $N=10$ and $N=15$ are chosen for the apertures $\theta=\pi$ and $\theta=\frac{3}{2}\pi$, respectively. The coordinates of the measurement points are
$$x_i=4\left(\cos{\frac{i\pi}{9}},\sin{\frac{i\pi}{9}}\right),\quad i=0,1,\cdots,9$$
for the aperture $\theta=\pi$ and
$$x_i=4\left(\cos{\frac{3i\pi}{28}},\sin{\frac{3i\pi}{28}}\right),\quad i=0,1,\cdots,14$$
for the aperture $\theta=\frac{3}{2}\pi$.
The exact boundaries of the scatterers are marked with green dash lines. From Figure \ref{fig-circle}-Figure \ref{fig-starfish-lim}, we can find that Algorithm 1 and Algorithm 2 can roughly reconstruct a normal size scatterer centered at $(a,b)=(0,0)$, but are not feasible to reconstruct a scatterer that is not centered at the origin. Nevertheless, Algorithm 3 can reconstruct a normal size scatterer no matter the scatterer is centered at the origin or not.
For the reconstructions of the scatterers with limited aperture data, only the part of the boundary curve that is close to the measurement points can be well reconstructed.
% Figure3 - circular scatterer
\begin{figure}
	\centering
	\begin{tabular}{ccc}
		\includegraphics[width=0.33\textwidth]{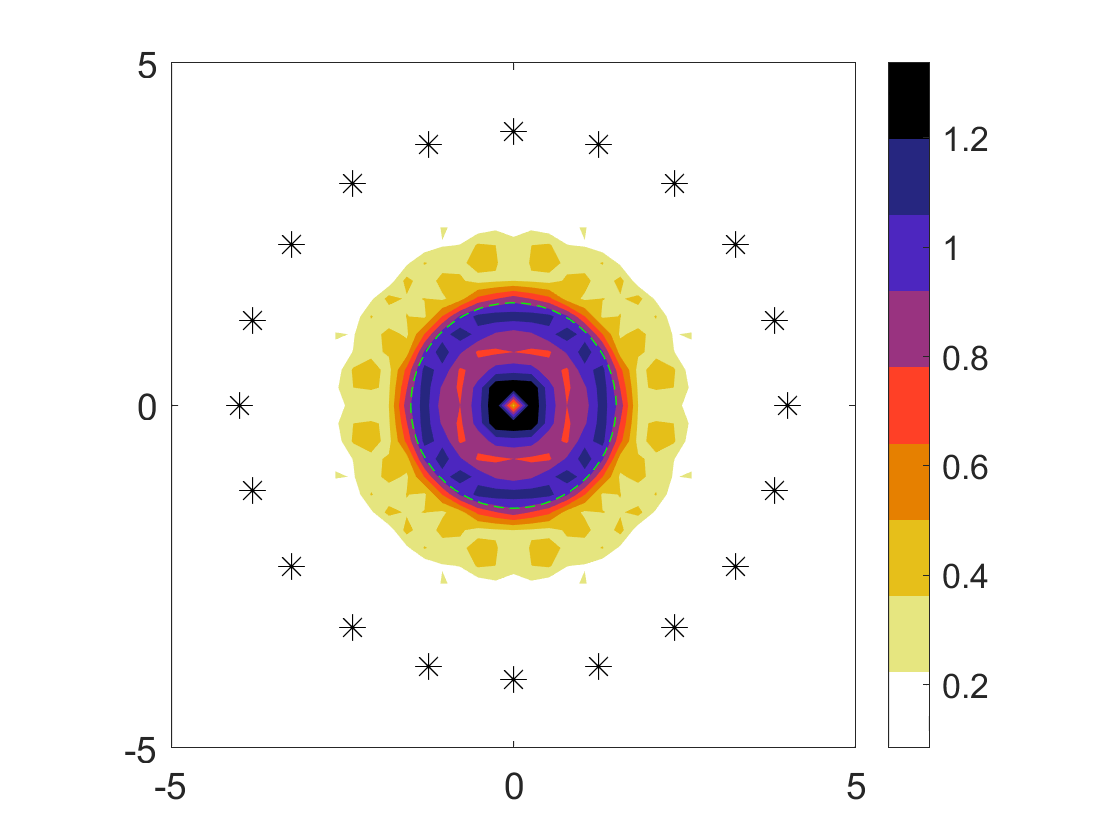}
		& \includegraphics[width=0.33\textwidth]{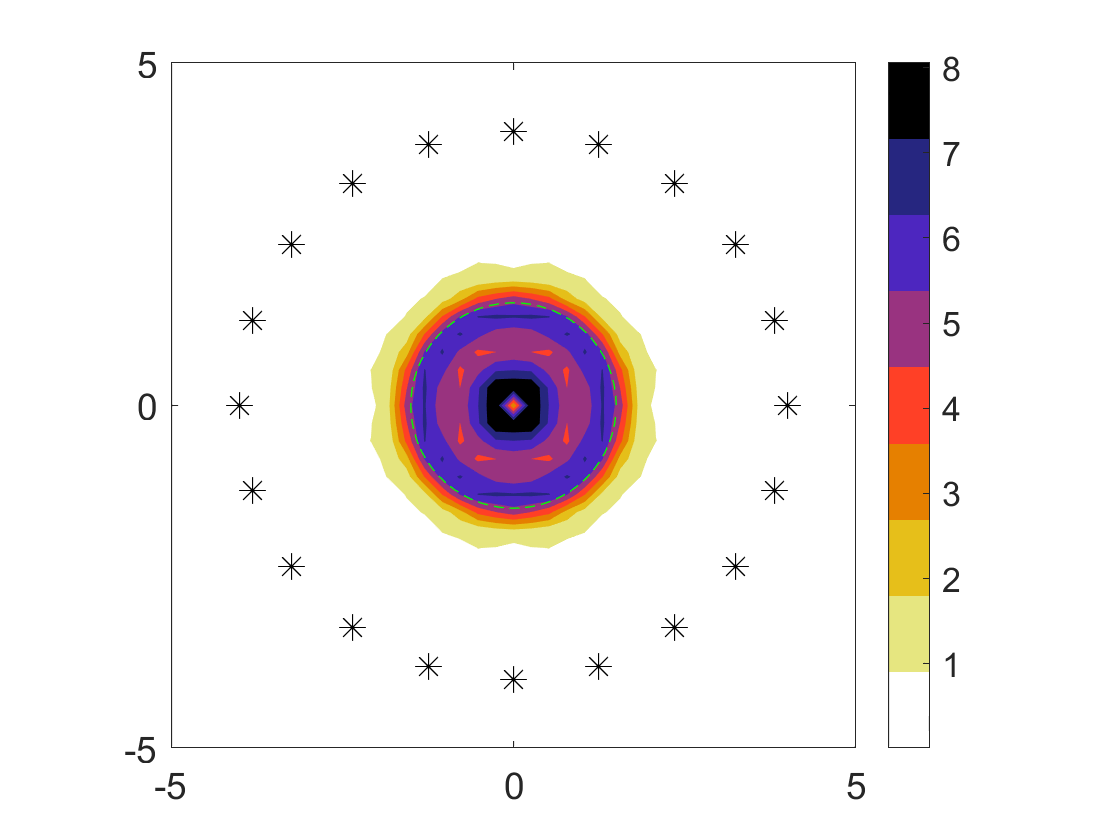}
		& \includegraphics[width=0.33\textwidth]{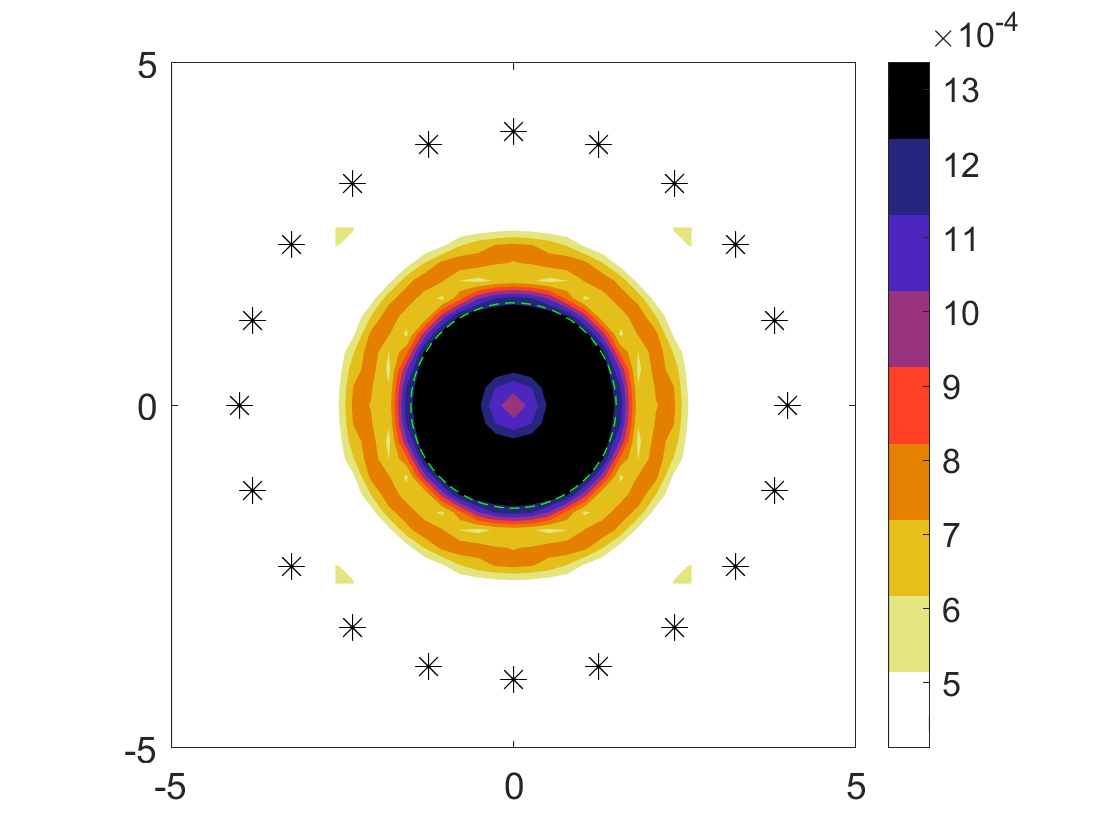} \\
		(a)~$I_1$ & (b)~$I_2$ & (c)~$I_3$ \\
		\includegraphics[width=0.33\textwidth]{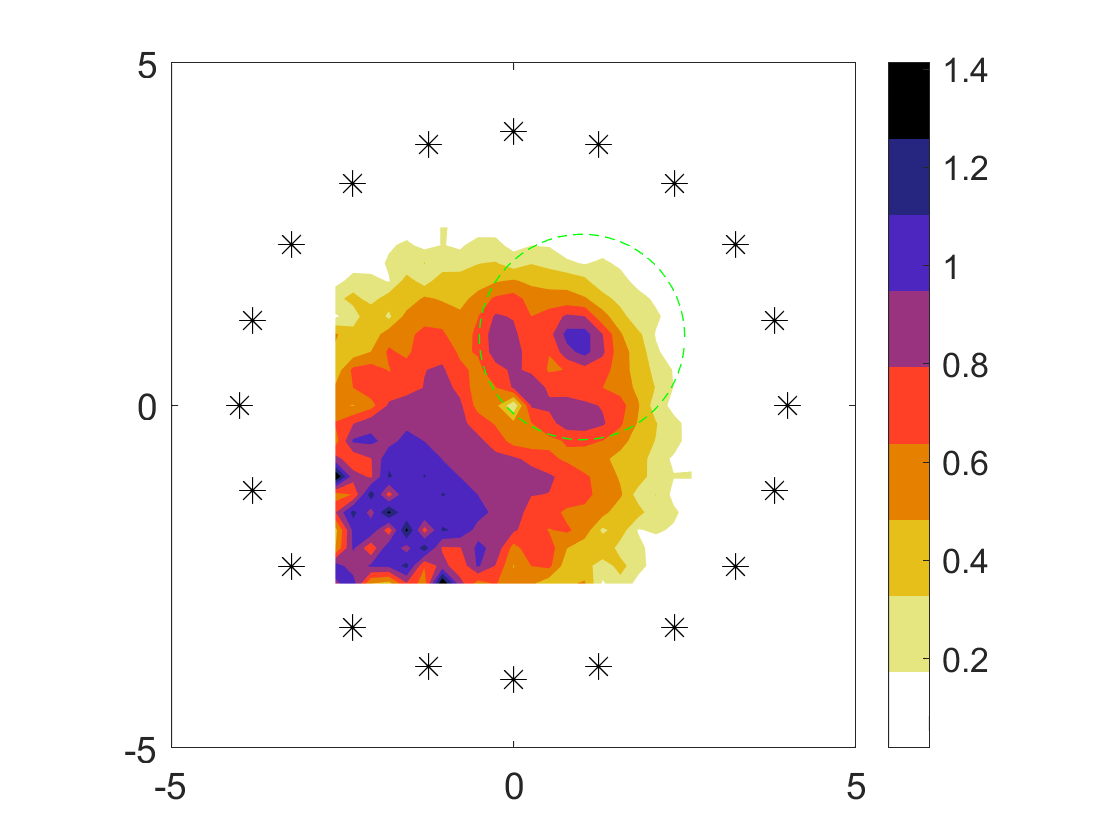}
		& \includegraphics[width=0.33\textwidth]{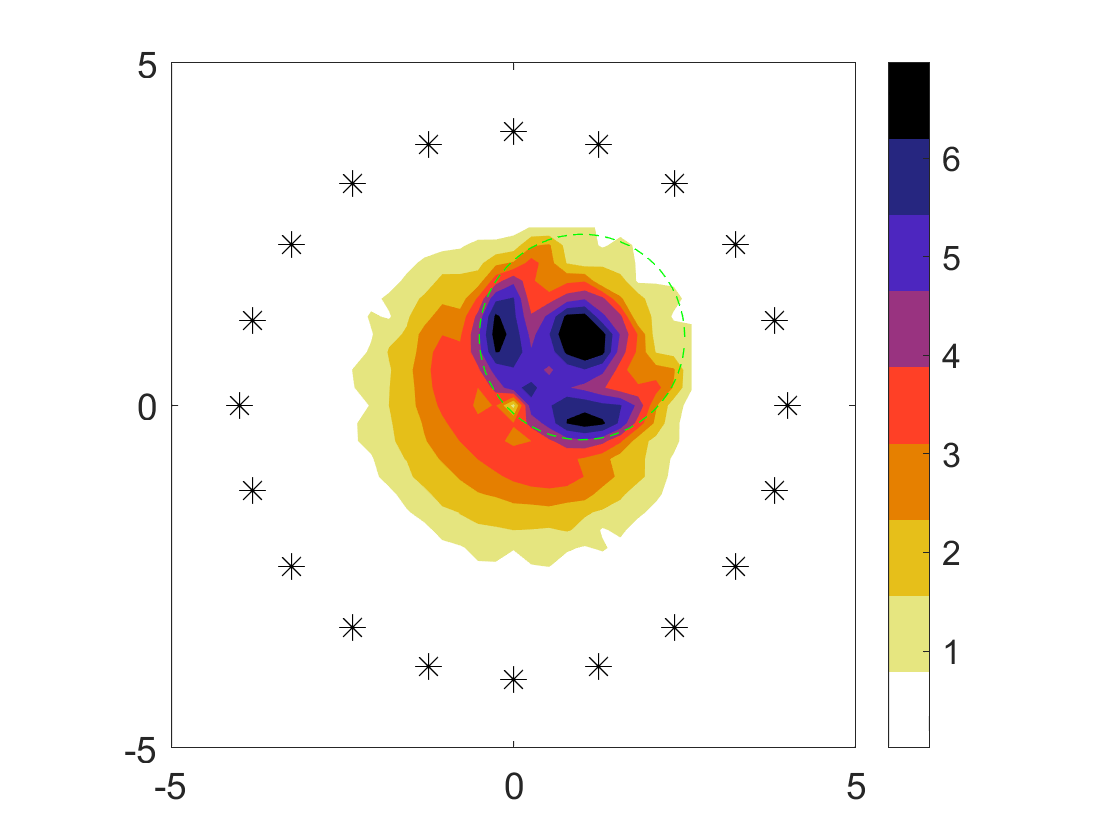}
		& \includegraphics[width=0.33\textwidth]{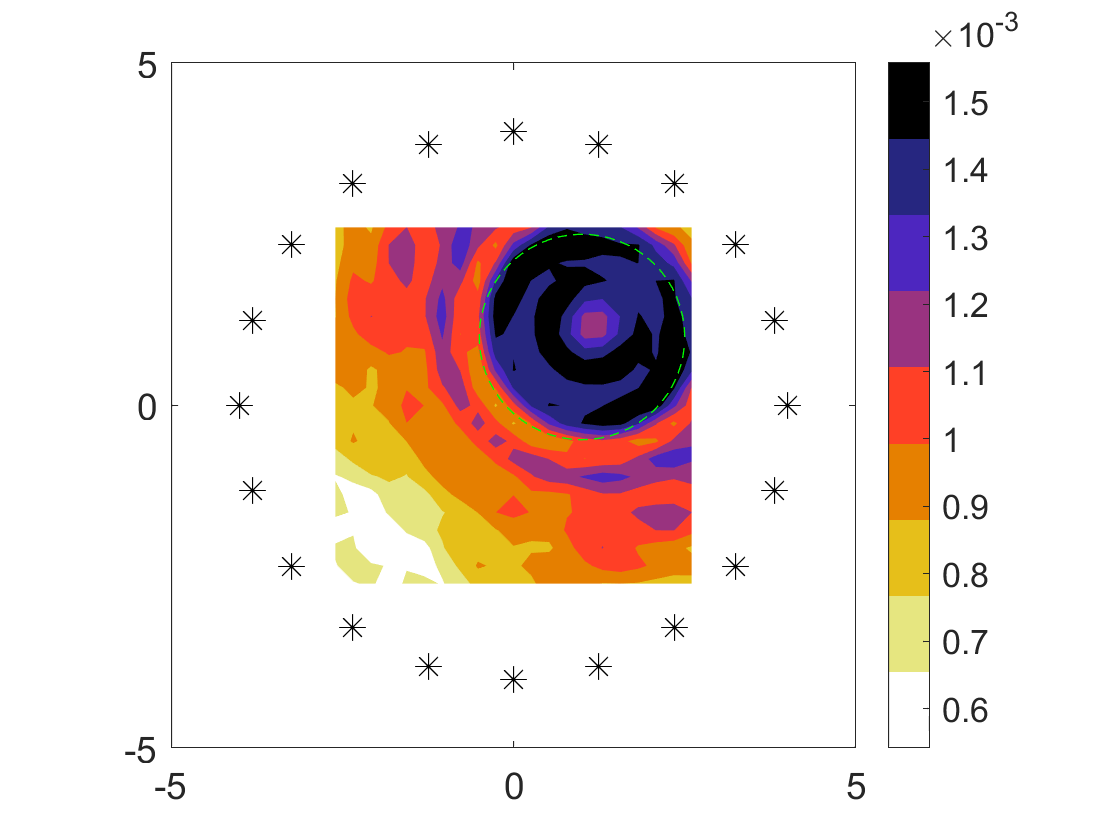} \\
		(d)~$I_1$ & (e)~$I_2$ & (f)~$I_3$
	\end{tabular}
	\caption{Reconstructions of circular scatterers using different indicator functions, $\varepsilon=5\%$. The circles are centered at $(0,0)$ in the first row and centered at $(1,1)$ in the second row.}\label{fig-circle}
\end{figure}

% Figure4 - kite shaped scatterer
\begin{figure}
	\centering
	\begin{tabular}{ccc}
		\includegraphics[width=0.33\textwidth]{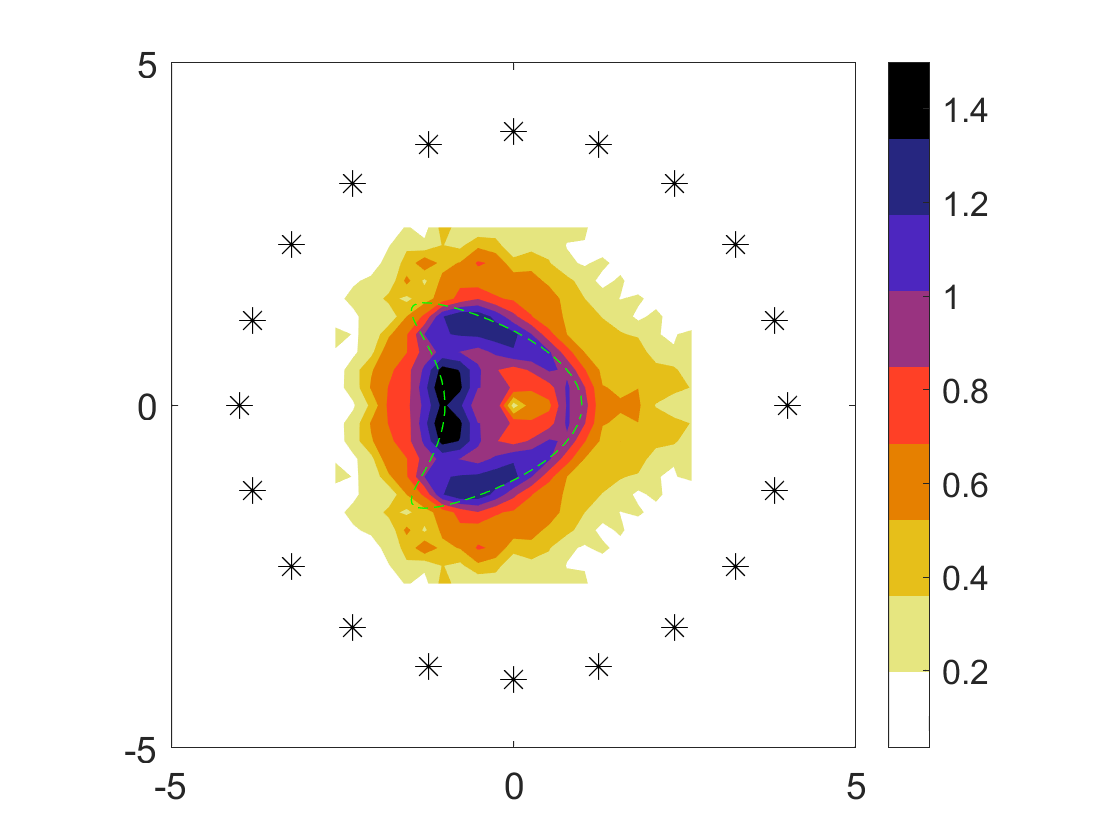}
		& \includegraphics[width=0.33\textwidth]{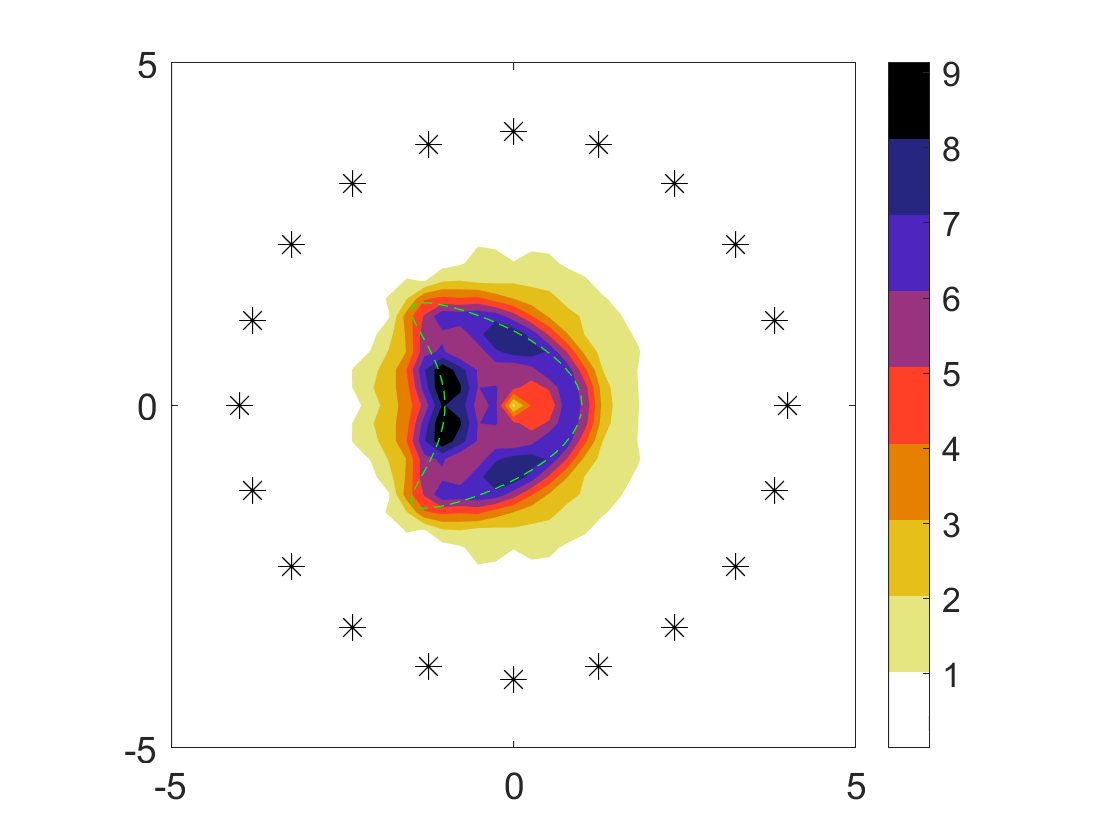}
		& \includegraphics[width=0.33\textwidth]{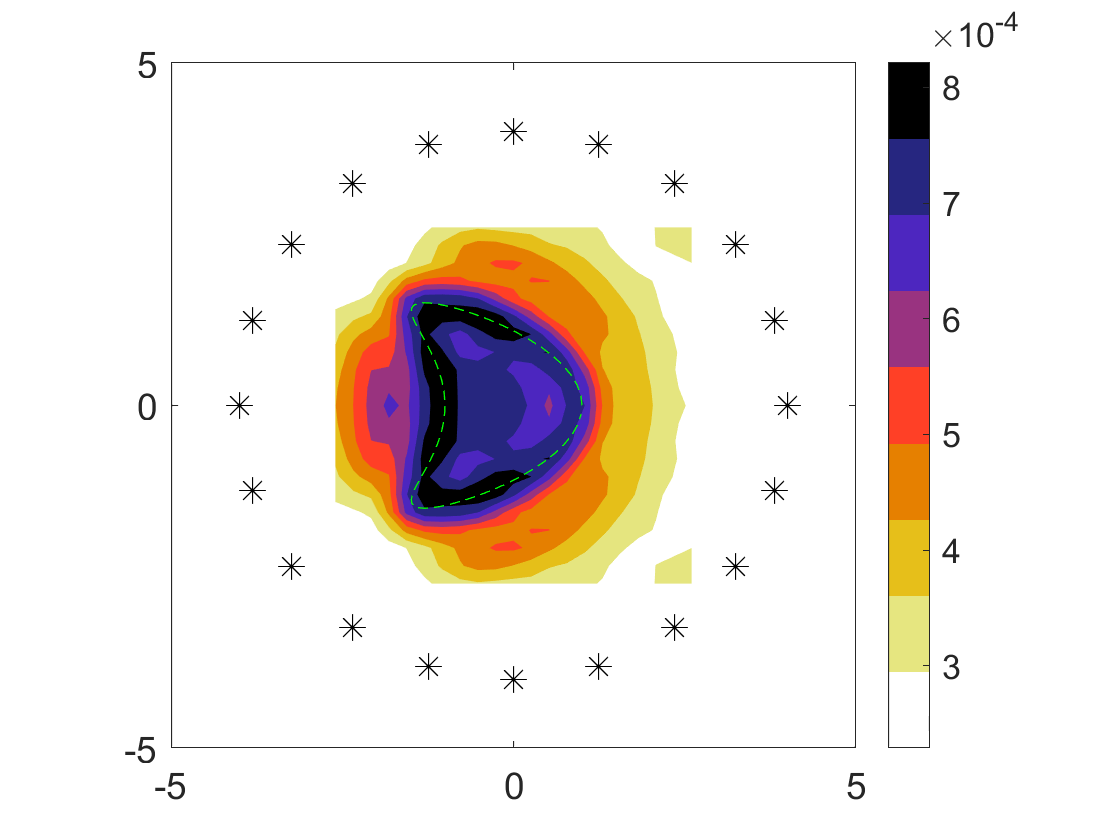} \\
		(a)~$I_1$ & (b)~$I_2$ & (c)~$I_3$ \\
		\includegraphics[width=0.33\textwidth]{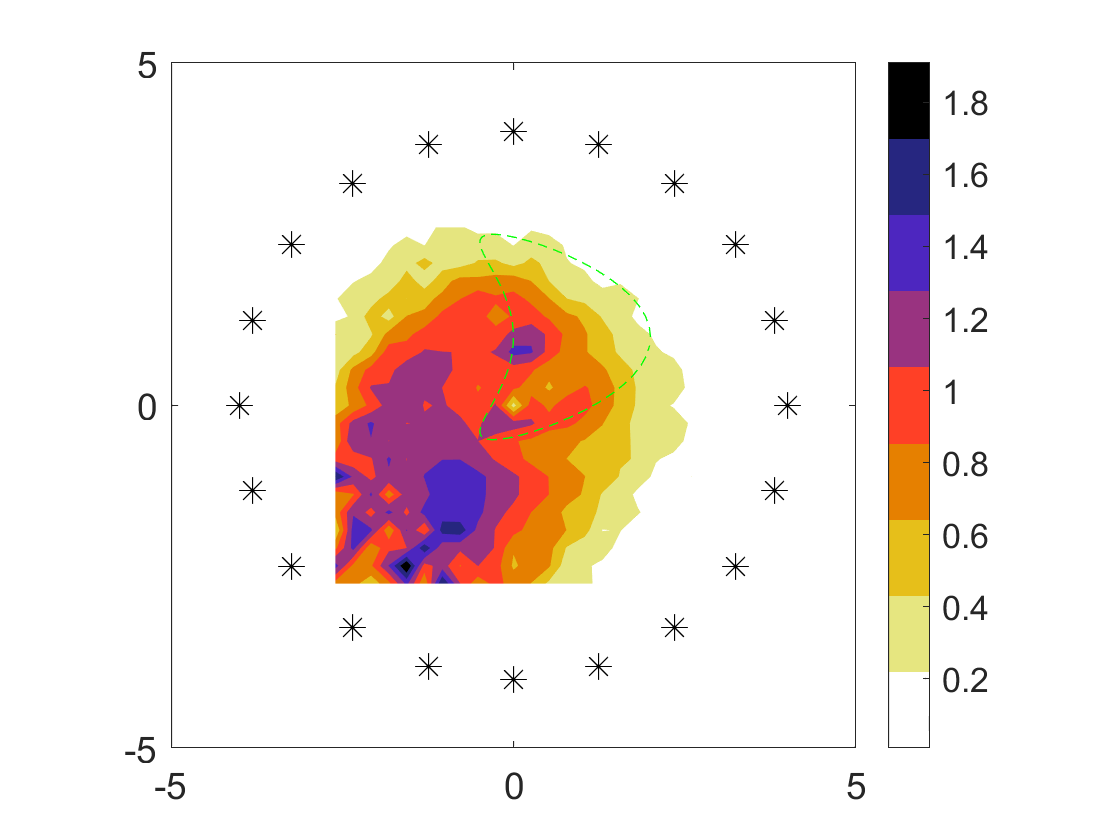}
		& \includegraphics[width=0.33\textwidth]{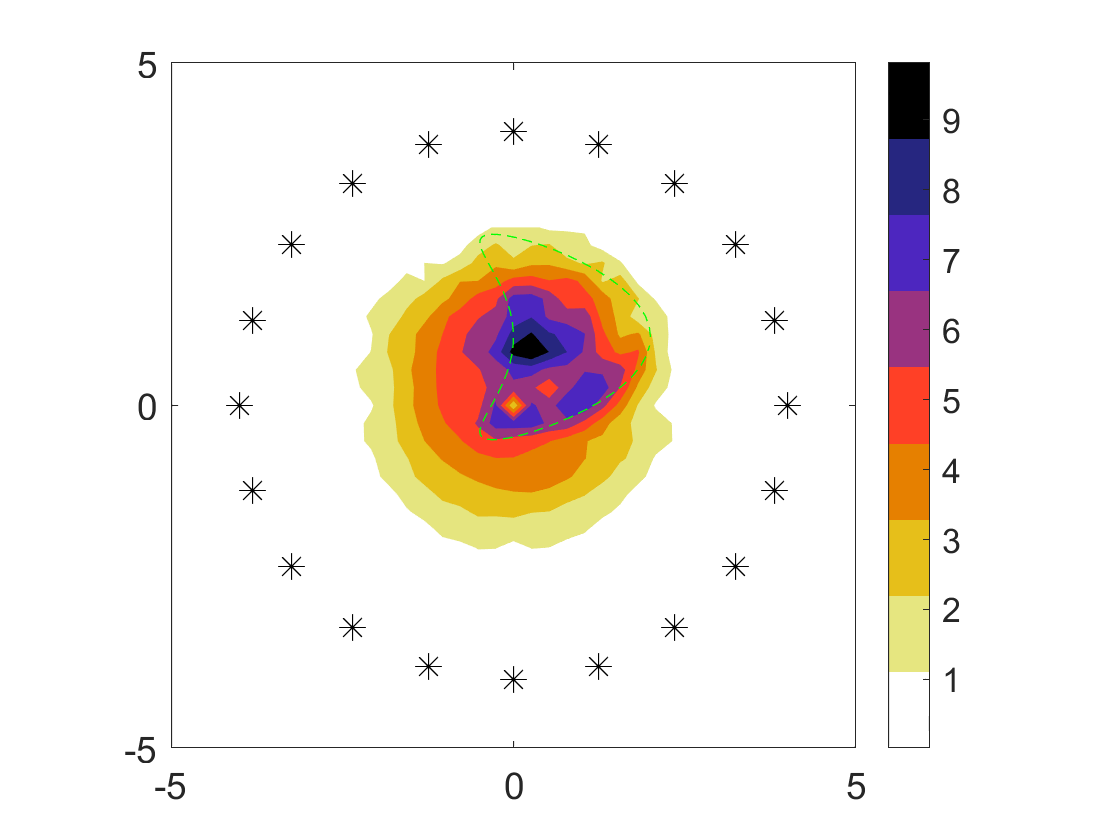}
		& \includegraphics[width=0.33\textwidth]{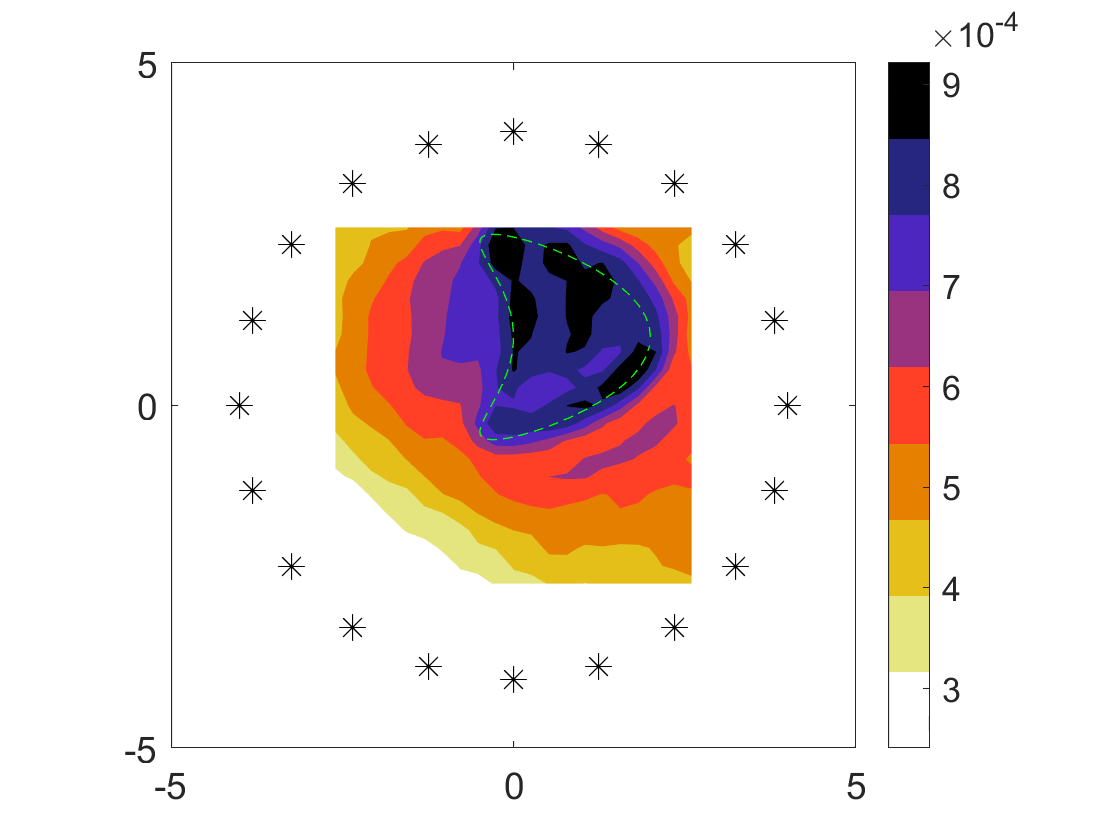} \\
		(d)~$I_1$ & (e)~$I_2$ & (f)~$I_3$
	\end{tabular}
	\caption{Reconstructions of kite-shaped scatterers using different indicator functions, $\varepsilon=5\%$. The kite is centered at $(0,0)$ in the first row and centered at $(1,1)$ in the second row.}\label{fig-kite}
\end{figure}

%% Figure5 - starfish shaped scatterer
\begin{figure}
	\centering
	\begin{tabular}{ccc}
		\includegraphics[width=0.33\textwidth]{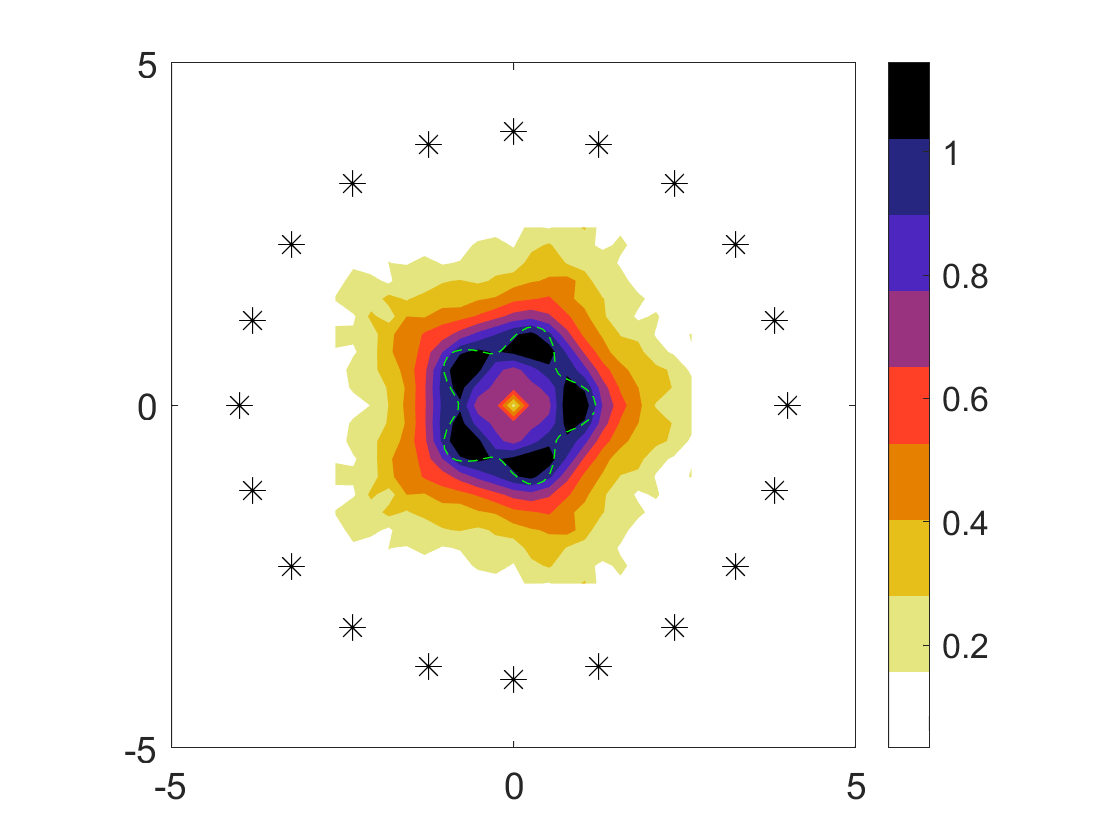}
		& \includegraphics[width=0.33\textwidth]{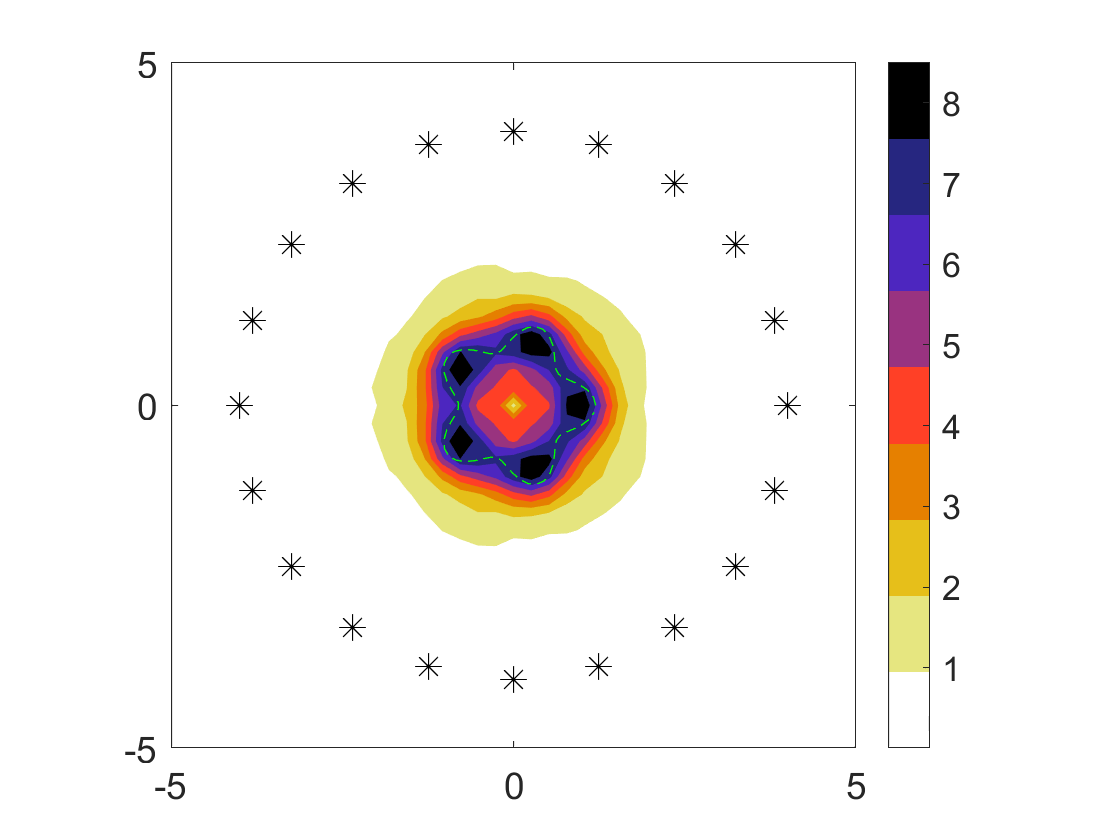}
		& \includegraphics[width=0.33\textwidth]{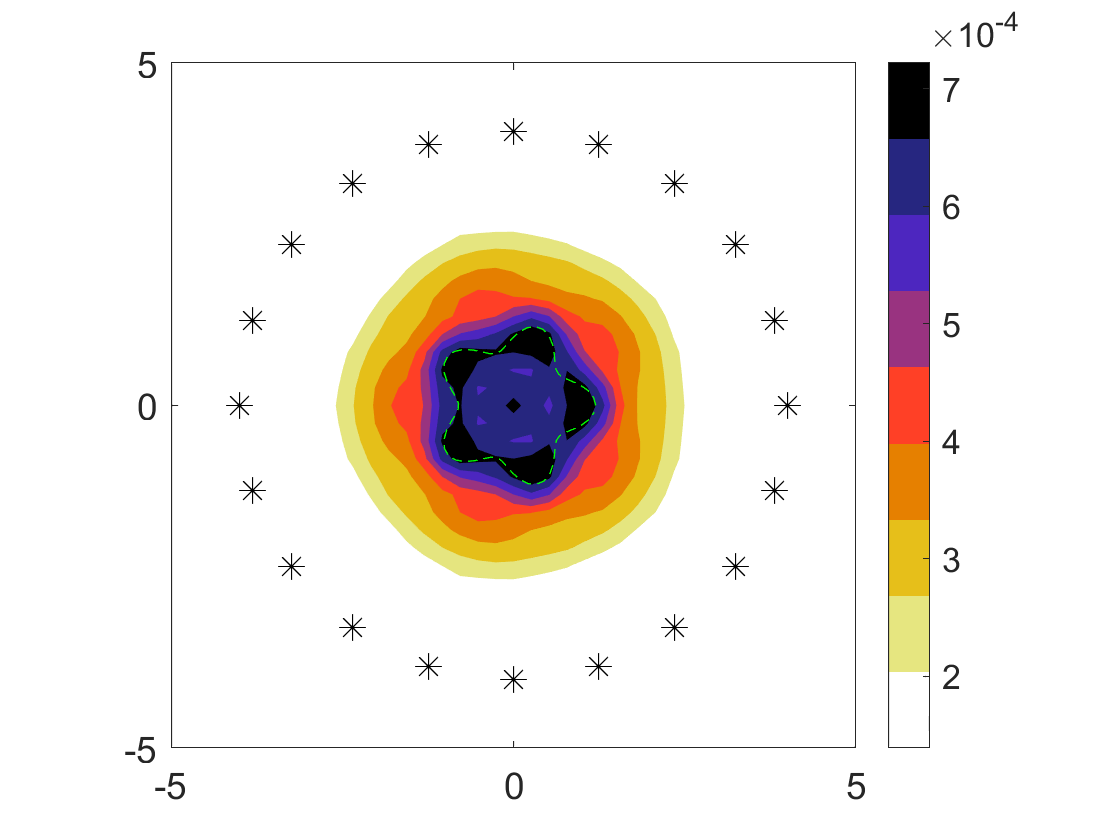} \\
		(a)~$I_1$ & (b)~$I_2$ & (c)~$I_3$ \\
		\includegraphics[width=0.33\textwidth]{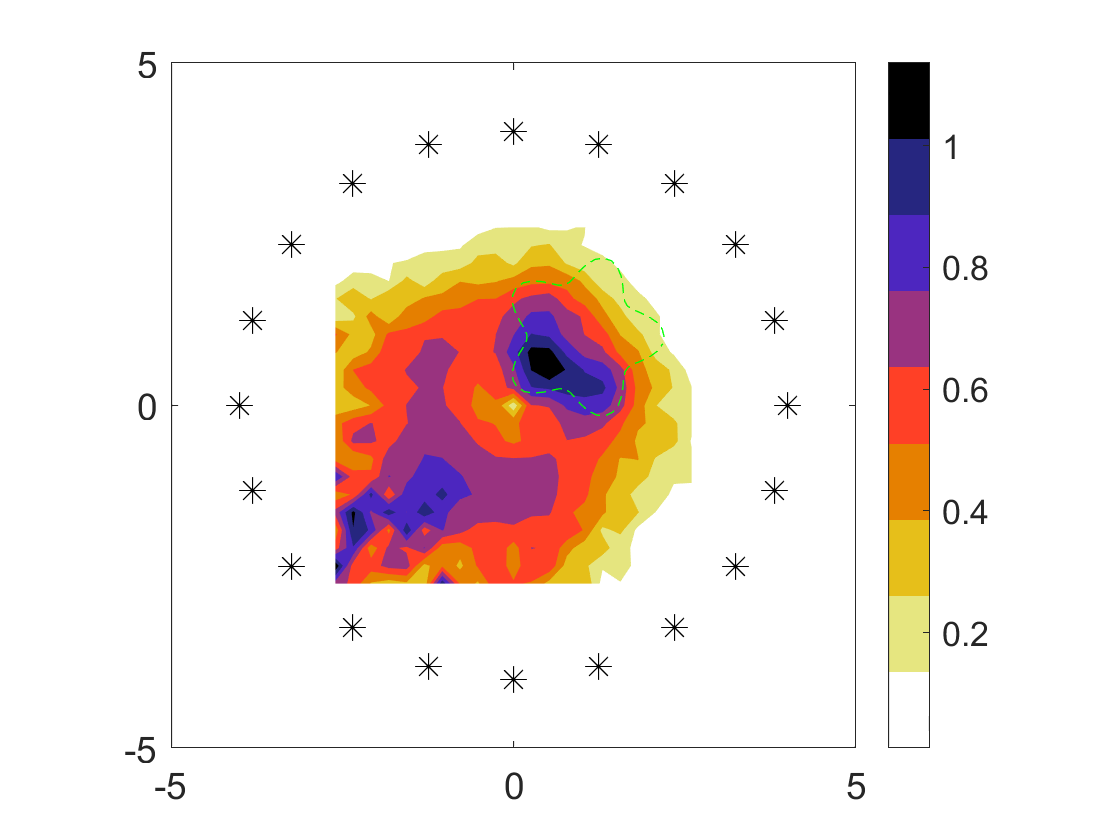}
		& \includegraphics[width=0.33\textwidth]{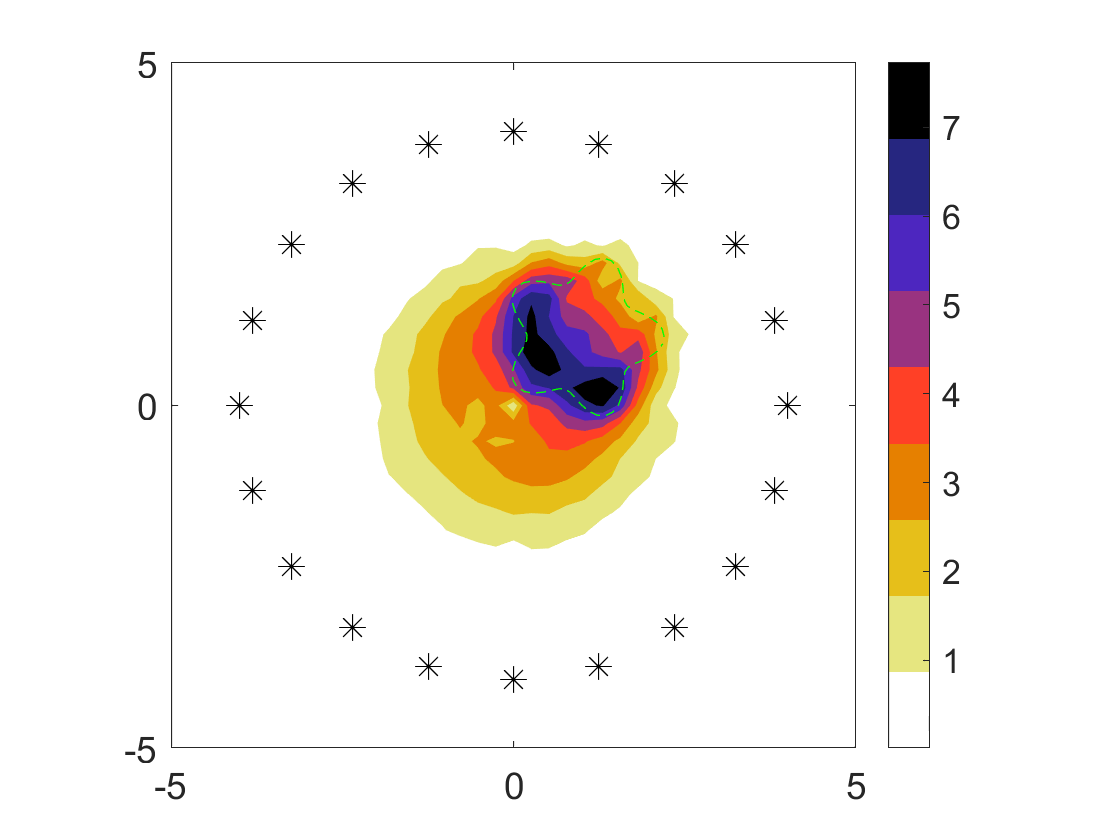}
		& \includegraphics[width=0.33\textwidth]{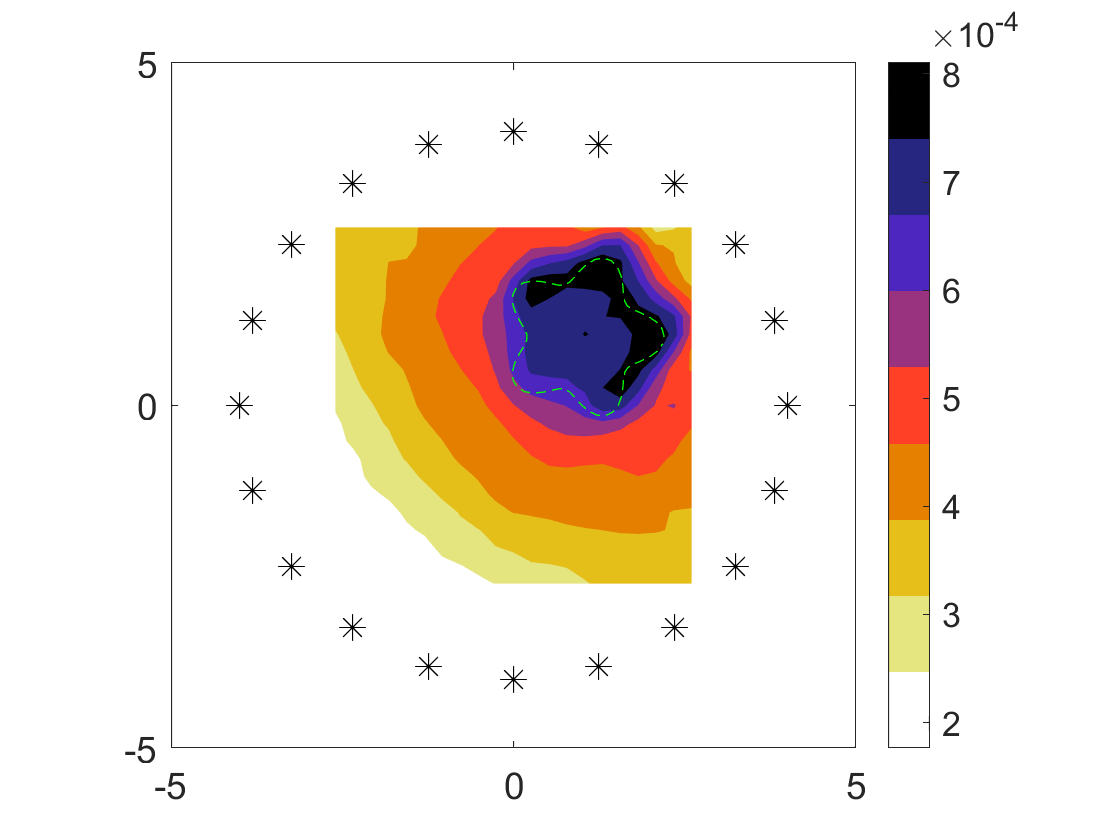} \\
		(d)~$I_1$ & (e)~$I_2$ & (f)~$I_3$
	\end{tabular}
	\caption{Reconstructions of starfish-shaped scatterers using different indicator functions, $\varepsilon=5\%$. The starfish is centered at $(0,0)$ in the first row and centered at $(1,1)$ in the second row.}\label{fig-starfish}
\end{figure}

%% Figure6 - starfish shaped scatterer - limited aperture data

\begin{figure}
	\centering
	\begin{tabular}{ccc}
		\includegraphics[width=0.33\textwidth]{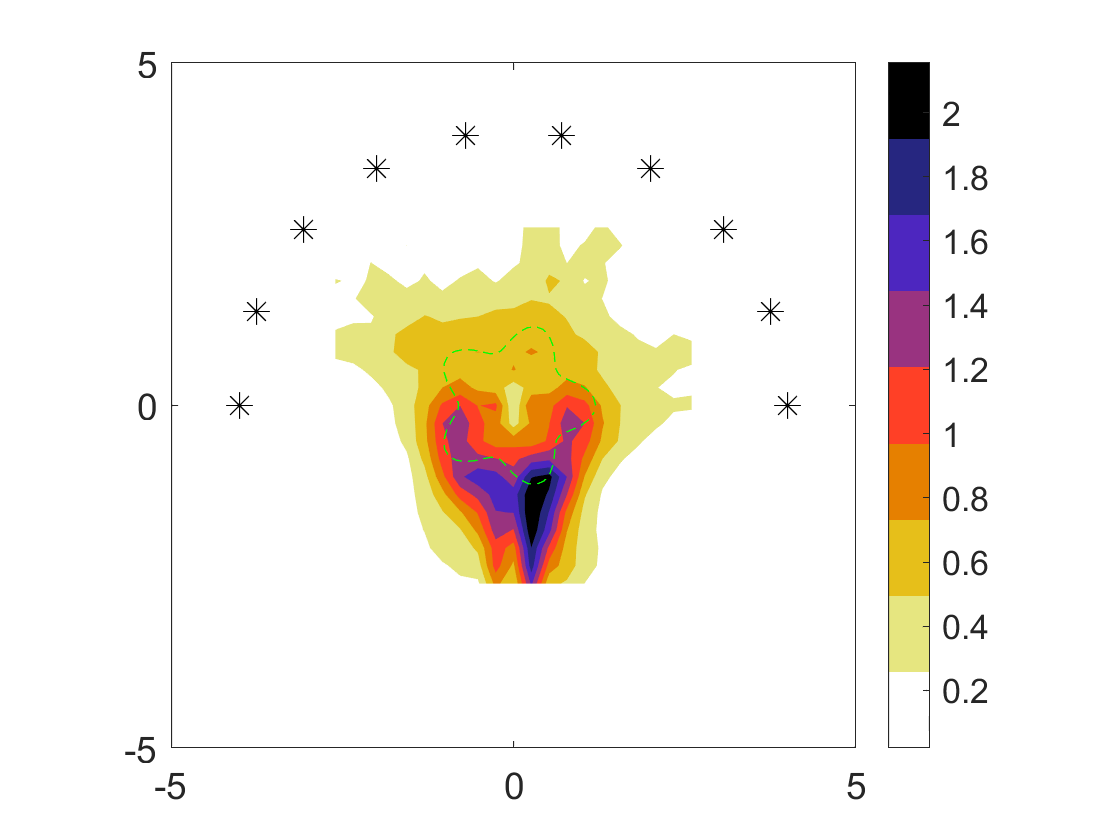}
		& \includegraphics[width=0.33\textwidth]{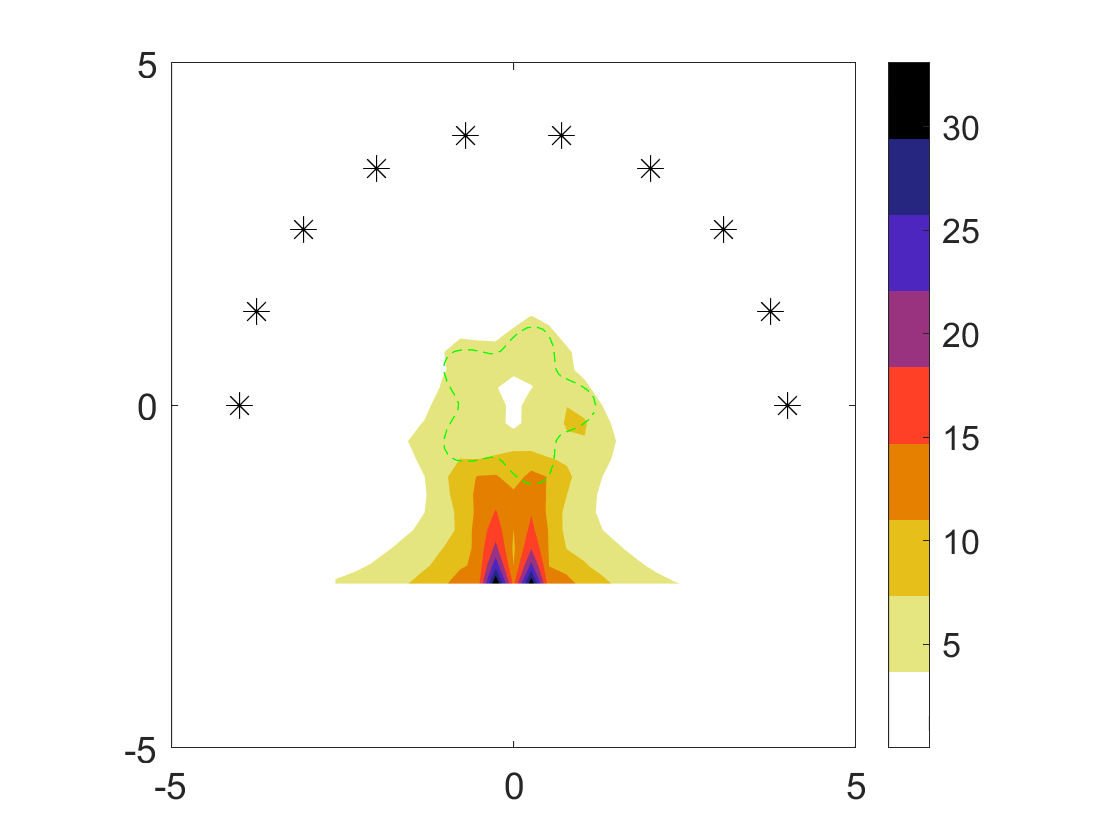}
		& \includegraphics[width=0.33\textwidth]{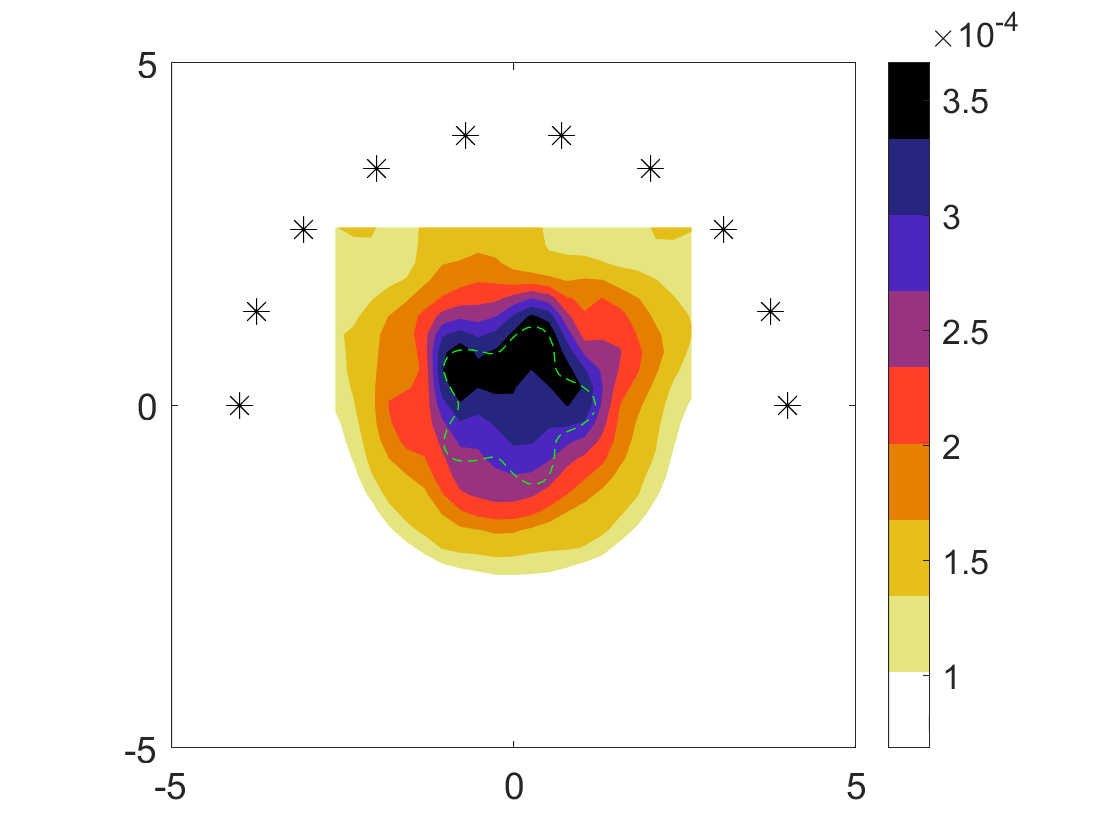} \\
		(a)~$I_1,\theta=\pi$ & (b)~$I_2,\theta=\pi$ & (c)~$I_3,\theta=\pi$\\
		\includegraphics[width=0.33\textwidth]{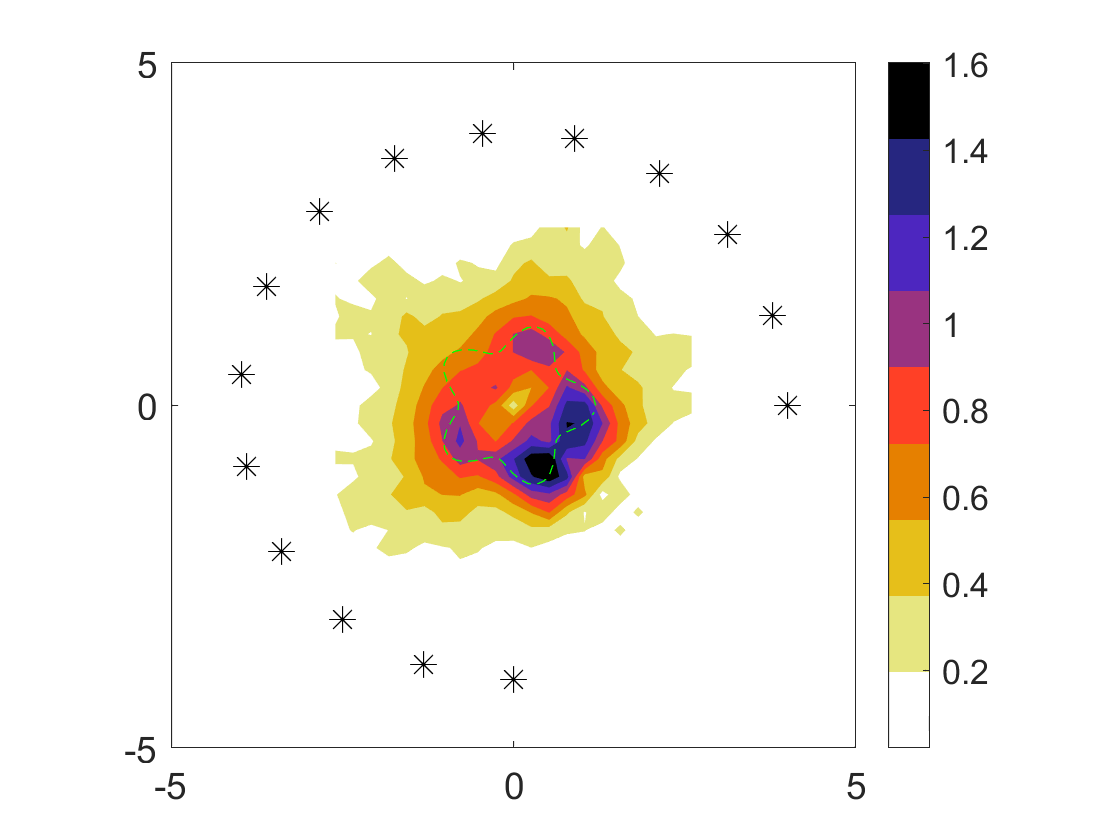}
		& \includegraphics[width=0.33\textwidth]{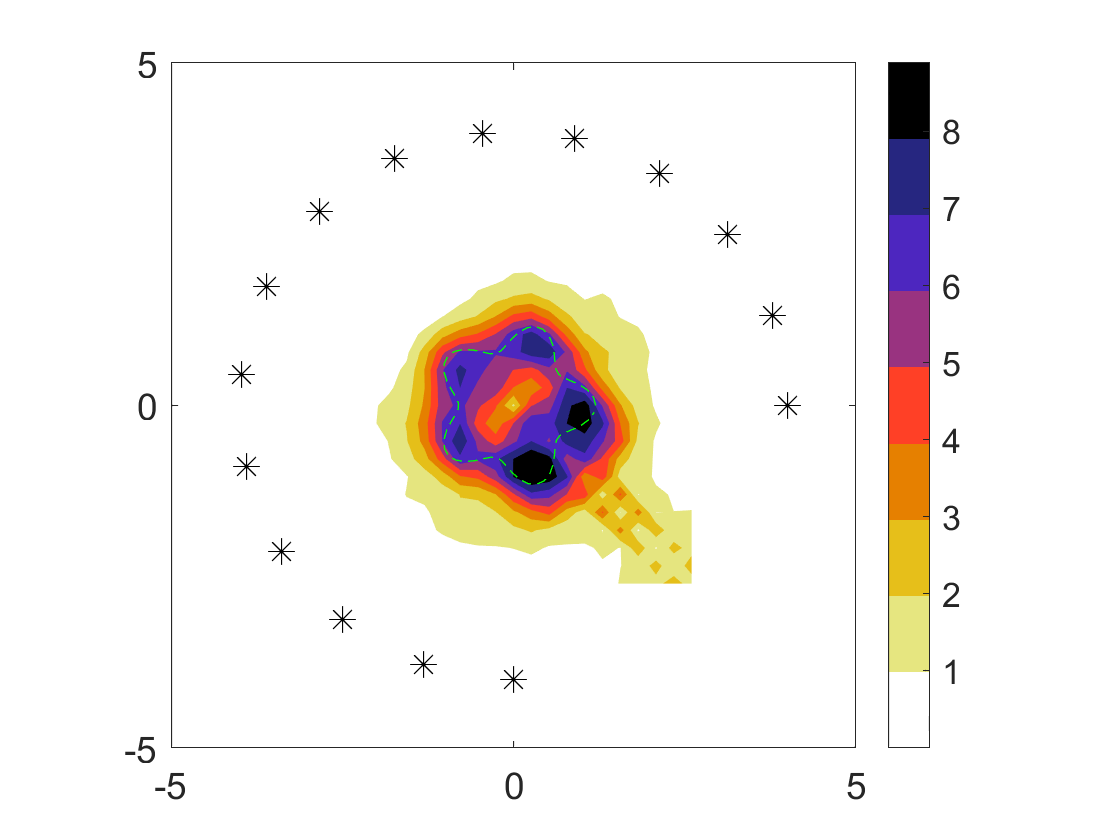}
		& \includegraphics[width=0.33\textwidth]{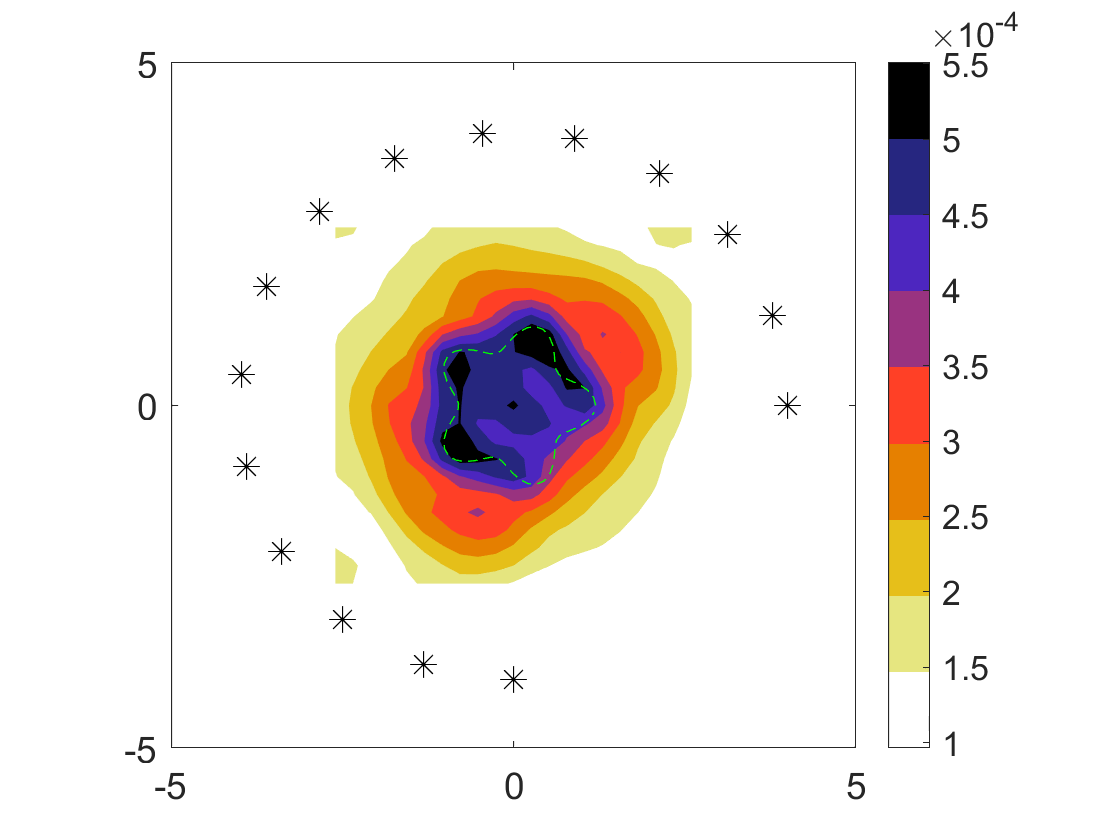}\\
		(d)~$I_1,\theta=\frac{3}{2}\pi$ & (e)~$I_2,\theta=\frac{3}{2}\pi$ & (f)~$I_3,\theta=\frac{3}{2}\pi$
	\end{tabular}
	\caption{Reconstructions of starfish-shaped scatterer centered at $(0,0)$ from limited aperture data using different indicator functions, $\varepsilon=5\%$.}\label{fig-starfish-lim}
\end{figure}
\end{example}

\begin{example}\textbf{Simultaneous reconstructions of a normal size scatterer and a point-like scatterer}\label{example4}

In this example, the simultaneous reconstructions of a normal size scatterer and a point-like scatterer are considered. The boundaries of the point-like, acorn-shaped and rounded-square-shaped scatterers centered at $(a,b)$ are parameterized as
\begin{small}
	\begin{align}
		Point: s(\theta) =&  (a,b)+0.1(\cos\theta,\sin\theta),\quad \theta\in[0,2\pi). \label{pointt}  \\
		Acorn: s(\theta) =&  (a,b)+0.84\left(\frac{17}{4}+2\cos3\theta\right)^{1/2} (\cos\theta,\sin\theta),\quad \theta\in[0,2\pi). \label{acorn} \\
		Rounded-square: s(\theta) =&  (a,b)+\frac{\sqrt{2}}{2}(\cos^3\theta+\sin^3\theta+ \cos\theta+\sin\theta, \nonumber \\
		& -\cos^3\theta+\sin^3\theta-\cos\theta+\sin\theta),\quad \theta\in[0,2\pi).\label{square}
	\end{align}
\end{small}

The point-like scatterer is chosen to be centered at $(2.2,2.2)$, which is marked with green asterisk. The acorn-shaped and rounded-square-shaped scatterers are centered at $(0,0)$, the exact boundaries are marked with green dash lines. The reconstructions are shown in Figure \ref{fig-normal+point}. The result shows that Algorithm 3 can roughly reconstruct the shapes of acorn-shaped or rounded-square-shaped scatterers together with the location of the point-like scatterer.

%% Figure7 - a normal size scatterer and a point-like scatterer
\begin{figure}
	\centering
	\begin{tabular}{ccc}
		\includegraphics[width=0.33\textwidth]{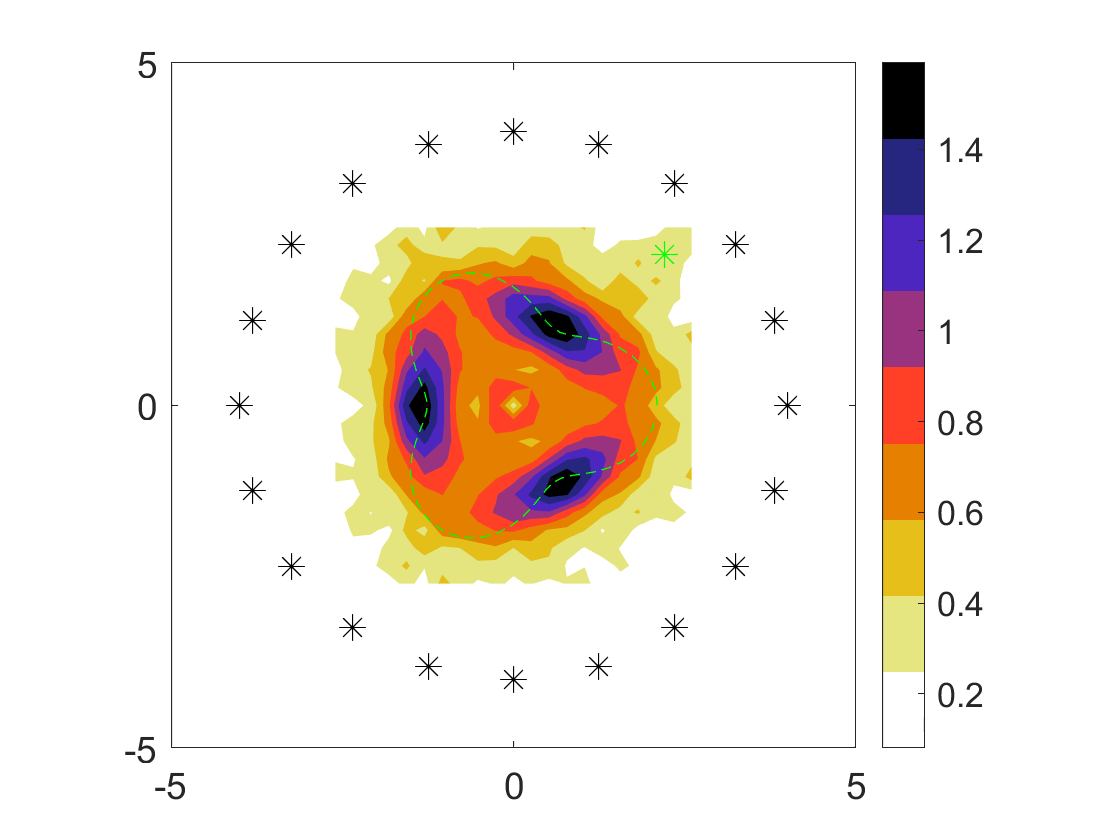}
		& \includegraphics[width=0.33\textwidth]{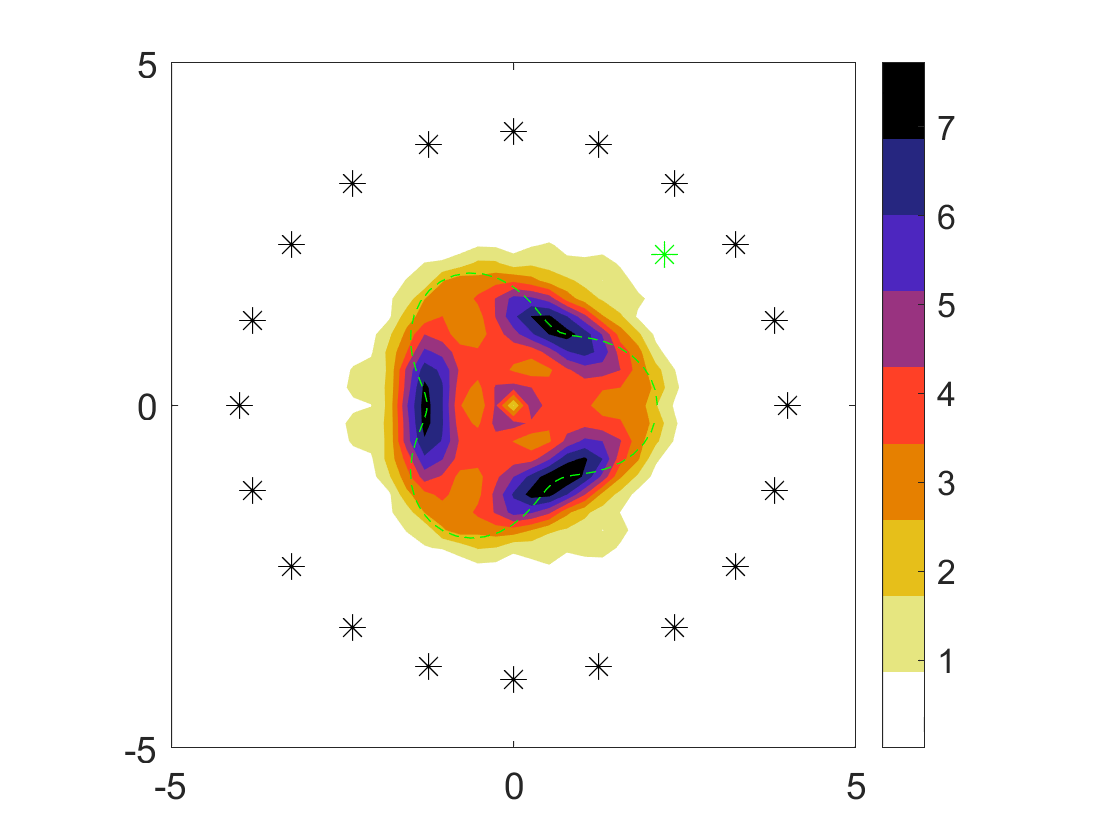}
		& \includegraphics[width=0.33\textwidth]{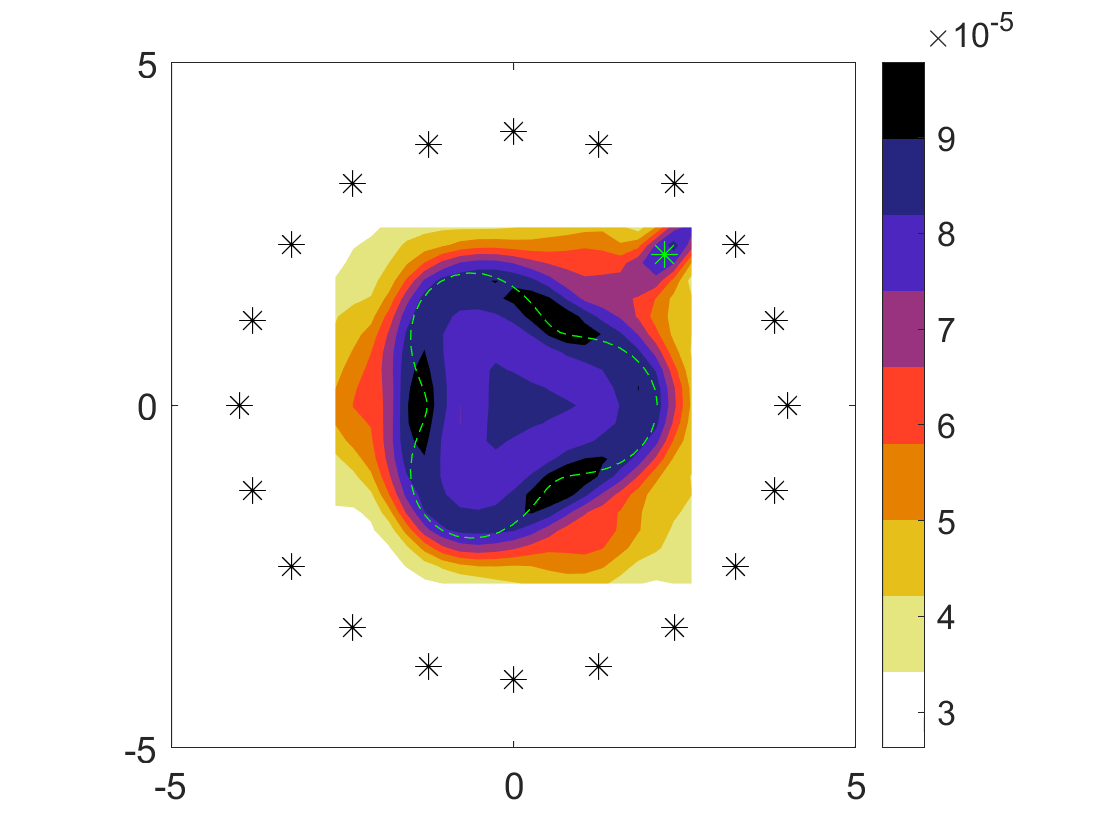} \\
		(a)~$I_1$ & (b)~$I_2$ & (c)~$I_3$\\
		\includegraphics[width=0.33\textwidth]{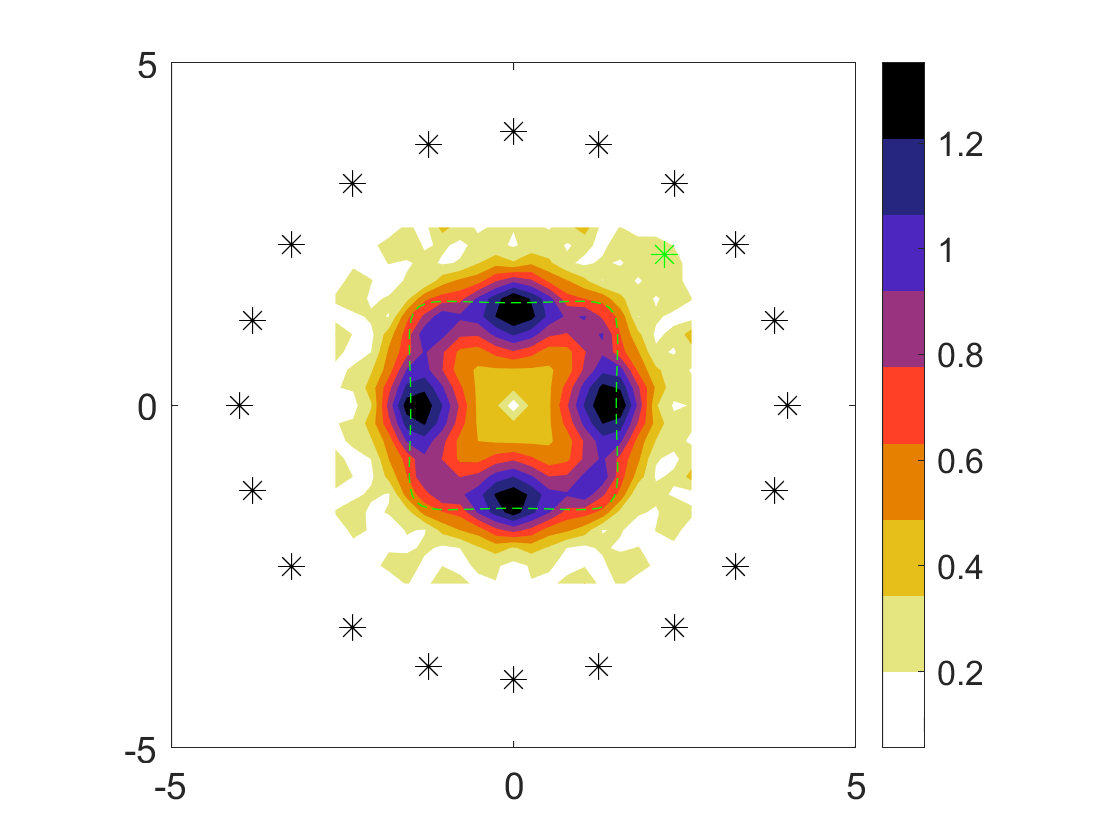}
		& \includegraphics[width=0.33\textwidth]{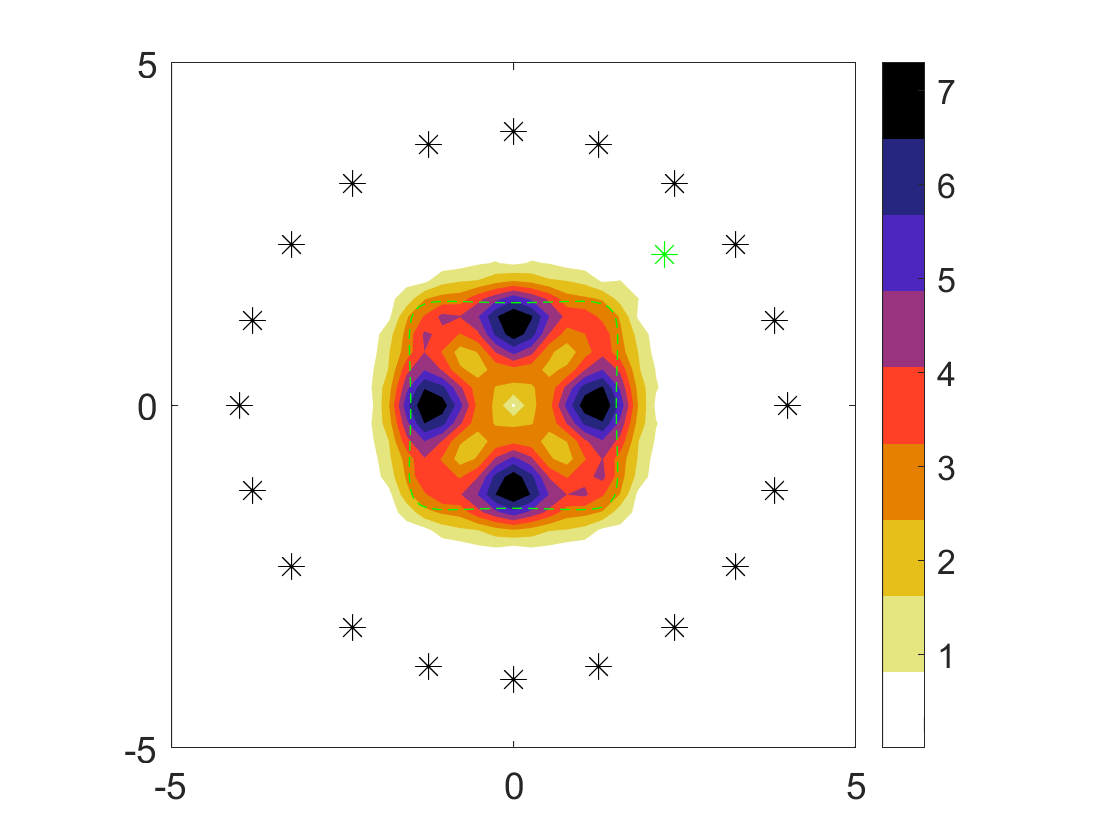}
		& \includegraphics[width=0.33\textwidth]{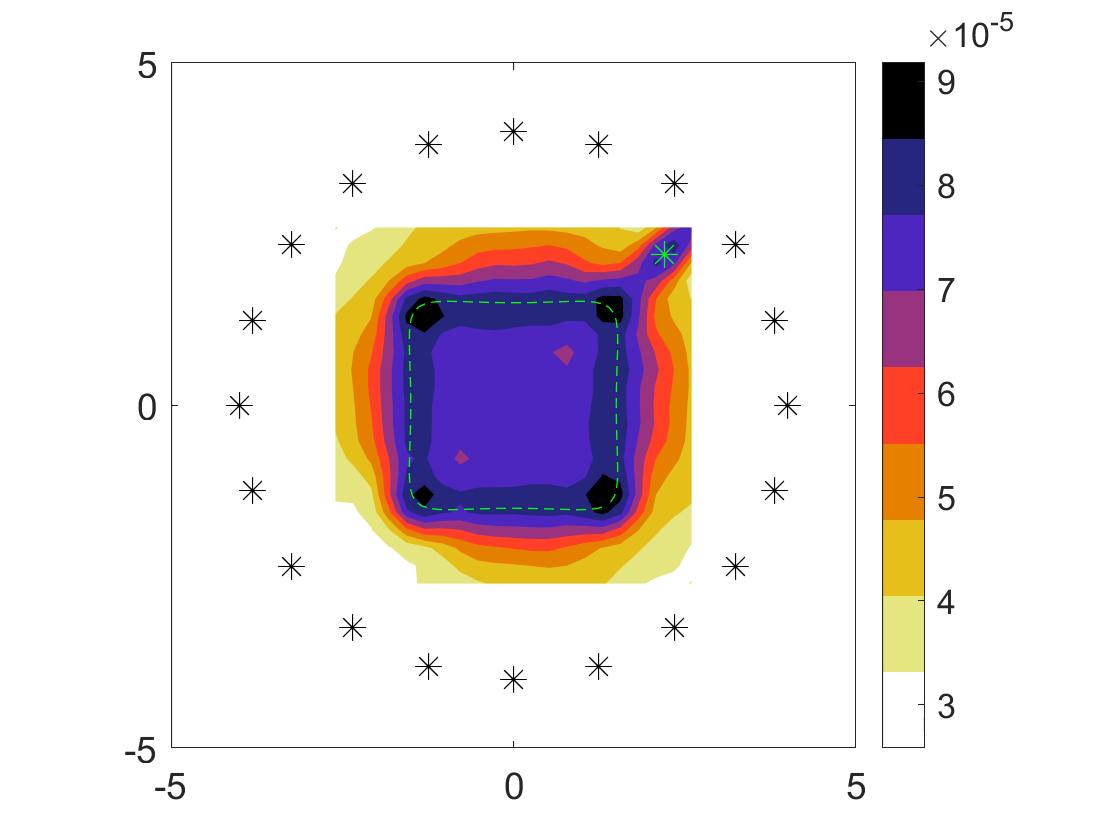}\\
		(d)~$I_1$ & (e)~$I_2$ & (f)~$I_3$
	\end{tabular}
	\caption{Reconstructions of a normal size scatterer centered at $(0,0)$ and a point-like scatterer centered at $(2.2,2.2)$ using different indicator functions, $\varepsilon=5\%$. (a-c)~Reconstructions of the acorn-shaped scatterer and the point-like scatterer. (d-f)~Reconstructions of the rounded-square-shaped scatterer and the point-like scatterer.}\label{fig-normal+point}
\end{figure}
\end{example}

\begin{example}\textbf{Reconstructions of two disconnected normal size scatterers}\label{example5}

In this example, we consider the reconstructions of two disconnected normal size scatterers using Algorithm 3. The boundaries of the circle-shaped, kite-shaped, acorn-shaped and starfish-shaped scatterers are parameterized similar as (\ref{circle}), (\ref{kite}), (\ref{acorn}) and (\ref{starfish}) except that the sizes of the scatterers are chosen to be respectively four-ninths, half, half and two-thirds of the original sizes. The boundary of the peanut-shaped scatterer centered at $(a,b)$ is parameterized as
\begin{equation}
	\label{peanut} Peanut:s(\theta)=(a,b)+\frac{5}{12}(4\cos^{2}\theta+ \sin^{2}\theta)^{1/2}(\cos\theta,\sin\theta), \quad\theta\in[0,2\pi).
\end{equation}

Choose the noise level $\varepsilon=5\%$, the terminal time $T=25$ and $N_t=256$. The reconstructions are shown in Figure \ref{fig-2scatterers}. The exact boundaries of the scatterers are marked with green dash lines, the measurement points are marked with black asterisks.

%% Figure8 - 2 normal size scatterers
\begin{figure}
	\centering
	\begin{tabular}{ccc}
		\includegraphics[width=0.33\textwidth]{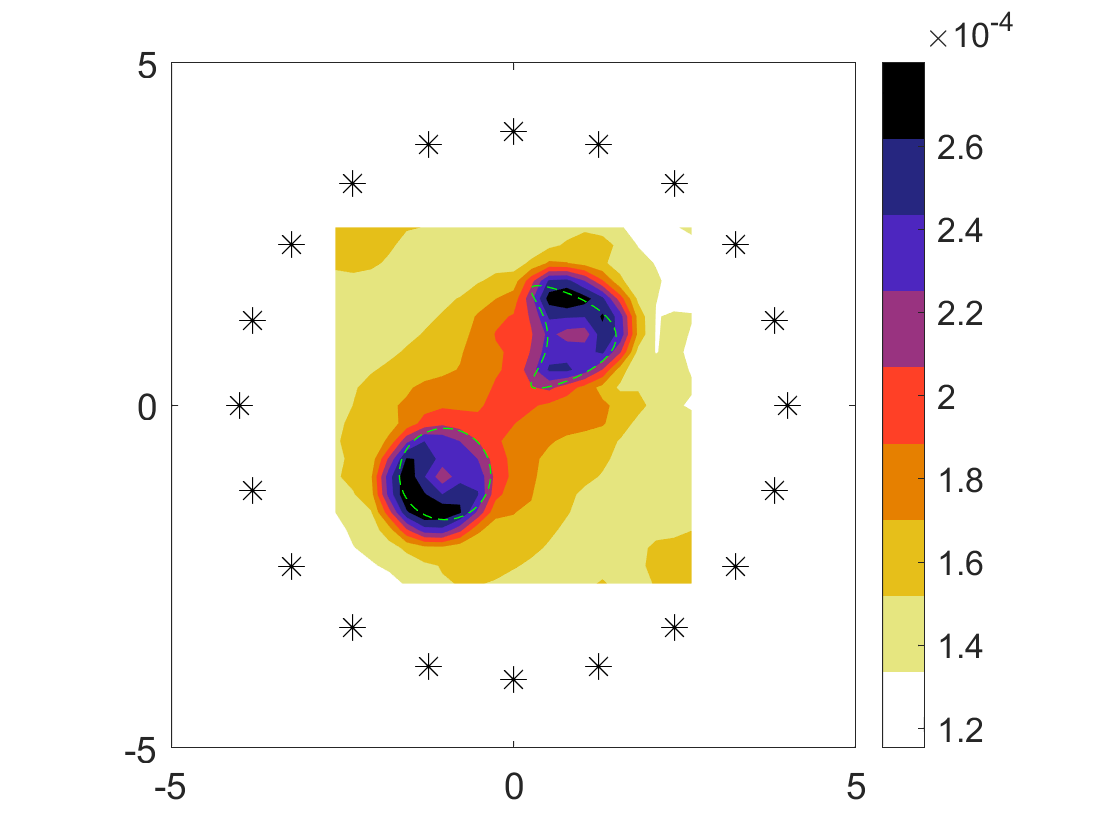}
		& \includegraphics[width=0.33\textwidth]{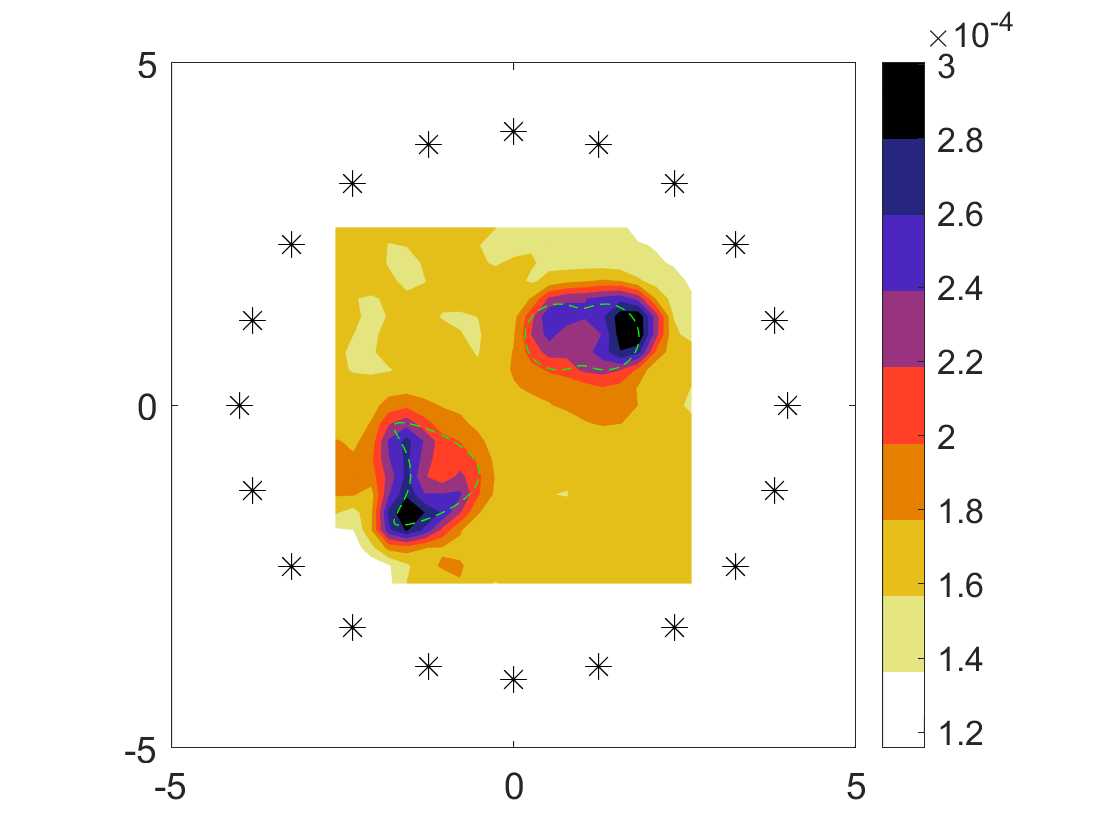}
		& \includegraphics[width=0.33\textwidth]{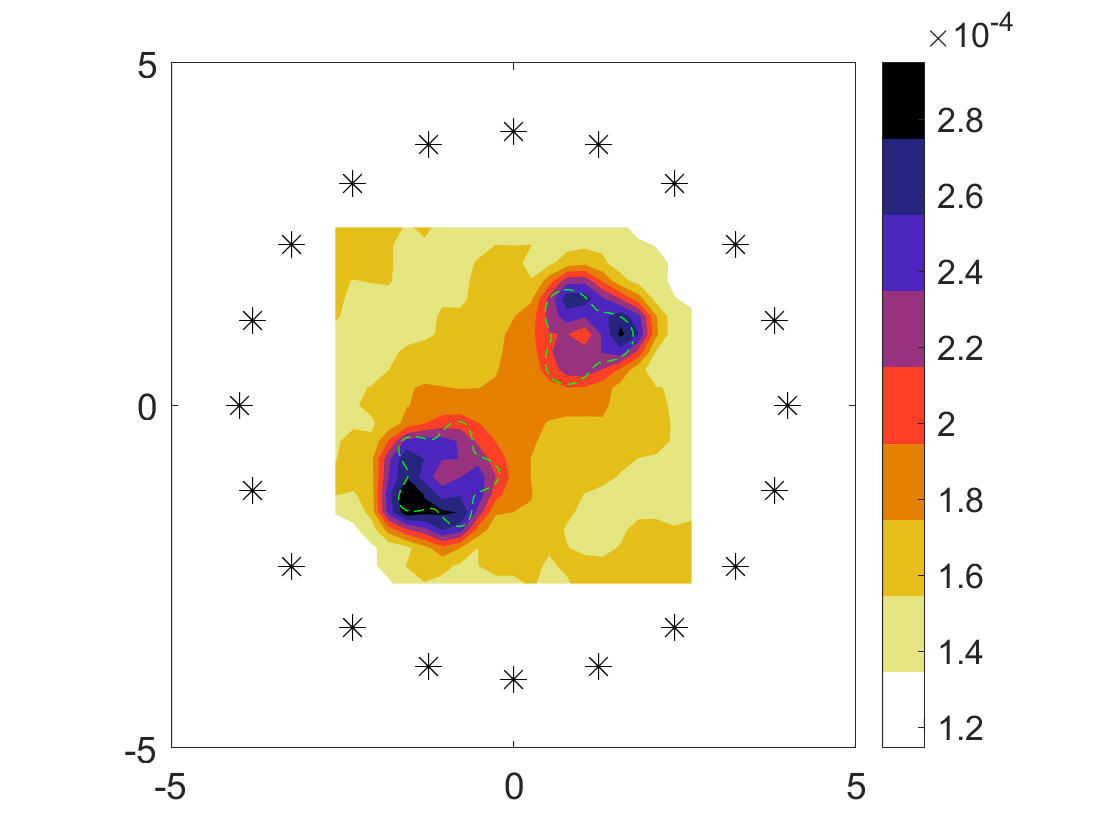}\\
		(a)~circle and kite & (b)~kite and peanut & (c)~acorn and starfish
	\end{tabular}
	\caption{Reconstructions of two disconnected normal size scatterers using the indicator function $I_3$, $\varepsilon=5\%$. (a)~The circle-shaped scatterer centered at $(-1,-1)$ and the kite-shaped scatterer centered at $(1,1)$. (b)~The kite-shaped scatterer centered at $(-1,-1)$ and peanut-shaped scatterer centered at $(1,1)$. (c)~The acorn-shaped scatterer centered at $(-1,-1)$ and the starfish-shaped scatterer centered at $(1,1)$.}\label{fig-2scatterers}
\end{figure}

In the following two examples, we consider the effectiveness of Algorithm 3 to reconstruct point-like scatterers and cubes in $\mathbb{R}^3$. Note that the computation cost of the three-dimensional time domain forward scattering problem is expensive and the k-Wave software is utilized for the calculation \cite{kWave}.
In Examples \ref{example6}-\ref{example7}, the measurement surface $\Gamma_m$ is chosen to be a spherical surface centered at the origin with radius 4, the incident surface is chosen as $\Gamma_i=\Gamma_m$. Denote by $N$ the number of the measurement points and incident points whose locations are generated by the function called "makeCartSphere" utilizing the k-Wave software. The Cartesian coordinates of the measurement points for $i=0,1,\cdots,N-1$ are
	$$x_i=4\left( \sqrt{1-\alpha_i^2} \cos(3-\sqrt{5})i\pi , \alpha_i , \sqrt{1-\alpha_i^2} \sin(3-\sqrt{5})i\pi  \right),~~\alpha_i=\dfrac{2i-N+1}{N}.$$
The sampling points are chosen as $21\times21\times21$ equidistant points in the sampling region $\Omega=[-2,2]\times[-2,2]\times[-2,2]$. The geometric diagram of the measurement points and sampling region is shown in Figure \ref{fig-50sensors}. The noise level is chosen as $\varepsilon=5\%$.
%% Figure9 - 3D - sensors
\begin{figure}
	\centering
	\includegraphics[width=0.5\textwidth]{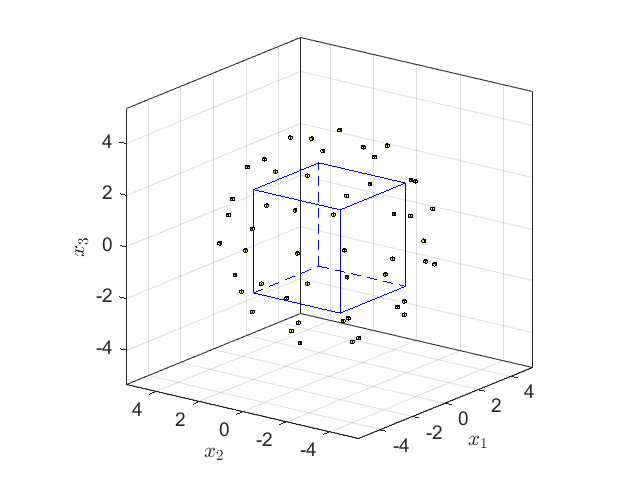}
	\caption{The geometric diagram of the measurement points and sampling region.}\label{fig-50sensors}
\end{figure}
\end{example}
	
\begin{example}\textbf{Reconstructions of point-like scatterers in $\mathbb{R}^3$}\label{example6}

In this example, the reconstructions of point-like scatterers are considered. We choose $N=50$, $T=19$ and $N_t=256$ in this example.
The exact locations of the point-like scatterers are marked with green asterisks.
In fact, the point-like scatterers is cube-shaped due to the characteristics of the k-Wave software when calculating the wave field.
For the reconstruction of a single point-like scatterer, the locations of the scatterers are chosen as $$D_1=[-0.05,0.05]\times[-0.05,0.05]\times[-0.05,0.05]$$
and
$$D_2=[0.35,0.45]\times[-0.85,-0.75]\times[0.15,0.25].$$
The location of a single point-like scatterer can be roughly identified even if the scatterer is not centered at the origin, the results can be seen in Figure \ref{point-3d}.

%% Figure10 - 3D - a single point-like scatterer
\begin{figure}
	\centering
	\begin{tabular}{ccc}
		\includegraphics[width=0.33\textwidth]{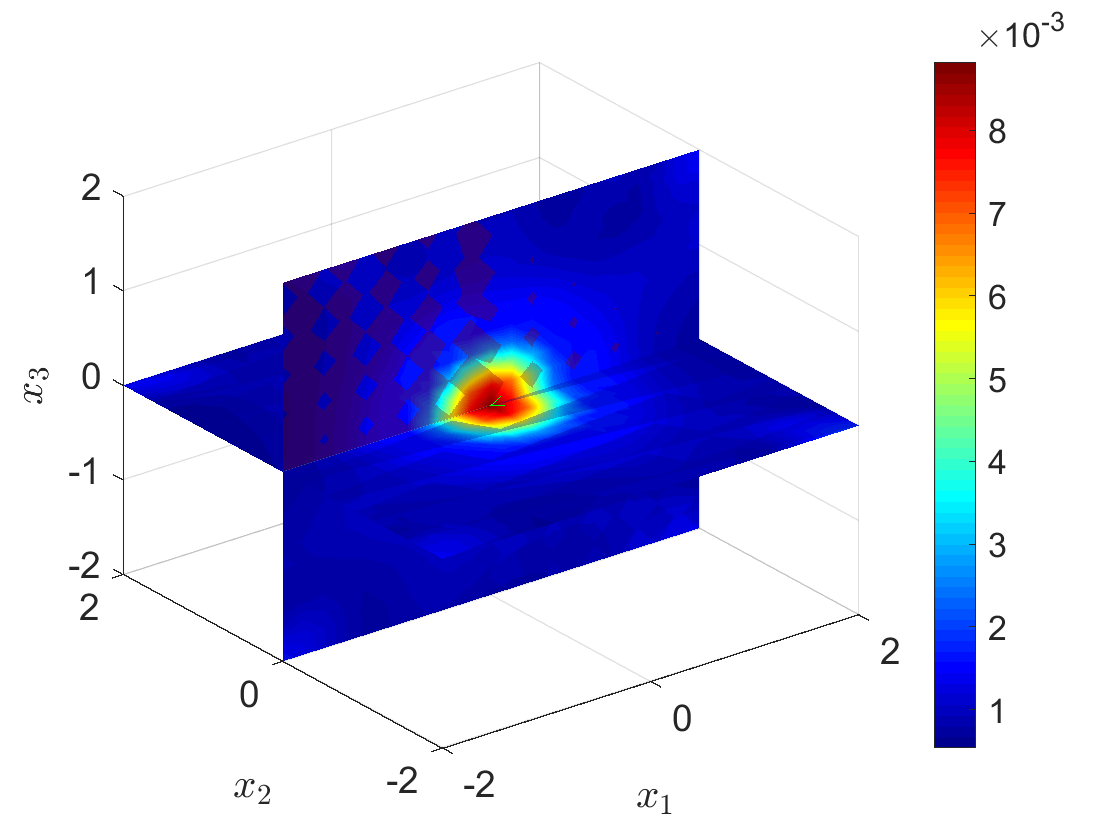}
		& \includegraphics[width=0.33\textwidth]{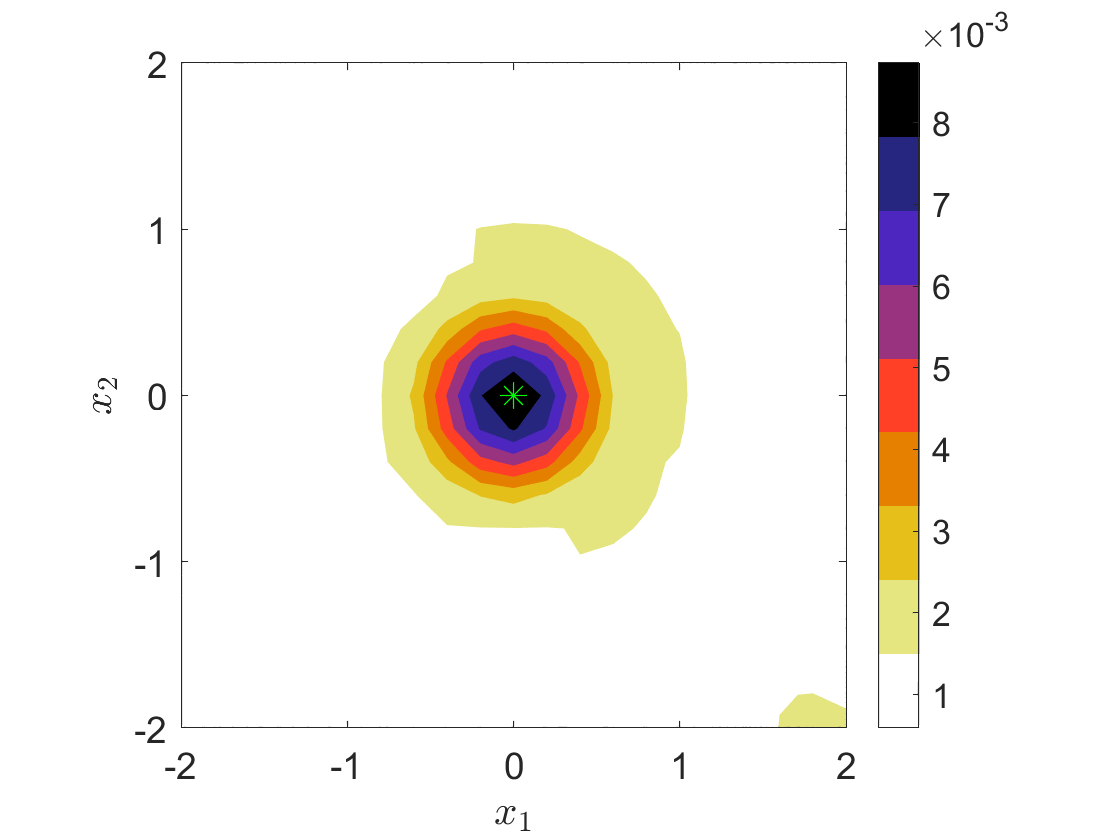}
		& \includegraphics[width=0.33\textwidth]{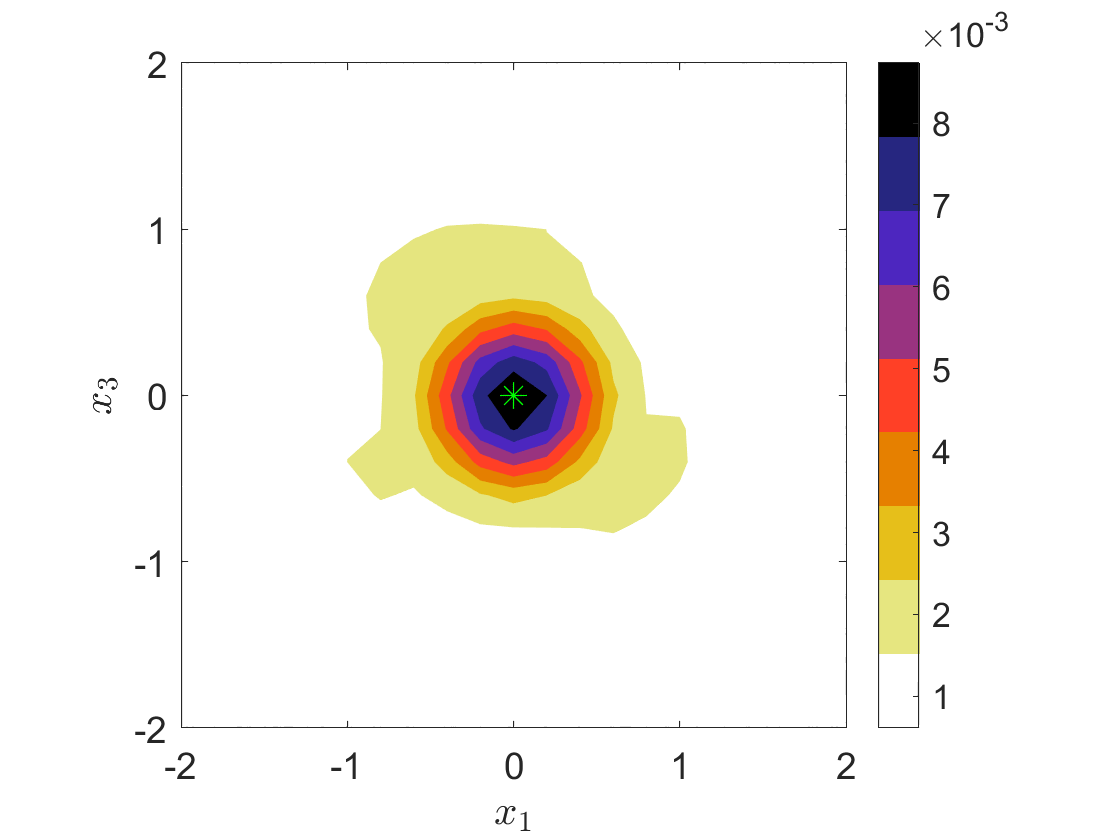} \\
		(a) & (b)  & (c)\\
		\includegraphics[width=0.33\textwidth]{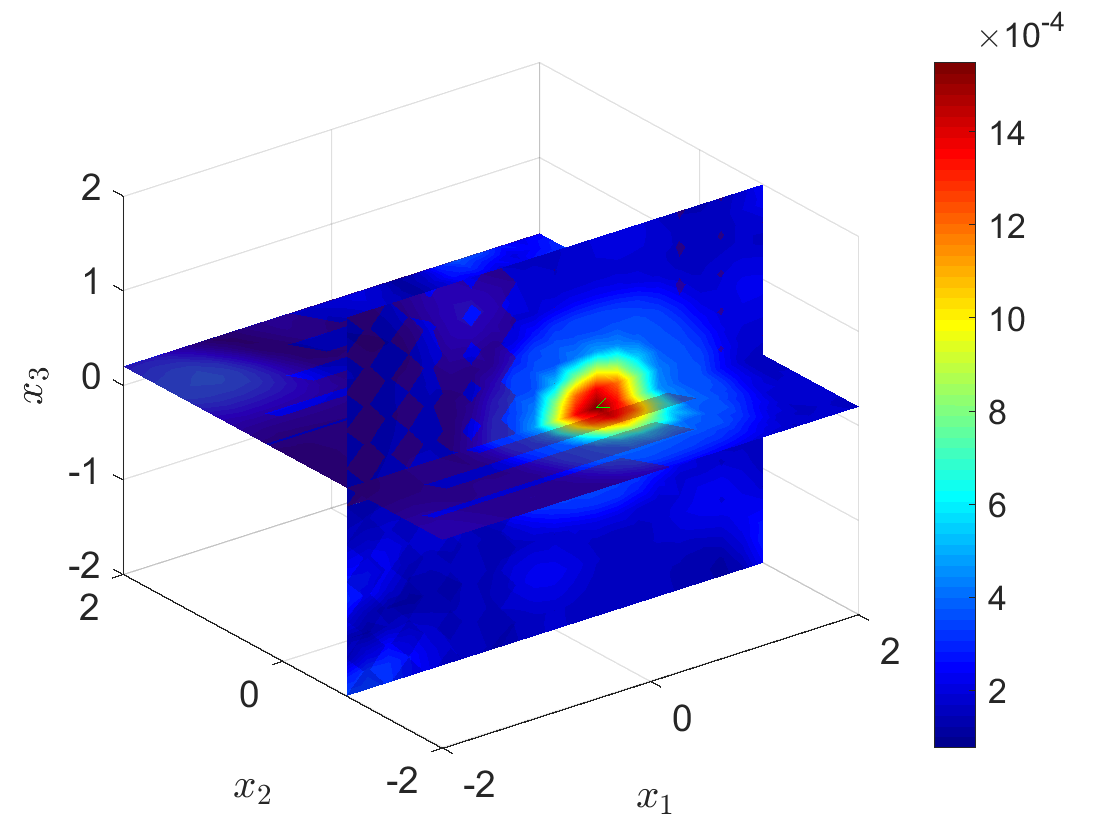}
		& \includegraphics[width=0.33\textwidth]{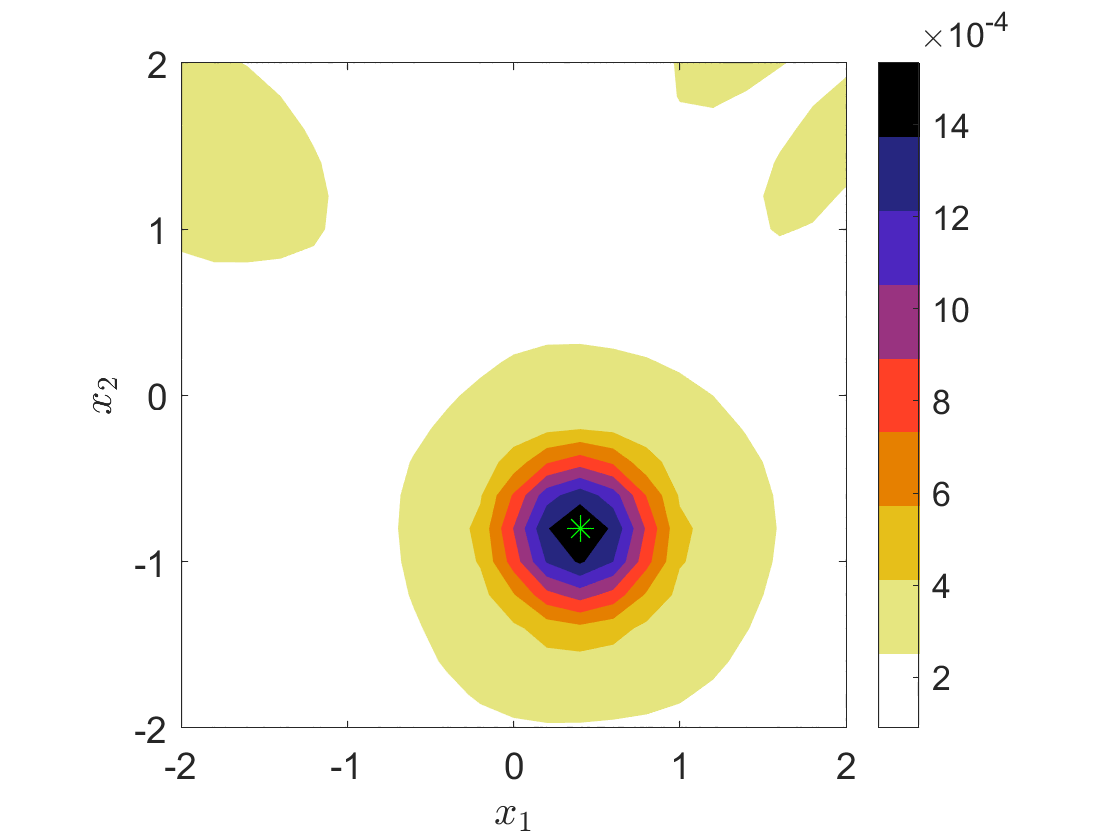}
		& \includegraphics[width=0.33\textwidth]{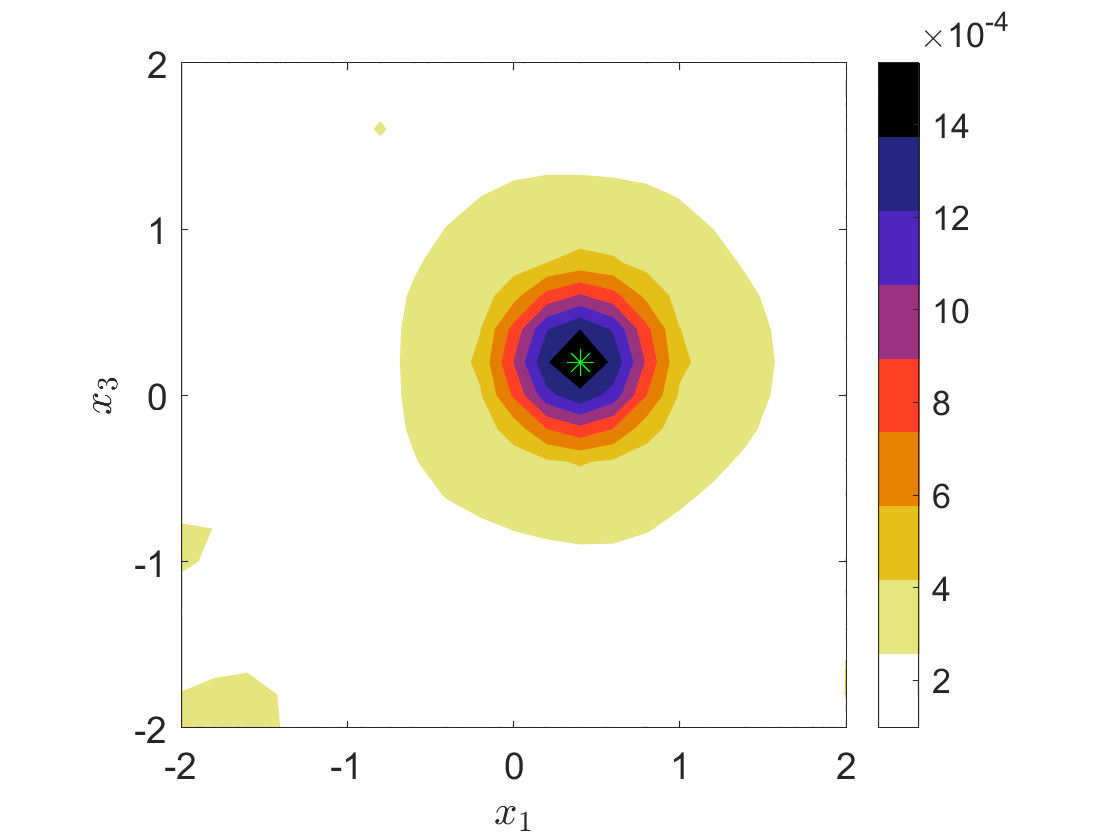} \\
		(d) & (e)  & (f)
	\end{tabular}
	\caption{Reconstructions of a single point-like scatterer using the indicator function $I_3$, $\varepsilon=5\%$. (a,d)~The reconstructions of the scatterers $D_1$ and $D_2$, respectively. (b,e)~The slices through the center of the scatterers along the $x_{1}$-$x_{2}$ plane. (c,f)~The slices through the center of the scatterers along the $x_{1}$-$x_{3}$ plane.}\label{point-3d}
\end{figure}

For the reconstructions of multiple point-like scatterers in $\mathbb{R}^3$,
the locations of the scatterers are chosen to be
$$D_3=[0.55,0.65]\times[0.75,0.85]\times[0.95,1.05]$$
and
$$D_4=[-1.05,-0.95]\times[-0.85,-0.75]\times[-0.65,-0.55],$$ respectively.
The reconstructions can be seen in Figure \ref{2points-3d}.

% Figure11 - 3D - two point-like scatterers
\begin{figure}
	\centering
	\begin{tabular}{ccc}
		\includegraphics[width=0.33\textwidth]{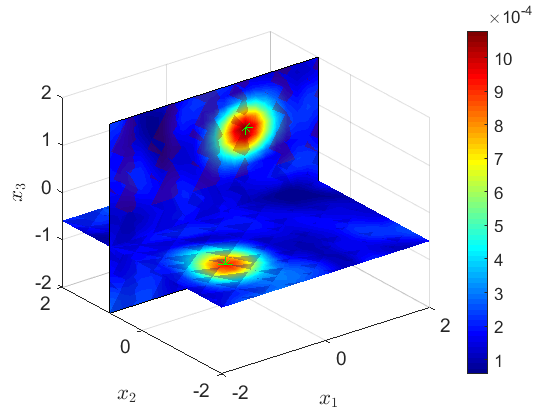}
		& \includegraphics[width=0.33\textwidth]{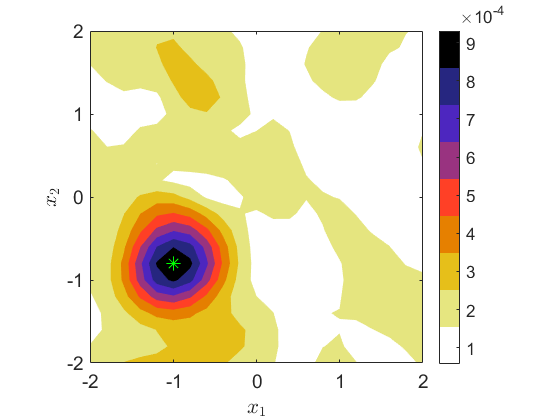}
		& \includegraphics[width=0.33\textwidth]{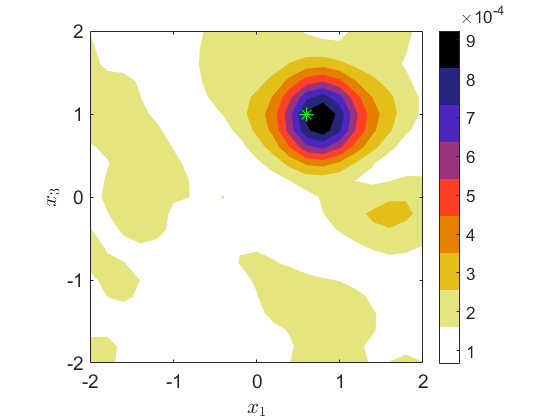}\\
		(a) & (b)& (c)
	\end{tabular}
	\caption{Simultaneous reconstruction of two point-like scatterers $D_3$ and $D_4$ using the indicator function $I_3$, $\varepsilon=5\%$. (a)~The reconstruction of the scatterers. (b)~The slice through the center of the scatterer along the $x_{1}$-$x_{2}$ plane. (c)~The slice through the center of the scatterer along the $x_{1}$-$x_{3}$ plane.}\label{2points-3d}
\end{figure}
\end{example}

\begin{example}\textbf{Reconstructions of cube scatterers}\label{example7}

The reconstructions of cube scatterers are considered in this example. We choose $N=50$, $T=25$ and $N_t=256$. Again, the forward scattering problem is solved utilizing the k-Wave software. The measurement points, the source points, the sampling grids and other parameters are chosen the same as that in Example \ref{example6}. See Figure \ref{cube-3d} and Figure \ref{cubes-3d} for the reconstructions.
In Figure \ref{cube-3d}, the cube is chosen as $$D_5=[-0.4,0.8]\times[-0.4,0.8]\times[-0.4,0.8].$$
In Figure \ref{cubes-3d}, the two disconnected cubes are chosen as $$D_6=[0.4,1]\times[0.4,1]\times[0.4,1]$$
and
$$D_7=[-1,-0.4]\times[-1,-0.4]\times[-1,-0.4],$$ respectively.
The noise level is chosen as $\varepsilon=5\%$. This example shows that Algorithm 3 could give the approximate positions of the cubes, but is no effective to determine the shape of the cubes.

% Figure12 - 3D - a cube
\begin{figure}
	\centering
	\begin{tabular}{ccc}
		\includegraphics[width=0.33\textwidth]{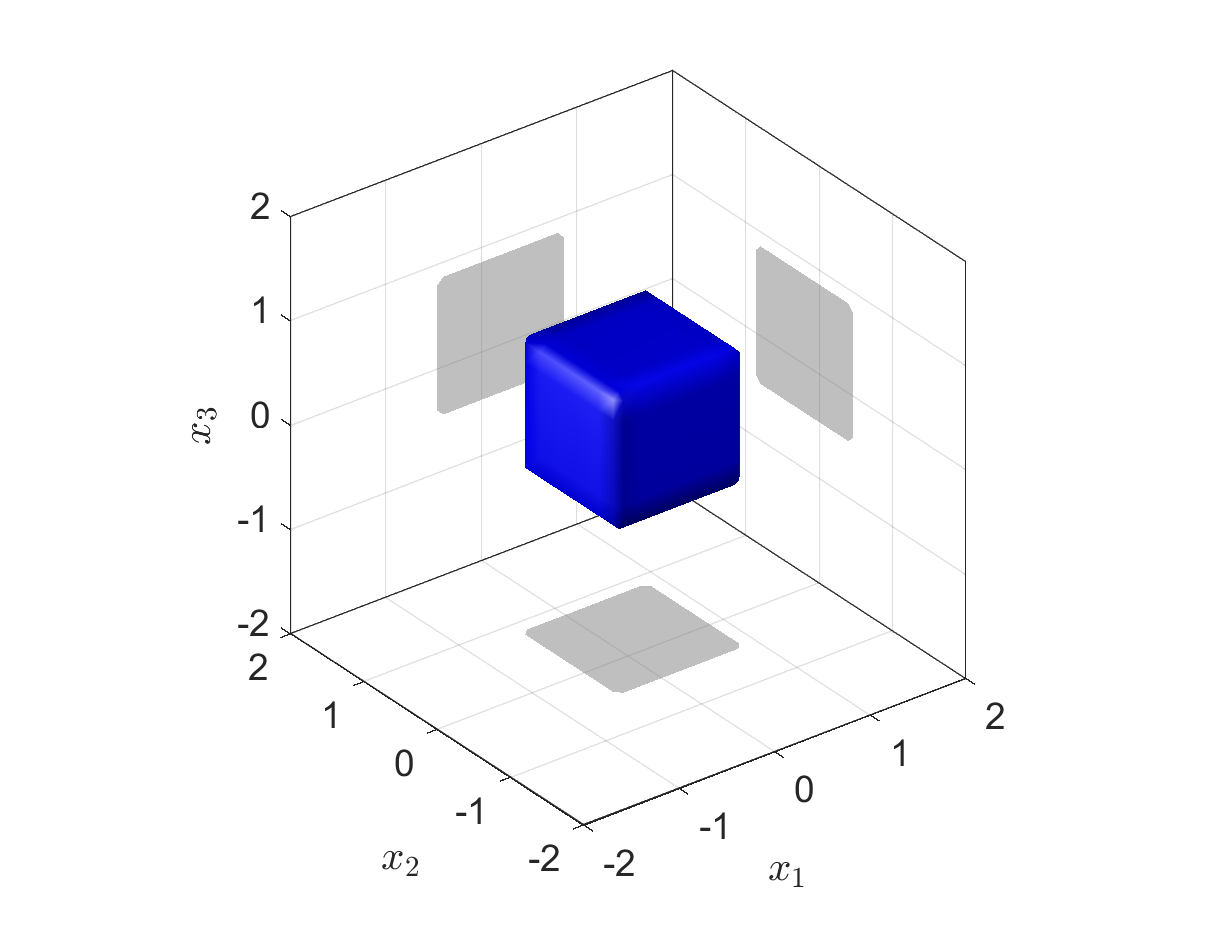}
		& \includegraphics[width=0.33\textwidth]{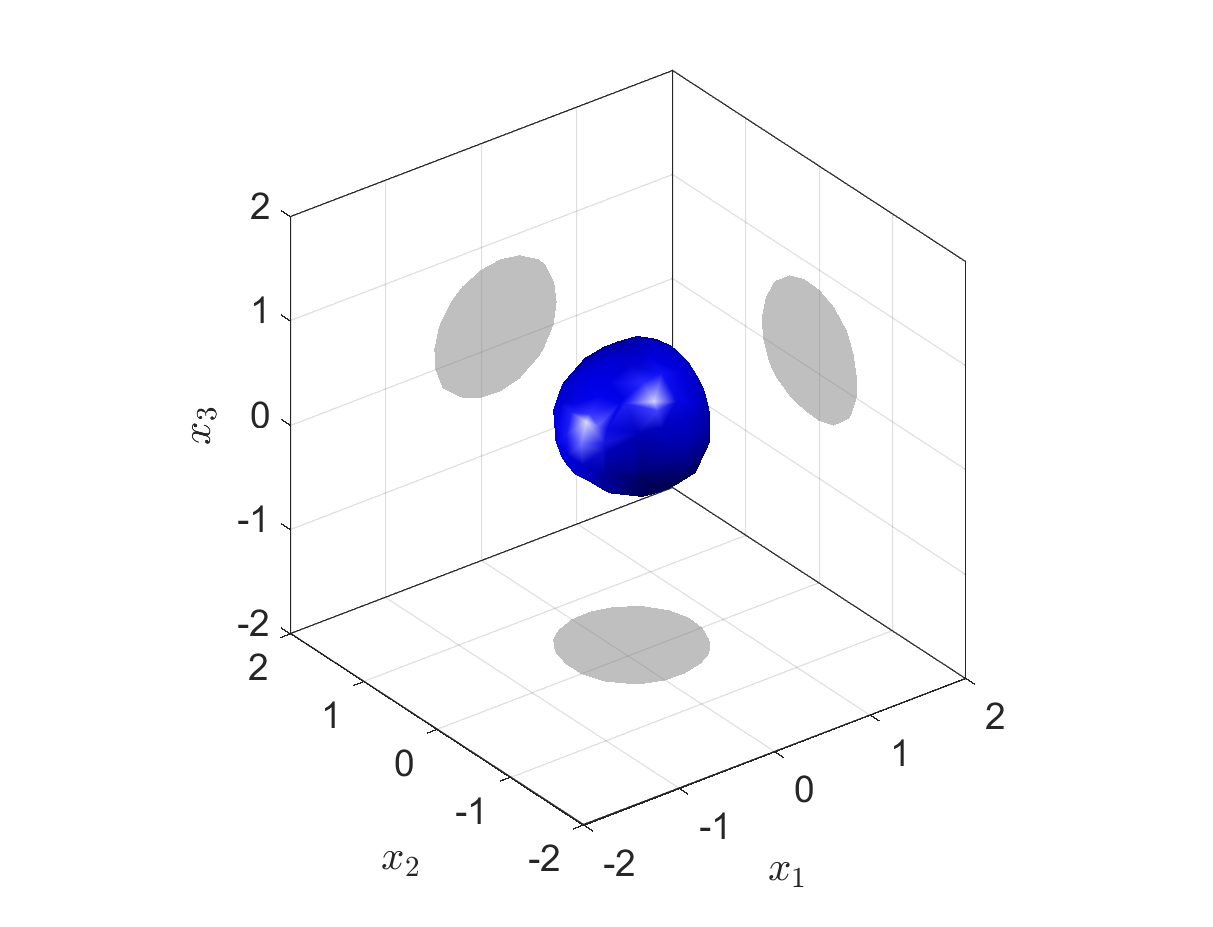}
		& \includegraphics[width=0.33\textwidth]{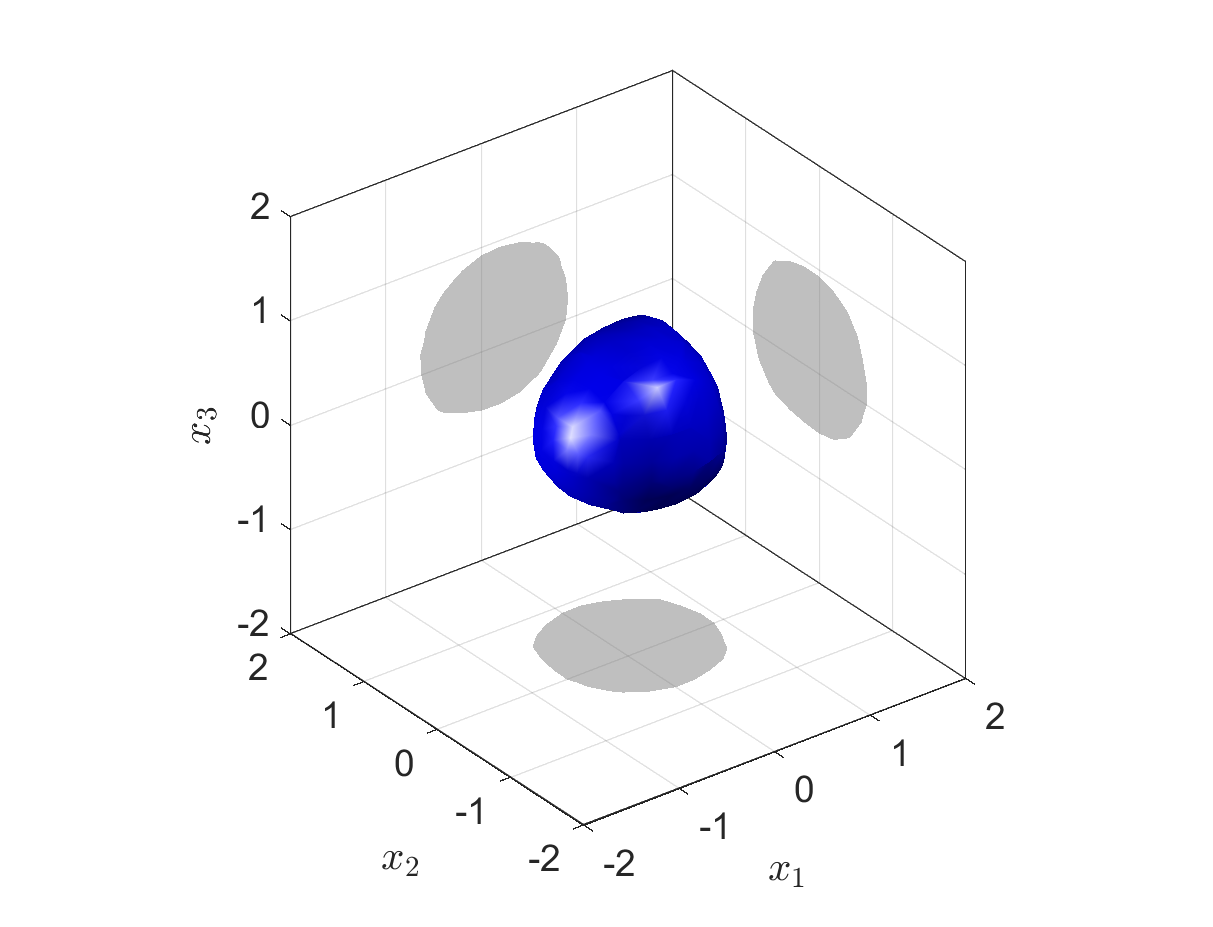} \\
		(a) & (b) & (c)
	\end{tabular}
	\caption{Reconstructions of the cube $D_5$ using the indicator function $I_3$, $\varepsilon=5\%$. (a)~The geometry setting. (b)~The reconstruction of the cube with iso-surface level $0.02$. (c)~The reconstruction of the cube with iso-surface level $0.015$.}\label{cube-3d}
\end{figure}

% Figure13 - 3D - 2 cubes
\begin{figure}
	\centering
	\begin{tabular}{ccc}
		\includegraphics[width=0.33\textwidth]{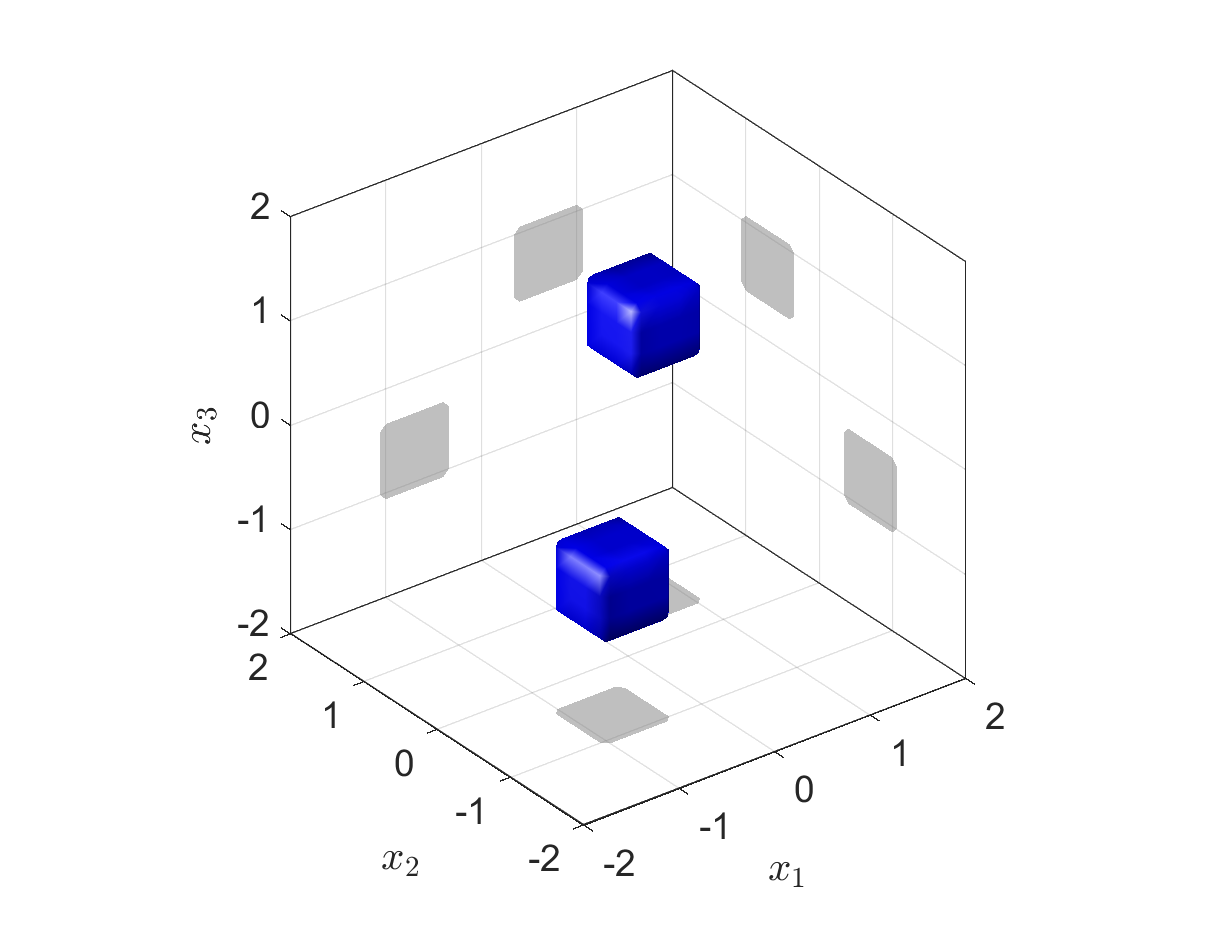}
		& \includegraphics[width=0.33\textwidth]{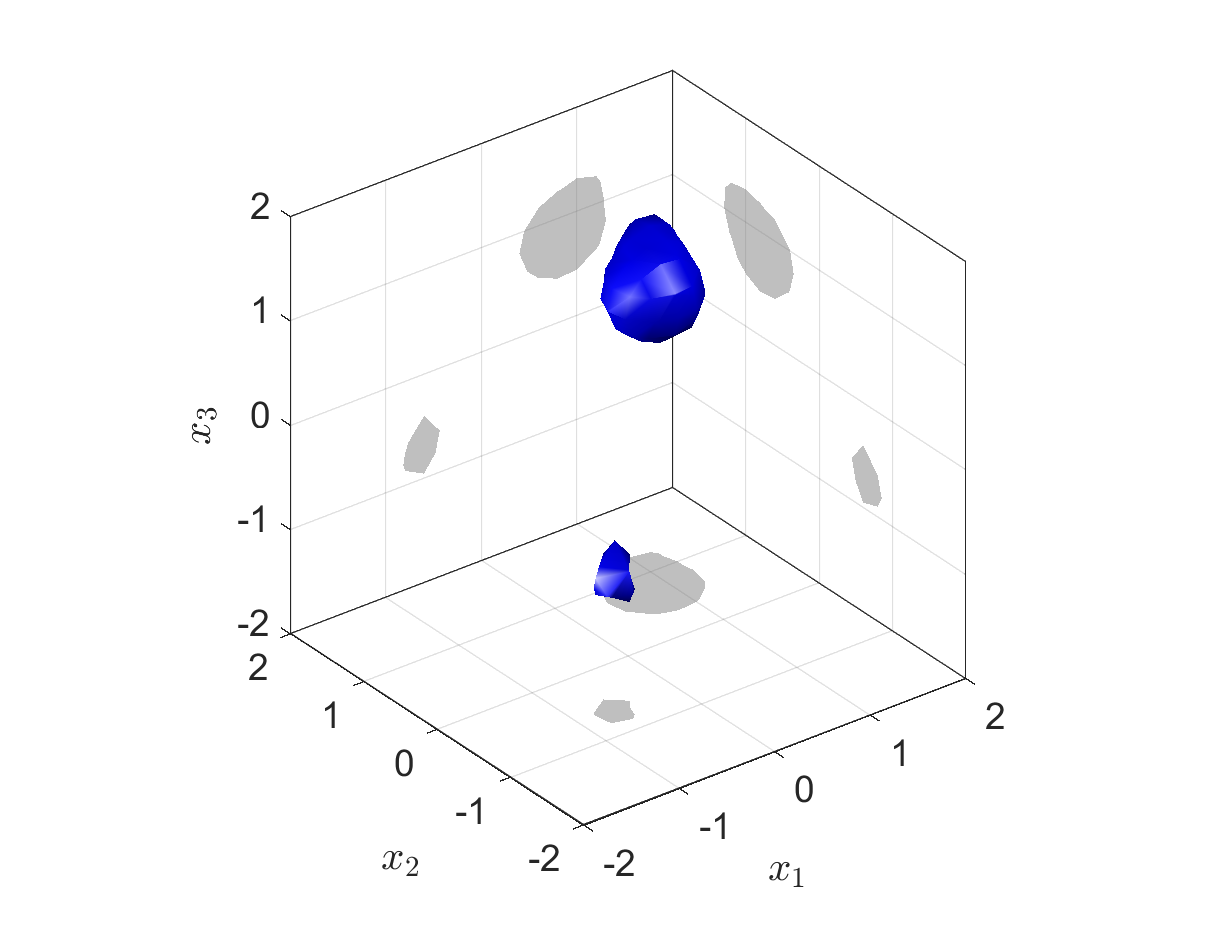}
		& \includegraphics[width=0.33\textwidth]{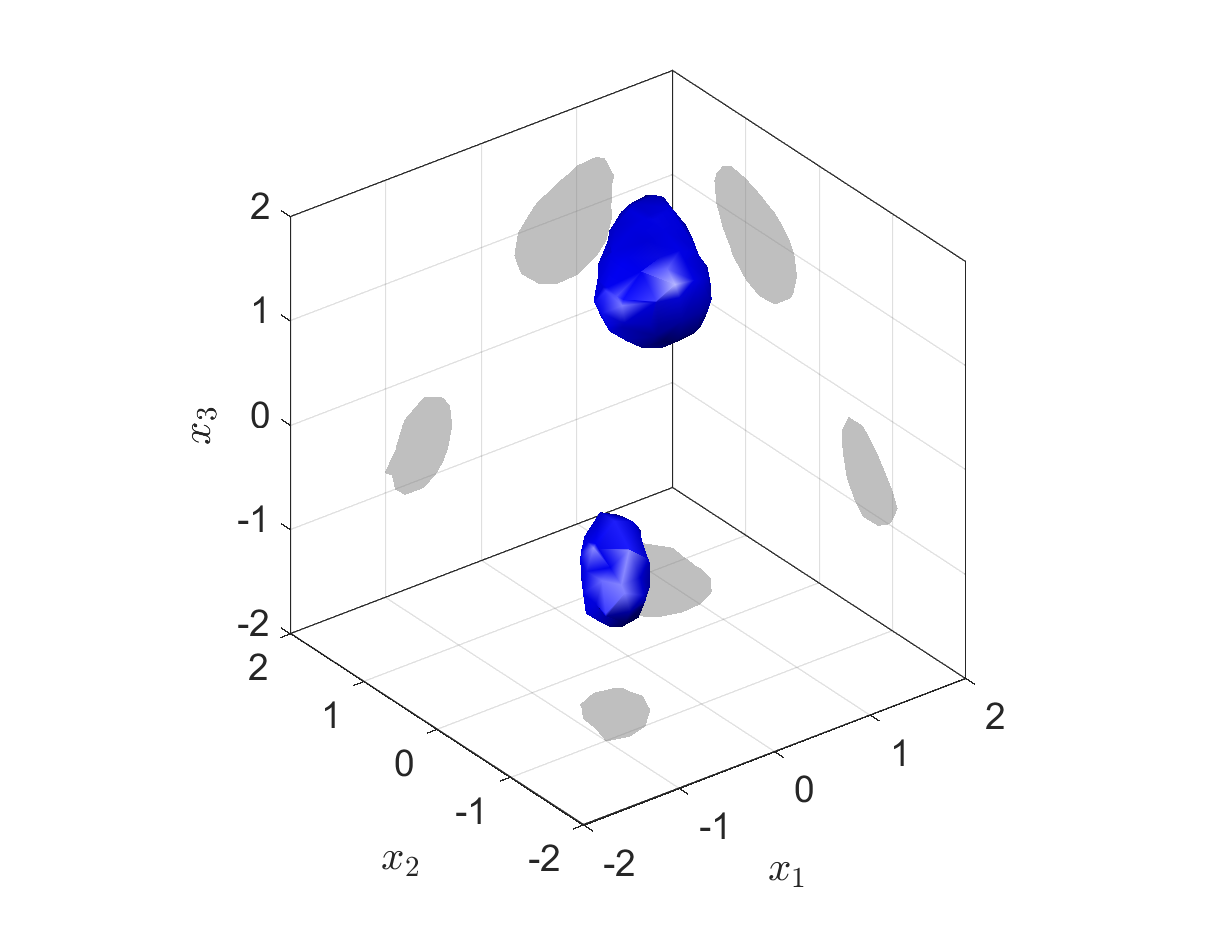} \\
		(a) & (b) & (c)
	\end{tabular}
	\caption{Simultaneous reconstructions of two disconnected cubes $D_6$ and $D_7$ using the indicator function $I_3$, $\varepsilon=5\%$. (a)~The geometry setting. (b)~The reconstruction of the cubes with iso-surface level $0.008$. (c)~The reconstruction of the cubes with iso-surface level $0.007$.}\label{cubes-3d}
\end{figure}
\end{example}

\section{Conclusion}
Simple direct sampling methods based on the time convolution of the Green's function and the signal function have been provided to reconstruct point-like and normal size scatterers in this paper. Numerical examples in both two and three dimensional spaces have been provided to demonstrate the effectiveness and robustness of the methods.

\section*{Acknowledgments}
The work of Bo Chen was supported by the National Natural Science Foundation of China (grand number: 12101603), the Fundamental Research Funds for the Central Universities (Special Project for Civil Aviation University of China, grand number: 3122021072) and the supplementary project of Civil Aviation University of China (grand number: 3122022PT19).

%\bibliographystyle{unsrt}
%\bibliography{reference}

\begin{thebibliography}{99}

\bibitem{Springer-2}D.~Colton and R.~Kress, {\em Inverse Acoustic and Electromagnetic Scattering Theory}, 3rd ed. Springer-Berlin (2013).

\bibitem{CCP}M.~Zhao, J.~He and L.~Wang, {\em Numerical Solutions of the Electromagnetic Scattering by Overfilled Cavities with Inhomogeneous Anisotropic Media}, Commun. Comput. Phys. \textbf{34},530--562 (2023).

\bibitem{EABE}Y.~Sun, X.~Lu and B.~Chen, {\em The method of fundamental solutions for the high frequency acoustic-elastic problem and its relationship to a pure acoustic problem}, Eng. Anal. Bound. Elem. \textbf{156}(11),299--310 (2023).

\bibitem{JCP-3}G.~Bao, S.~Hou and P.~Li, {\em Inverse scattering by a continuation method with initial guesses from a direct imaging algorithm}, J. Comput. Phys. \textbf{227}(1),755--762 (2007).

\bibitem{SIAM-4}G.~Bao and J.~Lin, {\em Imaging of local surface displacement on an infinite ground plane: the multiple frequency case}, SIAM J. Appl. Math. \textbf{71},1733--1752 (2011).

\bibitem{IP-5}H.~Liu and X.~Liu, {\em Recovery of an embedded obstacle and its surrounding medium from formally determined scattering data}, Inverse Probl. \textbf{33}(6),065001 (2017).

\bibitem{Springer-1}F.~J.~Sayas, {\em Retarded Potentials and Time Domain Boundary Integral Equations:a Road Map}, Springer-Switzerland (2016).

\bibitem{IP-1}D.~Colton and A.~Kirsch, {\em A simple method for solving inverse scattering problems in the resonance region}, Inverse probl. \textbf{12}(4),383 (1996).

\bibitem{IP-2}Q.~Chen, H.~Haddar, A.~Lechleiter and P.~Monk, {\em A sampling method for inverse scattering in the time domain}, Inverse Probl. \textbf{26}(8),085001 (2010).

\bibitem{SIAM-3}J.~Li, H.~Liu and J.~Zou, {\em Strengthened linear sampling method with a reference ball}, SIAM J. Sci. Comput. \textbf{31}(6),4013--4040 (2009).

\bibitem{IPI-2}J.~Li, H.~Liu, H.~Sun and J.~Zou, {\em Imaging acoustic obstacles by singular and hypersingular point sources}, Inverse Probl. imaging. \textbf{7}(2),545--563 (2013).

\bibitem{NMTMA1}Z.~Chen and G.~Huang, {\em Phaseless imaging by reverse time migration: acoustic waves}, Numer. Math. Theory Methods Appl. \textbf{10}(1),1--21 (2017).

\bibitem{IP-4}K.~H.~Leem, J.~Liu and G.~Pelekanos, {\em Two direct factorization methods for inverse scattering problems}, Inverse Probl. \textbf{34}(12),125004 (2018).

\bibitem{SIAM-2}J.~Li, H.~Liu and J.~Zou, {\em Multilevel linear sampling method for inverse scattering problems}, SIAM J. Sci. Comput. \textbf{30}(3),1228--1250 (2008).

\bibitem{JCP-1}J.~Li, H.~Liu and Q.~Wang. {\em Enhanced multilevel linear sampling methods for inverse scattering problems}, J. Comput. Phys. \textbf{257},554--571 (2014).

\bibitem{IP-6}X.~Liu, {\em A novel sampling method for multiple multiscale targets from scattering amplitudes at a fixed frequency}, Inverse Probl. \textbf{33}(8),085011 (2017).

\bibitem{ERA}Y.~Chang and Y.~Guo, {\em Simultaneous recovery of an obstacle and its excitation sources from near-field scattering data}, Electron. Res. Arch. \textbf{30}(4),1296--1321 (2022).

\bibitem{IP-8}D.~Zhang, Y.~Guo, Y.~Wang and Y.~Chang, {\em Co-inversion of a scattering cavity and its internal sources: uniqueness, decoupling and imaging}, Inverse Probl. \textbf{39}(6),065004 (2023).

\bibitem{AA-1}B.~Chen, F.~Ma and Y.~Guo, {\em Time domain scattering and inverse scattering problems in a locally perturbed half-plane}, Appl. Anal. \textbf{96}(8),1303--1325 (2017).

\bibitem{EAJAM}Y.~Yue, F.~Ma and B.~Chen, {\em Time domain linear sampling method for inverse scattering problems with cracks}, East Asian J. Appl. Math. \textbf{12}(1),96--110 (2022).

\bibitem{AA-2}H.~Haddar, A.~Lechleiter and S.~Marmorat. {\em An improved time domain linear sampling method for Robin and Neumann obstacles}, Appl. Anal. \textbf{93}(2),369--390 (2014).

\bibitem{IP-3}Y.~Guo, P.~Monk and D.~Colton, {\em Toward a time domain approach to the linear sampling method}, Inverse Probl. \textbf{29}(9),095016 (2013).

\bibitem{JCP-2}Y.~Guo, D.~Hömberg, G.~Hu, J.~Li and H.~Liu, {\em A time domain sampling method for inverse acoustic scattering problems}, J. Comput. Phys. \textbf{314},647--660 (2016).

\bibitem{IPI-2}J.~Lai, M.~Li, P.~Li and W.~Li, {\em A fast direct imaging method for the inverse obstacle scattering problem with nonlinear point scatterers}, Inverse Probl. imaging. \textbf{12}(3),635--665 (2018).

\bibitem{IP-7}B.~Chen, Y.~Guo, F.~Ma and Y.~Sun, {\em Numerical schemes to reconstruct three-dimensional time-dependent point sources of acoustic waves}, Inverse Probl. \textbf{36}(7),075009 (2020).

\bibitem{JSC}Y.~Sun, {\em Indirect boundary integral equation method for the Cauchy problem of the Laplace equation}, J. Sci. Comput. \textbf{71}(2),469--498 (2017).

\bibitem{NMTMA2}H.~Dong, W.~Ying and J.~Zhang, {\em An efficient cartesian grid-based method for scattering problems with inhomogeneous media}, Numer. Math. Theor. Meth. Appl. \textbf{16},541--564 (2023).

\bibitem{AAM}G.~Bao and L.~Zhang, {\em The acoustic scattering of a layered elastic shell medium}, Ann. Appl. Math. \textbf{39},462--492 (2023).

\bibitem{IPSE-1}B.~Chen and Y.~Sun, {\em A simple method of reconstructing a point-like scatterer according to time-dependent acoustic wave propagation}, Inverse Probl. Sci. En. \textbf{29}(12),1895--1911 (2021).

\bibitem{SIAM-5}L.~Banjai and S.~Sauter, {\em Rapid solution of the wave equation in unbounded domains}, SIAM J. Numer. Anal. \textbf{47}(1),227--249 (2008).

\bibitem{kWave}B.~E.~Treeby and B.~T.~Cox, {\em K-wave: Matlab toolbox for the simulation and reconstruction of photoacoustic wave-fields}, J. Biomed. Opt. \textbf{15},021314 (2010).






%\bibitem{Berger} M.J.~Berger and P.~Collela, {\em Local adaptive mesh
%refinement for shock hydrodynamics}, J. Comput. Phys. \textbf{82}, 62--84 (1989).
%\bibitem{firstauthor}  F.~Author and A.~Co-Author, \emph{Preparation of manuscript}, Intern. Public. \textbf{1}, 12--21 (2018).
%
%
%
%\bibitem{deBoor} C.~de~Boor, {\em Good approximation by splines with variable knots II}, Springer Lecture  Notes Series \textbf{363}, Springer-Verlag (1973).
%\bibitem{Can09}
%C.~Canuto, {\em High-order methods for {PDE}s: recent advances and new perspectives},
% in:  {\em 6th {I}nternational {C}ongress on {I}ndustrial  and {A}pplied {M}athematics}, pp. 57--87,
% European Mathematical Society (2009).
%
% \bibitem{coutsias1996}
%E.~Coutsias, T.~Hagstrom, J.S. Hesthaven and D.~Torres,
% \emph{Integration preconditioners for differential operators in spectral
%  $\tau$-methods},
% in: {\em Proceedings of the Third International Conference on Spectral
%  and High Order Methods}, A.~Ilin and R.~Scott (Eds), pp. 21--38,  Houston Journal of Mathematics (1996).
%
%
%\bibitem{TanTZ}  Z.J.~Tan, T.~Tang and Z.R.~Zhang, {\em A simple moving mesh method for one- and two-dimensional phase-field equations}, J. Comput. Appl. Math. (To appear).
%\bibitem{Toro} E.F.~Toro,  \emph{Riemann Solvers and Numerical Methods for Fluid Dynamics}, Springer-Verlag (1999).
\end{thebibliography}

\end{document}